\documentclass[a4paper, twoside, 11pt]{amsart}
\usepackage[margin=1 in]{geometry}                
\usepackage{mathtools}
\mathtoolsset{showonlyrefs,showmanualtags}
\usepackage[colorlinks,linkcolor=blue,citecolor=blue,urlcolor=blue]{hyperref}
\usepackage{graphicx}
\usepackage{amssymb, amsmath}
\usepackage{soul}
\usepackage{mathrsfs}
\usepackage{enumitem}
\usepackage[usenames,dvipsnames]{xcolor}
\usepackage{bbm}
\usepackage[T1]{fontenc}
\usepackage[utf8]{inputenc}
\usepackage[russian, english]{babel}

\usepackage{tikz-cd}
\usetikzlibrary{shapes.misc}
\usetikzlibrary{decorations.markings}
\usepackage{yfonts}
\usepackage{etoolbox}
\usepackage{stackengine}[2013-10-15]
\usepackage{mdframed}
\usepackage{bm}
\usepackage[square,numbers]{natbib}
\usepackage{mathdots}
\usepackage{arydshln}
\usetikzlibrary{3d,calc}
\usepackage[normalem]{ulem}
\usepackage[toc,page]{appendix}

\usepackage{soul}

\numberwithin{equation}{section}

\newcommand\bovermat[2]{%
    \makebox[0pt][l]{$\smash{\overbrace{\phantom{%
                     #2}}^{\text{#1}}}$}#2}

\newcommand{\rdots}[2]{
 \raisebox{#2}{\rotatebox{#1}{$\ddots$}}}

\newcommand{\Z}{\mathbb Z}

\newcommand{\R}{\mathbb R}

\newcommand{\Hom}{\text{Hom}}
\newcommand{\Aut}{\mbox{Aut}}
\newcommand{\ad}{\text{\textnormal{ad}}}

\makeatletter
\newcommand*\bigcdot{\mathpalette\bigcdot@{.7}}
\newcommand*\bigcdot@[2]{\mathbin{\vcenter{\hbox{\scalebox{#2}{$\m@th#1\bullet$}}}}}
\makeatother

\newtheorem{theorem}{Theorem}[section]
\newtheorem{lemma}[theorem]{Lemma}

\newtheorem{definition}[theorem]{Definition}

\newtheorem{corollary}[theorem]{Corollary}
\newtheorem{remark}[theorem]{Remark}

\usepackage[textwidth=1.8cm]{todonotes}  
\usepackage{marginnote}
\usepackage{zref-savepos}

\newcount\todocount

\newcommand{\checkxpos}[3][]{%
  \ifdim \zposx{#2}sp < 20000000sp%
    \mynote[#1]{#3}%
  \else%
    \note[#1]{#3}%
  \fi%
}

\newcommand{\mytodo}[2][]{%
  \zsaveposx{todo\the\todocount}%
  \checkxpos[#1]{todo\the\todocount}{#2}%
  \global\advance\todocount1\relax
}

\newcommand{\mynote}[2][]{{%
  \let\marginpar\marginnote
  \reversemarginpar
  \renewcommand{\baselinestretch}{0.8}%
  \todo[#1]{#2}}}

\newcommand{\note}[2][]{\renewcommand{\baselinestretch}{0.8}\todo[#1]{#2}}

\begin{document}
\title[Maximal   dimension of symmetries of homogeneous  2-nondegenerate CR structures] {Maximal  dimension of groups of symmetries  of homogeneous  2-nondegenerate CR structures of hypersurface type with a 1-dimensional Levi kernel}

\thanks{I.\ Zelenko is supported by Simons Foundation Collaboration Grant for Mathematicians 524213.}

\author{David Sykes}
\address{David Sykes,
	Department of Mathematics
	Texas A\&M University
	College Station
	Texas, 77843
	USA}\email{ dgsykes@tamu.edu}
\urladdr{\url{http://www.math.tamu.edu/~dgsykes}}

\author{Igor Zelenko}
\address{Igor Zelenko, Department of Mathematics
	Texas A\&M University
	College Station
	Texas, 77843
	USA}\email{ zelenko@math.tamu.edu}
\urladdr{\url{http://www.math.tamu.edu/~zelenko}}


 \begin{abstract}
We prove that for every $n\geq 3$ the sharp upper bound  for the dimension of the symmetry groups of homogeneous, 2-nondegenerate, $(2n+1)$-dimensional CR manifolds of hypersurface type with a $1$-dimensional Levi kernel is equal to $n^2+7$, {and simultaneously establish the same result for a more general class of structures characterized by weakening the homogeneity condition.}
 This supports  Beloshapka's conjecture stating that hypersurface models with a maximal finite dimensional group of symmetries for a given dimension of the underlying manifold are Levi nondegenerate.

\end{abstract}

\maketitle

\section{Introduction}\label{introduction}

A classical problem setting in differential geometry is to find  homogeneous structures  with the symmetry group of maximal dimension among all geometric structure of a certain class. Homogeneity here means, as usual, that the symmetry group of the structure acts transitively. In Cauchy-Riemann (CR) geometry this problem is classically solved for the class of Levi nondegenerate CR structures of hypersurface type of arbitrary dimension (\cite{tanakaCR, chernmoserCR}).  The present paper solves this problem for  $2$-nondegenerate CR structures of hypersurface type  with a $1$-dimensional Levi kernel. This class can be seen as the next one in a hierarchy of nondegeneracies to the class of Levi nondegenerate CR structures of hypersurface type.  {We furthermore obtain this result for structures that are not necessarily homogeneous, but that rather satisfy a weaker condition we term \emph{admitting a constant reduced modified CR symbol (Definition \ref{reduced modified symbol def} below).}} Previously the answer to this problem was given only  in the $5$-dimensional case \cite{isaev, medori, pocchiola}, which is the case of the smallest possible dimension in which $2$-nondegenrate structures exist.  We give the answer for arbitrary dimension (which a priori is odd) greater than $5$ extending the previous result of \cite{porter2017absolute} that  worked under additional restrictions of regularity of the CR symbol. The definition of the CR symbol an its regularity was introduced in \cite{porter2017absolute} and is discussed in Section \ref{CRsec} below. This result supports  Beloshapka's conjecture \cite[Conjecture 5.6]{isaev} stating that the hypersurface models with maximal finite dimensional groups of symmetries for a given dimension of the underlying manifold are Levi nondegenerate.

In more detail,  let $M$ be a $(2n+1)$-dimensional CR manifold with CR structure $H$ of hypersurface type, meaning that $H$ is an integrable,  totally real, complex rank $n$ distribution contained in the complexified tangent bundle $\mathbb C TM$ of $M$, that is,
\begin{align}\label{CR structure axioms}
[H,H]\subset H\quad\mbox{ and } \quad H\cap\overline{H}=0
\end{align}
where the overline in $\overline{H}$ denotes the natural complex conjugation in $\mathbb C TM$. 

Recall that the  \emph{Levi form} of the structure $H$ is a field over $M$ of Hermitian forms defined on fibers of $H$ by the formula
\begin{equation}
\label{Leviformdef}
\mathcal{L}(X_x,Y_x):=\frac{i}{2}\left[X,\overline{Y}\right]_x\mod H_x\oplus \overline{H}_x\quad\quad\forall\, X,Y\in \Gamma(H)\mbox{ and }x\in M.
\end{equation}
Here we are using the notation $\Gamma(E)$ to denote sections of a fiber bundle $E$.
The \emph{kernel} of the Levi form $\mathcal L$ is called the \emph{Levi kernel} and will be denoted by $K$. CR-structures with $K=0$ are called Levi-nondegenerate.

For a Levi-nondegenerate structure, if the Levi form has signature $(p,q)$ with $p+q=n$ then a maximally symmetric model can be obtained as a real hypersurface in the complex projective space $\mathbb C\mathbb P^{n+1}$, obtained by the  complex projectivization of the cone of nonzero vectors in $\mathbb C^{n+2}$ that are isotropic with respect to a Hermitian form of signature $(p+1, q+1)$, and the algebra of infinitesimal symmetries of this model is isomorphic to $\mathfrak {su}(p+1, q+1)$, having dimension $(n+2)^2-1$. 

In the present paper we assume that the fiber $K_x$ of the Levi kernel is $1$-dimensional at every point $x\in M$, that is, $K$ is a rank $1$ distribution, and that the following nondegeneracy condition holds: 
If for $v\in K_x$ and $y\in\overline{H}_x/\overline{K}_x$, we take $V\in\Gamma(K)$ and $Y\in\Gamma(\overline{H})$ such that $V(p)=v$ and $ Y(p)\equiv y\mod\overline{K}$, and define a linear map $\ad_v:\overline{H}_x/\overline{K}_x\to H_x/K_x$ by
\begin{equation}\label{adv}
 \ad_v(y):= [V,Y]_x \mod{ K_x\oplus \overline {H}_x},
\end{equation}
and similarly define a linear map $\ad_{v}:H_x/K_x\to \overline{H}_x/\overline{K}_x$ for $v\in \overline K_x$ (or simply take complex conjugates), then there is no nonzero $v\in K_x$ (equivalently, no nonzero $v\in \overline{K}_x$) such that  $\ad_v=0$. A CR-structure is called \emph{$2$-nondegenerate} if this last condition holds. 

The term \emph{2-nondegeneracy} comes from the more general notion of \emph{k-nondegeneracy}, see, for example,  \cite{freeman1977local} for the generalization of this definition to arbitrary  $k\geq 1$  and arbitrary dimension of Levi kernels via the \emph{Freeman sequence} under analogous constant rank assumptions, \cite[chapter XI]{BER99} for more general definition without the assumption that $K$ is a distribution, and \cite[Appendix]{kaupzaitsev} for the equivalence of the definitions in \cite{freeman1977local} and \cite[chapter XI]{BER99} under the constant rank assumptions. 


The focus of the present paper is on finding the sharp upper bound for the  dimension of the Lie group $\Aut(M,H)$ of symmetries of  2-nondegenerate CR structures $(M,H)$ of hypersurface type with a 1-dimensional Levi kernel {admiting a constant reduced modified symbol}, { which is a property with a rather technical definition given in Section \ref{modsec} (Definition \ref{reduced modified symbol def}). Until we give the exact definition of this property, it will suffice to note that structures admitting constant reduced modified symbols are uniformly $2$-nondegenerate and have constant CR symbols. In particular, if $(M,H)$ is homogeneous then it admits a constant reduced modified symbol.}
As shown in \cite{isaev, medori, pocchiola} for the lowest dimensional case, that is  when $\dim M=5$,  this sharp upper bound is equal to $10$,  and for the maximally symmetric model the algebra of infinitesimal symmetries is isomorphic to 
$\mathfrak {so}(3,2)$.
The main result here, see Theorem \ref{main theorem} below, gives this sharp upper bound 
expressed as a function of $\dim M\geq 7$ (equivalently, $n=\frac{1}{2}(\dim M-1)\geq 3$), namely
\begin{align}\label{symmetry group upper bound}
\dim \Aut(M,H)\leq \frac{1}{4}(\dim M-1)^2+7=n^2+7.
\end{align}
We also show that symmetries of $(M,H)$ are all determined by their third  weighted  jet. By the \emph{weighted jet} we mean that the derivatives in various directions are calculated according to the 
filtration 
\[
(K\oplus \overline{K})\cap TM \subset (H\oplus \overline{H})\cap TM\subset TM
\]
of $TM$ so that each derivative in a direction in  $(K\oplus \overline{K})\cap TM$ is assigned weight zero, each derivative in a direction in  $\Big((H\oplus \overline{H})\setminus (K\oplus \overline{K})\Big)\cap TM$ is assigned weight $1$, and each derivative in a direction in  $TM \setminus H\oplus \overline{H}$ is assigned weight $2$.
These results (even without assumption of homogeneity) were previously obtained in \cite{porter2017absolute} for the special class of CR structures whose symbols are known as \emph{regular},
wherein it was  shown by example that the upper bound in \eqref{symmetry group upper bound} is achieved.

The essential technical bulk of this paper consists of  showing that the dimension of $\Aut(M,H)$ for homogeneous structures with non-regular symbol is strictly less than the right side of \eqref{symmetry group upper bound} (in fact it is shown in Theorem \ref{first prolongation for non-regular structures is zero} below  that it is strictly less than $(n-1)^2+7$) and that in the non-regular case symmetries of $(M,H)$ are all determined by their first weighted jet. The notion of CR symbols and their regularity is explained in Section \ref{CRsec}. Note that, for the considered case $n\geq 3$, the previously treated regular symbols constitute only a finite subset in the space of all CR symbols for each $n$, which itself depends on continuous parameters.

In the proof of the bound \eqref{symmetry group upper bound} 
we use two   
main results from our previous papers \cite{sykes2020canonical} and \cite{SZ2020}: the classification of CR symbols \cite{sykes2020canonical} and the description of the upper bound for the dimension of symmetry groups in terms of a Tanaka prolongation of the symbol or its reduced version  \cite{SZ2020}. In the sequel, we calculate these prolongations and their dimensions for each reduced modified symbol corresponding to a non-regular CR symbol. In particular, we show (Theorem \ref{first prolongation for non-regular structures is zero}) that the first Tanaka prolongation of each reduced modified symbol corresponding to a non-regular CR symbol is equal to zero and we find the upper bound for the dimension of its  (entire)  Tanaka prolongation.  Analogous  analysis for regular CR symbols was previously obtained in \cite{porter2017absolute} with the help of the theory of biagraded Tanaka prolongation. The result on the $j$th-jet determinacy follows from its equivalence to the vanishing of the $j$th Tanaka prolongation. In Theorem \ref{more precise bound} for each reduced modified symbol corresponding to a non-regular CR symbol we give  more precise  upper bound for the dimension of its  (entire)  Tanaka prolongation in terms of the parameters of this non-regular symbol.
 
Note that at this moment for structures with non-regular CR symbols (and therefore in the general case) we are not able to remove completely  the assumption of {admitting a constant  reduced modified symbol} in our results, as this assumption 
implies that 
the reduced modified symbols 
are  Lie algebras, and we strongly use the latter fact. 
So the question of whether or not  there exist CR structures from the considered class not admitting a constant reduced modified symbol (Definition \ref{reduced modified symbol def}) and with symmetry group of dimension higher than the bound in \eqref{symmetry group upper bound} is still open, although the positive answer to this question is highly unlikely.

In the very recent paper \cite{beloshapka}  it was shown that for $\dim M=7$, without the homogeneity assumption, the upper bound for the dimension of the group of symmetries of $2$-nondegenerate CR structures of hypersurface type with a 1-dimensional Levi kernel is $17$. Our sharp bound  \eqref{symmetry group upper bound} for the homogeneous case is $16$ and an example of the structure from the considered class with 17-dimensional symmetry group is unknown. 
The  result of the present paper (communicated in a private correspondence)  was in fact used in \cite{beloshapka} to reduce the bound from $18$, obtained initially by the methods of normal forms, to $17$, see Proposition 16 there.    

In contrast to the case of $\dim M=5$, in the case where $M$ is of (odd) dimension greater than or equal to $7$, the infinitesimal symmetry algebras of the maximally symmetric homogeneous models are  not semisimple.
These algebras were calculated  in  some form in \cite[Subsection 5.3]{porter2017absolute}. A more visual description together with a hypersurface realizations of these models  will  feature in future joint work \cite{DPZ}. 

In the case where $\dim M=7$, the infinitesimal symmetry algebra of the maximally symmetric models is isomorphic to one of the real forms of the following complex Lie algebra: Let $\mathfrak{s} =\mathbb C\oplus \mathfrak{sl}(2, \mathbb C)\oplus \mathfrak{sl}(2, \mathbb C)$. The complexification of our algebra of interest is isomorphic to the natural semidirect sum of $\mathfrak{s}$ and the $9$-dimensional abelian Lie algebra $\mathbb C^9\cong \mathbb C^3\otimes \mathbb C^3$ so that  the first $\mathfrak{sl}(2, \mathbb C)$ component in $\mathfrak s$  acts irreducibly on the first factor  $\mathbb C^3$ in $\mathbb C^3\otimes \mathbb C^3$,  the second component $\mathfrak {sl}(2, \mathbb C)$ in $\mathfrak s$  acts irreducibly on the second factor  of $\mathbb C^3$ in $\mathbb C^3\otimes \mathbb C^3$, and the component $\mathbb C$ in $\mathfrak s$ acts just by rescaling. The desired real Lie algebra is the natural semidirect sum of the conformal Lorenzian algebra  $ \mathfrak{co}(3,1)$  
and  the $9$-dimensional real abelian  Lie algebra $\R^9$,  where $\mathfrak{co}(3,1)$ acts irreducibly on $\R^9$. This unique irreducible action is naturally induced from the standard action of  $\mathfrak{co}(3,1)$ on the Minkowski space, if one identifies $\R^9$ with the space of the traceless symmetric bilinear forms  on the Minkowski space.

Finally, for completeness, we offer without proof the (local) hypersurface realizations of the maximally symmetric homogeneous models in the considered class (the details will be given in \cite{DPZ}). If, as before, $n=\frac{1}{2}(\dim M-1)$, and the signature of the form obtained by the reduction of the Levi form at each point $x$ to the space $H_x/K_x$  is equal to $(p, q)$ with $p+q=n-1$, then in coordinates $(z_1,\ldots, z_n, w)$ for $\mathbb C^{n+1}$ these are the hypersurfaces are given by the equation
\begin{equation}
\label{maxmod}
\mathrm{Im}(w+z_1^2 \bar z_n)=z_1 \bar z_2 +\bar z_1 z_2 +\sum_{i=3}^{n-1} \varepsilon_i z_i\bar z_i,
\end{equation}
where $\varepsilon_i \in\{-1, 1\}$ and $\{\varepsilon _i\}_{i=3}^{n-1}$ consists of $p-1$ terms equal to $1$ and $q-1$ terms equal to $-1$ (note that, for $\dim M=7$, the last sum in the right side of \eqref{maxmod} disappears).


 \section{CR symbols and the main results}
\label{CRsec}
Our analysis branches depending on properties of the CR structure's local invariants. A basic local invariant of a hypersurface-type CR structure called the \emph{CR symbol} is introduced in \cite{porter2017absolute}. The CR symbol of $H$ (at a point $x$ in $M$) is a bigraded vector space 
\begin{equation}
\label{bigrade_g0}
\mathfrak{g}^0:=\mathfrak{g}_{-2,0}\oplus\mathfrak{g}_{-1,-1}\oplus\mathfrak{g}_{-1,1}\oplus\mathfrak{g}_{0,-2}\oplus\mathfrak{g}_{0,0}\oplus\mathfrak{g}_{0,2}
\end{equation}
 with involution $\bar{~}$ whose bigraded components $\mathfrak{g}_{i,j}$ are defined as follows. Ultimately our definitions of $\mathfrak{g}_{i,j}$ will not depend on the point $x$ because {going forward we will consider only structures with constant CR symbols}, but we still fix $x$ to state the initial definitions.
 We let $\ell$ denote the \emph{reduced Levi form}, which is the field of nondegenerate Hermitian forms defined on fibers of the quotient bundle $H/K$ by 
\[
\ell(X_x+K_x):=\mathcal{L}(X_x).
\]
 We define the coset spaces
\[
\mathfrak{g}_{-2,0}:=\mathbb C T_xM/H_x, \quad \mathfrak{g}_{-1,-1}:=\overline{H}_x/\overline{K}_x, \quad\mbox{ and }\quad \mathfrak{g}_{-1,1}:=H_x/K_x.
\]
The space 
\begin{equation}
\label{Heisenberg}
\mathfrak{g}_{-}:=\mathfrak{g}_{-2,0}\oplus\mathfrak{g}_{-1,-1}\oplus\mathfrak{g}_{-1,1}
\end{equation}
inherits a Heisenberg algebra structure with  nontrivial Lie brackets defined in terms of the reduced Levi form by 
\[
[v,w]:=i\ell(v,w)\quad\quad\forall v\in \mathfrak{g}_{-1,1},\, w\in \mathfrak{g}_{-1,-1}.
\]
Note that $\ell$ formally takes values in $\mathfrak{g}_{-2,0}$. By identifying $\mathfrak{g}_{-2,0}$ and $\mathbb C$, we regard $\ell$ as a $\mathbb C$-valued Hermitian form, but, since this identification is not naturally determined by the CR structure, in the sequel we consider the real line $\R \ell$ of $\mathbb C$-valued Hermitian forms spanned by $\ell$. While the one $\mathbb C$-valued form $\ell$ is not an invariant of the CR structure, the line $\R \ell$ is.

To define $\mathfrak{g}_{0,2}$, we consider special operators associated with vectors in $K_x$. For a vector $v$ in $K_x$,
define the antilinear operator $A_{v}:\mathfrak{g}_{-1,1}\to \mathfrak{g}_{-1,1}$ by
\begin{equation}\label{antilinear op}
A_{v}(x):= \mathrm{ad}_v(\overline{x}).
\end{equation}
The dependence of $A_{v}$  on $v$ is linear, that is,
\[
A_{\lambda v}=\lambda A_{v}\quad\quad\forall \lambda\in\mathbb C,
\]
so if the rank of $K$ is equal to 1 then there exists  an antilinear operator $\boldsymbol{A}$ such that
\[
\{A_v\,|\, v\in K_x\}=\mathbb C \boldsymbol{A}.
\]
The fact that $H$ is $2$-nondegenerate implies that $\boldsymbol{A}\neq 0$.

The reduced Levi form $\ell$ naturally extends to define a symplectic form on the space $\mathfrak{g}_{-1}:=\mathfrak{g}_{-1,-1}\oplus \mathfrak{g}_{-1,1}$ via a standard construction from the study of Heisenberg algebras. Hence $\mathfrak{g}_{-1}$ inherets a symplectic structure from the CR structure with respect to which we obtain the conformal symplectic algebra $\mathfrak{csp}(\mathfrak{g}_{-1})$ defined in the standard way. We define $\mathfrak{g}_{0,2}$ to be the subspace of $\mathfrak{csp}(\mathfrak{g}_{-1})$  given by the formula
\[
\mathfrak{g}_{0,2}:=\left\{\varphi:\mathfrak{g}_{-1}\to \mathfrak{g}_{-1}\,\left|\parbox{5cm}{$\varphi(v)=0\quad\forall v\in \mathfrak{g}_{-1,1}$ and there exists $\lambda\in \mathbb C$ such that \\ $\varphi(v)= \lambda \boldsymbol{A}(\overline{v})\quad\forall v \in \mathfrak{g}_{-1,-1}$}\right.\right\}.
\]
The natural complex conjugation on $\mathbb C T_xM$ induces an antilinear involution $v\mapsto \overline{v}$ on $\mathfrak{g}_{-1}$, which in turn induces an antilinear involution on $\mathfrak{csp}(\mathfrak{g}_{-1})$ by the rule
\begin{equation}
\label{conj}
\overline{\varphi}(v):=\overline{\varphi(\overline{v})}.
\end{equation}
Using this involution, we define
\[
\mathfrak{g}_{0,-2}:=\{\varphi\,|\,\overline{\varphi}\in \mathfrak{g}_{0,2}\}.
\]
Lastly, using the standard Lie brackets of $\mathfrak{csp}(\mathfrak{g}_{-1})$ we define 
\begin{equation}
\label{g00}
\mathfrak{g}_{0,0}:=\big\{v\in \mathfrak{csp}(\mathfrak{g}_{-1})\,\left|\, [v,\mathfrak{g}_{0,i}]\subset \mathfrak{g}_{0,i}\quad\forall \, i\in\{-2,2\}\big\}\right.,
\end{equation}
which completes our definition of the CR symbol $\mathfrak{g}^0$ of $H$ (at the point $x$).
Note that by construction 
\begin{equation}
\label{bigradedcond}
[\mathfrak g_{i_1, j_1},\mathfrak g_{i_2, j_2}]\subset \mathfrak g_{i_1+i_2, j_1+j_2}, \quad\quad\forall\, \{(i_1, j_1), (i_2, j_2)\}\neq \{(0,2), (0, -2)\}.
\end{equation}

Conversely a vector space $\mathfrak{g}^0$ as in \eqref{bigrade_g0} with $\mathfrak{g}_-$ as in \eqref{Heisenberg} being the Heisenberg algebra 
 is called an \emph{abstract CR symbol} for $2$-nondegenerate, hypersurface-type CR structures if it satisfies \eqref{bigradedcond}, $\mathfrak g_{0,0}$ is the maximal subalgebra of $\mathfrak{csp}(\mathfrak{g}_{-})$ satisfying \eqref{g00}, and it is endowed with an antilinear involution $\bar{~}$ satisfying \eqref{conj}.
\begin{remark}
The CR symbol $\mathfrak{g}^0$ of a CR structure with a $1$-dimensional kernel encodes and is encoded by the pair $(\R\ell, \mathbb{C}\boldsymbol{A})$. 
\end{remark}
Note that an abstract CR symbol $\mathfrak g^0$ is not necessarily a Lie algebra, as the bigrading  conditions in \eqref{bigradedcond}  are only applied for $\{(i_1, j_1), (i_2, j_2)\}\neq \{(0,2), (0, -2)\}$, so that $[\mathfrak g_{0,-2},\mathfrak g_{0,2}]$ does not necessarily belong to $\mathfrak g_{0,0}$  and therefore does not necessarily belong to $\mathfrak g^0$.  
Following the terminology of \cite{porter2017absolute}, we say that a CR symbol is \emph{regular} if it is a subalgebra of $\mathfrak{g}_{-}\rtimes\mathfrak{csp}(\mathfrak{g}_{-})$ and \emph{non-regular} otherwise.  As shown in \cite[Lemma 4.2]{porter2017absolute}, the symbol $\mathfrak{g}^0$ of a CR structure with a  $1$-dimensional kernel corresponding to the pair $(\R\ell, \mathbb{C}\boldsymbol{A})$ is regular if and only if \begin{equation}
\label{cube}
\boldsymbol{A}^3\in\mathbb{C}\boldsymbol{A}.
\end{equation}  
 
To any abstract regular CR symbol  $\mathfrak g^{0}$, we construct a corresponding special homogeneous CR structure as follows.   Denote by $G^{0}$ and $G_{0,0}$ connected Lie groups with Lie algebras $\mathfrak  g^{0}$ and $\mathfrak g_{0,0}$, respectively, such that $G_{0,0}\subset G^{0}$, and denote by  $\Re G^{0}$ and $\Re G_{0,0}$ the corresponding real parts with respect to the involution on $\mathfrak g^{0}$, meaning that $\Re G^{0}$ and $\Re G_{0,0}$ are the maximal subgroups of $ G^{0}$ and $ G_{0,0}$ whose tangent spaces belong to the left translations of the fixed point set of the involution on $\mathfrak g^{0}$ on $G^{0}$.

Let $ M_0^{\mathbb C} = G^{0}/ G_{0,0}$ and $M_0= \Re G^{0}/\Re G_{0,0}$. In both cases here we  use left cosets. Let $\widehat D_{i,j}^{\mathrm{flat}}$  be the left-invariant distribution on $G^{0}$ such that it is equal to $\mathfrak{g}_{i,j}$ at the identity. 
Since all $\mathfrak{g}_{i,j}$ are invariant under the adjoint action of $G_{0,0}$, the push-forward of each $\widehat D_{i,j}^{\mathrm{flat}}$ to $M_0^\mathbb C$  is a well defined distribution, which we denote by $D^{\mathrm{flat}}_{i,j}$. Let $D_{-1}^{\mathrm{flat}}$ be the distribution that is the sum of $D_{i,j}^{\mathrm{flat}}$ with $i=-1$. We restrict all of these distributions to $M_0$, considering them as subbundles of the complexified tangent bundle of $M_0$. The distribution $ H^{\mathrm{flat}}:=D_{-1,1}^{\mathrm{flat}}\oplus D_{0,2}^{\mathrm{flat}}$ defines a CR structure of hypersurface type on $M_0$
called the \emph{flat CR structure with constant CR symbol $\mathfrak g^{0}$}.

As a consequence of \cite{porter2017absolute}, see Theorems 3.2, 5.1, 5.3 and the last paragraph of section 5 there, one gets

\begin{theorem}[\citet{porter2017absolute}]\label{main theorem for regular symbols}
If $(M,H)$ is a 2-nondegenerate CR structure of hypersurface type with a 1-dimensional Levi kernel and constant regular symbol, then 
\begin{enumerate}
    \item the dimension of the algebra of infinitesimal symmetries of $(M,H)$  is not greater than  $\tfrac{1}{4}(\dim M-1)^2+7$; 
     \item these symmetries are determined by their third weighted jet;
    \item the dimension of the algebra of infinitesimal symmetries of $(M,H)$  is equal to   $\tfrac{1}{4}(\dim M-1)^2+7$ if and only if $(M, H)$ is locally equivalent to the flat structure with CR symbol such that the corresponding line of antilinear operators consists of nilpotent ones of rank 1.
    \end{enumerate}
\end{theorem}
A natural question is whether or not the assumption of regularity of symbol can be removed in the previous theorem. Addressing this question, the main result of the present paper is the following.

\begin{theorem}\label{main theorem}
If $(M,H)$ is a 2-nondegenerate CR structure of hypersurface type with a 1-dimensional Levi kernel {admitting a constant reduced modified symbol as in Definition \ref{reduced modified symbol def}} (and, in particular, if it is homogeneous), then 
\begin{enumerate}
    \item statements (1) and (3) of Theorem \ref{main theorem for regular symbols} are valid;
    \item if the symbol is non-regular then the (infinitesimal) symmetries of $(M, H)$ are determined by their first  weighted jet.
\end{enumerate}
\end{theorem}

The proof of this theorem is given in Sections \ref{modsec} through \ref{First prolongation of non-regular symbols} and the appendix. 
In Section \ref{modsec} we give the 
scheme of the proof of this theorem, based on the constructions and results of our previous paper \cite{SZ2020}, namely the construction of reduced modified symbols for sufficiently symmetric CR structures and the application of Tanaka prolongation of these reduced modified symbols to obtain an upper bound for the dimension of their infinitesimal symmetry algebras (see Theorem \ref{main theorem for non-regular symbols} below). In this way Theorem \ref{main theorem} will be essentially reduced to Theorem \ref{first prolongation for non-regular structures is zero}. The latter theorem is proved in Section \ref{First prolongation of non-regular symbols} with the help of the appendix (Section \ref{Matrix representations of intersection algebra general formula}). In this proof we also use the classification of symbols from our previous paper \cite{sykes2020canonical}  and the system of matrix equations for the reduced modified symbols derived in \cite [section 5]{SZ2020}. The latter two topics are briefly reviewed in Section \ref{Matrix representations of local invariants} below.


\section {Reduced modified symbol and the significance of its Tanaka prolongation} 
\label{modsec}

Now we will discuss the scheme of the proof of Theorem \ref{main theorem}, based on the constructions and results of our previous paper \cite{SZ2020}. 
In particular, there 
we introduced other local invariants of sufficiently symmetric hypersurface-type CR structures encoded in objects called \emph{modified CR symbols} and \emph{reduced modified CR symbols}  (see sections 4 and 6 of  \cite{SZ2020}, respectively). Although modified and reduced modified CR symbols are defined in \cite{SZ2020}, we outline their definitions here for completeness because these objects (especially the latter one) are both nonstandard and fundamental for the present study. Some technical details that are not essential for understanding the principal concepts are omitted here and we refer to \cite{SZ2020} for those gaps. {Following these definitions, we introduce Theorem \ref{first prolongation for non-regular structures is zero}, and describe how Theorem \ref{main theorem} essentially follows from Theorem \ref{first prolongation for non-regular structures is zero}. The subsequent sections of this paper are dedicated to the proof of Theorem \ref{first prolongation for non-regular structures is zero}.}

{Proceeding, we assume that $(M,H)$ has a constant CR symbol.} Let $\mathfrak{g}^0$ be {an abstract CR symbol isomorphic to} the CR symbol {$\mathfrak{g}^0(x)$ of $(M,H)$ at every point $x$ in $M$}. 
And write $\mathfrak{g}_{i, j}(x)$ to denote the bigraded components of $\mathfrak{g}^0(x)$.

There is a natural way to locally \emph{complexify} $M$ by working in local coordinates and replacing real coordinates with complex ones, and, moreover, the CR structure $H$, as well as the distributions $\overline{H}$, $K$, and $\overline{K}$, naturally extend to this complexified manifold (see \cite{SZ2020} for full details) yielding a so-called \emph{complexified CR manifold} that we denote by $\mathbb C M$ (a detail omitted here is that, since the construction is local, this may only be well defined after replacing $M$ with some neighborhood in $M$). Note that $\dim_\R(\mathbb C M)=2\dim (M)$ and there is a submanifold in $\mathbb C M$ that can be naturally identified with $M$. The distribution $K+\overline{K}$ on $\mathbb C M$ is involutive. We let $\mathcal{N}$ be the leaf space of the foliation of $\mathbb C M$ generated by $K+\overline{K}$, sometimes called the \emph{Levi leaf space}, and let $\pi:\mathbb C M\to \mathcal{N}$ denote the natural projection. That is, points in $\mathcal{N}$ are maximal integral submanifolds of $K+\overline{K}$ in $\mathbb C M$.

From the resulting construction, $\mathfrak{g}^0(x)$ remains well defined (in terms of $H$) for all $x$ in $\mathbb C M$. We define the fiber bundle $\mathrm{pr}: P^0\to \mathbb C M$ whose fiber $\mathrm {pr}^{-1} (x)$ over a point $x$ in $\mathbb C M$ is comprised of what we call \emph{adapted frames}, that is, 
\begin{align}
\label{fiberP0}
\mathrm{pr}^{-1}(x)=\left\{
\varphi:\mathfrak g_-\to \mathfrak g_-(x) \left|\ 
\parbox{7.9cm}{
$\varphi(\mathfrak g_{i,j})=\mathfrak g_{i,j}(x)\quad\forall\, (i,j)\in \{(-1, \pm 1),(-2, 0)\}$,\\
$\varphi^{-1}\circ \mathfrak g_{0,\pm2}(x)\circ\varphi=\mathfrak g_{0,\pm2}$, and \\
$\varphi([ y_1, y_2])=[\varphi( y_1),\varphi( y_2)] \quad\forall\, y_1, y_2\in\mathfrak g_-$
}
\right.\right\}.
\end{align}
We also consider a second fiber bundle $\pi\circ \mathrm{pr}:P^0\to \mathcal{N}$, a bundle with total space $P^0$ and base space $\mathcal{N}$.

For any $\psi\in P^0$ and $\gamma=\pi\circ\mathrm{pr}(\psi)$, the tangent space of the fiber $(P^0)_{\gamma}=(\pi\circ \mathrm{pr})^{-1}(\gamma)$ of the second bundle at $\psi$ can be identified with a subspace of $\mathfrak{csp}(\mathfrak{g}_{-1})$ by the map $\theta_0:T_\psi(P^0)_{\gamma}\to \mathfrak{csp}(\mathfrak{g}_{-1})$ given by
\begin{align}\label{degree zero soldering form section 4}
\theta_0\big(\psi^\prime(0)\big):=\big(\psi(0)\big)^{-1}\psi^\prime(0)
\end{align}
where $\psi:(-\epsilon,\epsilon)\to (P^0)_{\gamma}$  denotes an arbitrary curve in $(P^0)_{\gamma}$ with $\psi(0)=\psi$. The notation $\theta_0$ is used here to match the notation in \cite{SZ2020}. Let 
\begin{equation}
\label{modsymeq}
\mathfrak g_0^{\mathrm{mod}}(\psi):=\theta_0(T_\psi(P^0)_{\gamma}).
\end{equation}

\begin{definition}\label{modified CR symbol def}
The space $\mathfrak g^{0, \mathrm{mod}}(\psi):=\mathfrak g_-\oplus \mathfrak g_0^{\mathrm{mod}}(\psi) $ is called the \emph{modified CR symbol} of the CR structure $H$ at the point $\psi\in P^0$.  
\end{definition}

\begin{remark}
Modified CR symbols depend on points in the bundle $P^0$ rather than points in the original CR manifold. Accordingly, a modified CR symbol is not itself a local invariant of the CR strucuture from which it arises, but rather, for $x\in M$, the set $\{\mathfrak{g}^{0, \mathrm{mod}}(\psi)\,|\, \mathrm{pr}(\psi)=x\}$ is a local invariant at $x$. This invariant encodes more data than is encoded in the corresponding CR symbol.
\end{remark}
\begin{remark}
Definition \ref{modified CR symbol def} can be made without assuming that $(M,H)$ is homogeneous, and instead assuming only that the CR symbols $\mathfrak{g}^0(x)$ are constant on $M$.
\end{remark}

We consider the map $\psi\mapsto \varphi_0(\psi)$ sending each point in $P^0$ to a subspace of $\mathfrak{csp}(\mathfrak{g}_{-})$. If, for some subspace $\widetilde{\mathfrak{g}_0}\subset\mathfrak{csp}(\mathfrak{g}_{-})$, there is a maximal connected submanifold $\widetilde{P^0}$ of $P^0$ belonging to the level set  
\[
\left\{\psi\in P^0\,\left|\,\theta_0\left(T_\psi\left( P^{0}\right)_{\pi\circ\mathrm{pr}(\psi)}\right)=\widetilde{\mathfrak{g}_0}\right.\right\}
\]
such that $\mathrm{pr}(\widetilde{P^0})=\mathbb C M$, then we call $\widetilde{P^0}$ a reduction of $P^0$. After, replacing $P^0$ and $\theta_0$ with $\widetilde{P^0}$ and the restriction of $\theta_0$ to the vertical tangent vectors of $\pi\circ \mathrm{pr}:\widetilde{P^0}\to\mathcal{N}$, we can repeat this reduction procedure by finding a maximal connected submanifold of $P^0$ that is in the level set of the new mapping $\psi\mapsto \theta_0\left(T_{\psi}\left(\widetilde{P^0}\right)_{\pi\circ\mathrm{pr}(\psi)}\right)$ also covering $\mathbb C M$ under the projection $\mathrm{pr}$, which we again call a reduction of $P^0$. In general this reduction procedure can be repeated many times, and eventually terminates in the sense that iterating the reduction procedure again will not yield new reductions. For a reduction $P^{0,\mathrm{red}}$ of $P^0$ we label the corresponding space
\begin{align}\label{reduced modified symbol zero degree def}
\mathfrak{g}_{0}^\mathrm{red}(\psi):=\theta_0\left(T_\psi\left( P^{0,\mathrm{red}}\right)_{\pi\circ\mathrm{pr}(\psi)}\right)
\quad\quad\forall\,\psi\in P^{0,\mathrm{red}}.
\end{align}

\begin{definition}\label{reduced modified symbol def}
If $P^{0,\mathrm{red}}$ is a reduction of $P^0$ then the space $\mathfrak g^{0, \mathrm{red}}(\psi):=\mathfrak g_-\oplus \mathfrak g_0^{\mathrm{red}}(\psi)$, with $\mathfrak{g}_0^{\mathrm{red}}(\psi)$ given by \eqref{reduced modified symbol zero degree def}, is called a \emph{reduced modified CR symbol} of the CR structure $H$ at $\psi$. We say that $H$ \emph{admits a constant reduced modified CR symbol} $\mathfrak{g}^{0,\mathrm{red}}$ if there exists a reduction $P^{0,\mathrm{red}}$ of $P^0$ together with $\mathfrak{g}_0^{\mathrm{red}}(\psi)$ given by \eqref{reduced modified symbol zero degree def} such that
\[
\mathfrak{g}^{0,\mathrm{red}}=\mathfrak{g}^{0,\mathrm{red}}(\psi)\quad\quad\forall\, \psi\in P^{0,\mathrm{red}}.
\]
\end{definition}

\begin{lemma}
If $(M,H)$ is homogeneous then it admits a constant reduced modified symbol, that is, there exists a reduction $P^{0,\mathrm{red}}$ of $P^0$ such that the map $\psi\mapsto \mathfrak{g}_{0}^\mathrm{red}(\psi)$ given by \eqref{reduced modified symbol zero degree def} is constant.
\end{lemma}
\begin{proof}
Since $(M, H)$ is homogeneous, so is $P^0$, and hence each reduction $\widetilde{P^0}$ of $P^0$ can be taken so that its fibers $\left(\widetilde{P^0}\right)_{x}:=\left\{\psi\in\widetilde{P^0}\,|\, \pi(\psi)=x \right\}$ have the same image under the mapping $\psi\mapsto \theta_0\left(T_\psi \widetilde{P^0}\right)$. Therefore, if $\psi\mapsto \theta_0\left(T_\psi \widetilde{P^0}\right)$ is not already constant on $\widetilde{P^0}$ then we can repeat the reduction procedure to find a proper submanifold of $\widetilde{P^0}$ that is also a reduction of $P^0$. Eventually, this iterated procedure ends with a reduction for which either the image of $\theta_0$ applied to its tangent spaces is constant, or a its fibers are 0-dimensional. But, in the latter case, using homogeneity, we can take this final reduction $P^{0,\mathrm{red}}$ such that its fibers have the same image under the mapping $\psi\mapsto \theta_0\left(T_\psi P^{0,\mathrm{red}}\right)$. Accordingly, $\psi\mapsto \theta_0\left(T_\psi P^{0,\mathrm{red}}\right)$ would be constant on $P^{0,\mathrm{red}}$ because it is constant on fibers and the fibers are singletons.
\end{proof}

For the remainder of this paper, we let $\mathfrak{g}^{0,\mathrm{red}}$ denote a constant reduced modified CR symbol of $H$. Like the CR symbol of $H$, $\mathfrak{g}^{0,\mathrm{red}}$ is also a graded subspace of $\mathfrak{g}_{-}\rtimes \mathfrak{csp}(\mathfrak{g}_{-1})$. It has the decomposition $\mathfrak{g}^{0,\mathrm{red}}=\mathfrak{g}_{-2,0}\oplus\mathfrak{g}_{-1,-1}\oplus\mathfrak{g}_{-1,1}\oplus\mathfrak{g}^{\mathrm{red}}_{0}$ where the components whose first weight is negative coincide with those of the CR symbol. 
Here we state some of the properties of $\mathfrak{g}_0^{\mathrm{red}}$. For this we consider weighted components of $\mathfrak{csp}(\mathfrak{g}_{-1})$ defined by
\[
\big(\mathfrak{csp}(\mathfrak{g}_{-1})\big)_{0,i}=\big\{\varphi\in\mathfrak{csp}(\mathfrak{g}_{-1})\,\left|\, \varphi(\mathfrak{g}_{-1,j})\subset \mathfrak{g}_{-1,i+j}\, \forall j\in\{-1,1\} \big\}\right..
\]

The space $\mathfrak{g}_0^{\mathrm{red}}$ is a subspace of $\mathfrak{csp}(\mathfrak{g}_{-1})$ with a decomposition  
\begin{align}\label{gZeroRed splitting}
\mathfrak{g}_0^{\mathrm{red}}=\mathfrak{g}^{\mathrm{red}}_{0,0}\oplus\mathfrak{g}^{\mathrm{red}}_{0,-}\oplus\mathfrak{g}^{\mathrm{red}}_{0,+}
\end{align} such that 
\begin{enumerate}
    \item $\mathfrak{g}_{0,0}^{\mathrm{red}}\subset\mathfrak{g}_{0,0}$;
    \item $\mathfrak{g}^{\mathrm{red}}_{0,+}=\overline{\mathfrak{g}^{\mathrm{red}}_{0,-}}$;
    \item the natural projection of $\mathfrak{csp}(\mathfrak{g}_{-1})$ onto $\big(\mathfrak{csp}(\mathfrak{g}_{-1})\big)_{0,2}$ defines an isomorphism between $\mathfrak{g}^{\mathrm{red}}_{0,+}$ and $\mathfrak{g}_{0,2}$;
    \item  The subspace $\mathfrak g_0^{\mathrm{red}}$ is invariant with respect to the involution on $\mathfrak {csp}(\mathfrak g_{-1})$
    \item  The subspace $\mathfrak g_0^{\mathrm{red}}$ is a subalgebra of $\mathfrak {csp}(\mathfrak g_{-1})$.
\end{enumerate}
We stress that the decomposition $\mathfrak{g}_0^{\mathrm{red}}=\mathfrak{g}^{\mathrm{red}}_{0,0}\oplus\mathfrak{g}^{\mathrm{red}}_{0,-}\oplus\mathfrak{g}^{\mathrm{red}}_{0,+}$ satisfying these properties is not unique, and, furthermore, no such splitting is naturally determined by the CR structure.

\begin{remark}
The CR symbol of $(M,H)$ is determined by any of its modified CR symbols, which in turn are all determined by any constant reduced modified CR symbol $\mathfrak{g}^{0,\mathrm{red}}$ that $(M,H)$ admits. 
\end{remark}

The underlying theory that we will apply to treat structures with non-regular CR symbols is developed in \cite{SZ2020}, wherein it is shown that the upper bounds that we wish to compute can be found by computing the universal Tanaka prolongation \cite{tanaka1} of $\mathfrak{g}^{0,\mathrm{red}}$, which is defined as follows. Starting with $k=1$ and setting $\mathfrak{g}_{-2}=\mathfrak{g}_{-2,0}$, we recursively define the vector spaces
\begin{align}\label{kth prolongation}
\mathfrak g_k^{\mathrm{red}} :=\left\{\varphi\in \bigoplus_{i=-2}^{-1}{\rm Hom}(\mathfrak  g_i,\mathfrak g_{i+k})\left|\,\parbox{6cm}{$\varphi ([v_1,v_2])=[\varphi (v_1),v_2]+[v_1, \varphi( v_2)]$\\$\forall\, v_1, v_2 \in \mathfrak g_-$}\right.\right\} \quad\quad\forall\, k\geq 1,
\end{align}
The \emph{universal Tanaka prolongation} of $\mathfrak g^{0, \mathrm {red}}$
is the vector space
\begin{equation}
\label{universal prolongation definition}
\mathfrak u(\mathfrak g^{0, \mathrm {red}}):=\mathfrak g_-\oplus \bigoplus_{k\geq 0} \mathfrak g_k^{\mathrm{red}}.
\end{equation}
\begin{theorem}[follows immediately from  {\cite[Corollary 2.8 and Theorem 6.2]{SZ2020}}]\label{main theorem for non-regular symbols}
If $(M,H)$ is a 2-nondegenerate CR structure of hypersurface type with a 1-dimensional Levi kernel and constant reduced modified symbol $\mathfrak g^{0, \mathrm {red}}$, then the dimension of the algebra of infinitesimal symmetries of $(M,H)$  is not greater than $\dim \mathfrak{u}(\mathfrak g^{0, \mathrm {red}})$.
\end{theorem}
Hence, if we can explicitly calculate $\dim \mathfrak{u}(\mathfrak g^{0, \mathrm {red}})$ for non-regular CR symbols, then we can obtain an upper bound for the algebra of infinitesimal symmetries of $(M,H)$. This motivates the following theorem, proved in Section \ref{First prolongation of non-regular symbols}.

\begin{theorem}\label{first prolongation for non-regular structures is zero}
If a {constant} reduced modified CR symbol $\mathfrak g^{0, \mathrm {red}}$ corresponds to a non-regular CR symbol  then the following statements hold:
	\begin{enumerate} 
	 \item The first Tanaka  prolongation $\mathfrak{g}_1^{\mathrm{red}}$ of $\mathfrak g^{0, \mathrm {red}}$ vanishes or, equivalently,  the universal Tanaka prolongation $\mathfrak u(\mathfrak g^{0, \mathrm {red}})$  of $\mathfrak g^{0, \mathrm {red}}$ is equal to $\mathfrak g^{0, \mathrm {red}}$. 
	 \item  $\dim\, \mathfrak g^{0, \mathrm {red}}$ and therefore  the dimension of the algebra of infinitesimal symmetries of a $(2n+1)$-dimensional $2$-nondegenerate CR structure of hypersurface type with rank $1$ Levi kernel and  non-regular CR symbol {admitting a constant reduced modified symbol} is strictly less than $(n-1)^2+7$.  
	 \item {For $(M,H)$ as in item ($2$), the bundle $\mathrm{pr}:\Re(P^0)\to M$, consisting of frames in $P^0$ that commute with complex conjugation on the CR symbols, is a principal bundle over $M$ whose structure group has the Lie algebra $\mathfrak{g}_{0,0}^{\mathrm{red}}$ and it is equipped with an absolute parallelism invariant under the structure group's action and under the natural induced action of symmetries of $(M,H)$.}
     \end{enumerate}
\end{theorem}
{\begin{corollary}
The dimension of the algebra of infinitesimal symmetries of a homogeneous  $(2n+1)$-dimensional $2$-nondegenerate CR structure of hypersurface type with rank $1$ Levi kernel and  non-regular CR symbol is strictly less than $(n-1)^2+7$.  
\end{corollary}}

Theorem \ref{first prolongation for non-regular structures is zero} is proved in Section \ref{First prolongation of non-regular symbols} with the help of preliminary results established in Section \ref{Matrix representations of local invariants} and the appendix (Section \ref{Matrix representations of intersection algebra general formula}). In section \ref{Matrix representations of local invariants}, we introduce a standardized matrix representation of abstract reduced modified symbols, which is necessary for our study because there is no previously developed structure theory for these Lie algebras. In the appendix (Section \ref{Matrix representations of intersection algebra general formula}), we give explicit general formulas for matrix representations of elements in $\mathfrak{g}_{0,0}^{\mathrm{red}}$, and we use these formulas to calculate upper bounds for the dimension of $\mathfrak{g}^{0,\mathrm{red}}$, which are necessary for item (2) of Theorem \ref{first prolongation for non-regular structures is zero}. These results of Section \ref{Matrix representations of intersection algebra general formula} are differed to the appendix because their proofs require somewhat  digressive linear algebra that readers may wish to initially take for granted when studying the main points of this paper. Lastly, in Section \ref{First prolongation of non-regular symbols}, we apply the matrix representation formulas derived in Section \ref{Matrix representations of intersection algebra general formula} to prove item (1) of Theorem \ref{first prolongation for non-regular structures is zero} by directly calculating $\mathfrak{g}_1^{\mathrm{red}}=0$.

Based on the well-known fact \cite[Section 6]{tanaka1} that an infinitesimal symmetry of a filtered structure is determined by the $j$th weighted jet, where $j$ is the minimal nonnegative integer for which the $j$th Tanaka prolongation is equal to zero, this theorem immediately implies item (2) of Theorem \ref{main theorem}. Item (1) of Theorem \ref{main theorem} will follow from combining Theorems \ref{first prolongation for non-regular structures is zero} and \ref{main theorem for non-regular symbols}. In Theorem \ref{more precise bound} below, for each reduced modified symbol corresponding to a non-regular CR symbol, we give  more precise  upper bounds (than the ones in item (2) of Theorem \ref{first prolongation for non-regular structures is zero}) for the dimension of its  (entire)  Tanaka prolongation in terms of the parameters of this non-regular symbol.

\begin{remark}
To establish Theorem \ref{first prolongation for non-regular structures is zero}, we appeal to Theorem \ref{main theorem for non-regular symbols} and the Tanaka-theoretic prolongation procedures developed in \cite{SZ2020} which constructs a tower $\Re(P^s)\to\Re(P^{s-1})\to \cdots \to\Re(P^{0})\to M$ of fiber bundles (geometric prolongations) and confers an absolute parallelism onto largest prolongation $\Re(P^s)$. The familiar reader will notice that item (1) in Theorem \ref{first prolongation for non-regular structures is zero} implies that $\Re (P^0)$ is diffeomorphic to the largest prolongation, and may wonder if we can construct a parallelism on $P^0$ directly without invoking the full prolongation procedure theory. We stress, however, that
in general, for a Tanaka structure of depth $\mu$, where $\mu$ is the number of negatively graded components, if the $l$ is the maximal integer such that the $l$th algebraic prolongation is not equal to zero, then the parallelism construction requires constructing $(l+\mu)$th geometric prolongation, and in our setting $\mu=2$. Contrastingly, the classical prolongation theory for $G$-structures (whose depth is $\mu=1$) enjoys greater simplification whenever $\mathfrak{g}_1=0$, so that in this case the construction of the parallelism requires the first geometric prolongation only. See \cite{alekseevsky2015tanaka,alekseevsky2001prolongations,SZ2020,tanaka1,zelenko2009tanaka} for detailed exposition of the prolongation procedure.
\end{remark}






\section{Matrix representations of CR and reduced modified CR symbols} \label{Matrix representations of local invariants}

Throughout this section, we work with  a fixed CR symbol given by the pair $(\R\ell, \mathbb{C}\boldsymbol{A})$, where $\ell$ is an Hermitian  form and $\boldsymbol{A}$ is a self-adjoint antilinear operator on $\mathfrak g_{-1,1}$.  Let us fix a basis of $\mathfrak{g}_{-1}$. This basis can be fixed such that the pair $(\ell,\boldsymbol{A})$ is represented with respect to it by matrices in a canonical form, which is shown in \cite{sykes2020canonical}. We recall one such canonical form below in Theorem \ref{simultaneous canonical form theorem} (there are actually two canonical forms given in \cite{sykes2020canonical}).

For $\lambda\in\mathbb C$ and a positive integer $m$, let $J_{\lambda,m}$ denote the $m\times m$ Jordan matrix with a single eigenvalue $\lambda$ and this eigenvalue has geometric multiplicity 1; let $T_m=J_{0,m}$, and let $S_m$ be the $m\times m$ matrix whose $(i,j)$ entry is 1 if $j+i=m+1$ and zero otherwise, that is

\begingroup\setlength{\abovedisplayskip}{8pt}
\begin{equation}\label{Matrix representations first notations}
J_{\lambda,m}:=\left.\left(\bovermat{$m$ columns}{ \begin{matrix}
\lambda&1&0&\cdots &0\\
0&\ddots&\ddots&\ddots&\vdots\\
\vdots&\ddots&\ddots&\ddots&0\\
\vdots&&\ddots&\ddots&1\\
0&\cdots&\cdots&0&\lambda
\end{matrix}}\right)\right\} \text{\scriptsize $m$ rows}
\quad\mbox{ and }\quad
S_m=
\left.
\left(
\bovermat{$m$ columns}{\mbox{$\begin{matrix}
0&\cdots&0&1\\
\vdots&\iddots&\iddots&0\\
0& \iddots&\iddots&\vdots\\
1&0&\cdots&0
\end{matrix}$}}
\right)
\right\}\mbox{\scriptsize $m$ rows}.
\end{equation}
\endgroup
In the sequel, given square matrices $D_1, \ldots D_N$ we will denote by $D_1\oplus \ldots\oplus D_N$ the block diagonal matrix with diagonal blocks 
$D_1, \ldots, D_N$ in the order from the top left to the bottom right and all off-diagonal block equal to zero.

For $\lambda\in\mathbb C$, we define the $k\times k$ or $2k\times 2k$  matrix $M_{\lambda,k}$ by
\begin{align}\label{antilinear op canon form}
M_{\lambda,k}:= 
\begin{cases}
J_{\lambda, k}&\mbox{ if } \lambda\in\R\\
\left(
\begin{array}{cc}
0&  J_{\lambda^2,k}\\
I& 0
\end{array}
\right)
&\mbox{ otherwise},
\end{cases}
\end{align}
where  $0$ denotes a matrix of appropriate size with zero in all entries and $I$ denotes the identity matrix. We define corresponding matrices $N_{\lambda, k}$ by
\begin{equation}
\label{N_def}
N_{\lambda,k}:= 
\begin{cases}
S_{k}&\mbox{ if } \lambda\in\R\\ 
S_{2k}&\mbox{ otherwise}.
\end{cases}
\end{equation}
For the $\ell$-self-adjoint antilinear operator $\boldsymbol{A}$ referred to in the following theorem, let us enumerate the eigenvalues of $\boldsymbol{A}^2$ (counting them with multiplicity) that are contained in the upper-half plane $\{z\in\mathbb C\,|\, \Re(\Z)\geq0\}$ of $\mathbb C$, labeling them as $\lambda_{1}^2,\ldots,\lambda_\gamma^2$. Furthermore, we take each $\lambda_i$ to be the principle square root of $\lambda_i^2$.
\begin{theorem}[immediate consequence of the main result in \cite{sykes2020canonical}]\label{simultaneous canonical form theorem}
Given a nondegenerate Hermitian form $\ell$ on a vector space $V$ and an $\ell$-self-adjoing antilinear operator $\boldsymbol{A}$, there exists a basis of $V$ with respect to which $\ell$ and $\boldsymbol{A}$ are respectively represented by the matrices $H_\ell$ and $A$ given by 
\begin{align}\label{simultaneous canonical form theorem eqn}
H_\ell=\bigoplus_{i=1}^\gamma \epsilon_iN_{\lambda_i,m_i}\quad\mbox{ and }\quad A=\bigoplus_{i=1}^\gamma M_{\lambda_i,m_i},
\end{align}
for some sequence $\epsilon_1,\ldots, \epsilon_\gamma$ satisfying $\epsilon_i=\pm1$ and some sequence of positive integers $m_1,\ldots, m_\gamma$.
\end{theorem}

Letting $H_\ell$ and $A$ be matrices representing $\ell$ and $\boldsymbol{A}$ respectively in some basis of $\mathfrak{g}_{-1}$, we consider the Lie algebras of square matrices $\alpha$ satisfying
\begin{align}\label{firstalgebra}
\alpha AH_\ell ^{-1} +   AH_\ell ^{-1}\alpha^T  =\eta AH_\ell ^{-1} \quad\mbox{ for some }\eta\in\mathbb C
\end{align}
and respectively
\begin{align}\label{secondalgebra}
  \alpha^TH_\ell \overline{A}+H_\ell \overline{A} \alpha=\eta H_\ell \overline{A}\quad\mbox{ for some }\eta\in\mathbb C,
\end{align}
and define the algebra $\mathscr{A}$ to be their intersection, that is,
\begin{align}\label{intersection algebra}
\mathscr{A}:=\left\{\alpha\,\left| \,\parbox{7.4cm}{$ \alpha AH_\ell ^{-1} +   AH_\ell ^{-1}\alpha^T  =\eta AH_\ell ^{-1}$ and \\$ \alpha^TH_\ell \overline{A}+H_\ell \overline{A} \alpha=\eta^\prime H_\ell \overline{A}$  for some $\eta,\eta^\prime\in \mathbb C$}\right.\right\}.
\end{align}




Let us fix a splitting of $\mathfrak{g}^{\mathrm{red}}_0$ as given in \eqref{gZeroRed splitting}. With respect to the basis of $\mathfrak{g}_{-1}$ fixed above, there exists some $(n-1)\times (n-1)$ matrix $\Omega$ such that $\mathfrak{g}_{0,+}^{\mathrm{red}}$ and $\mathfrak{g}_{0,-}^{\mathrm{red}}$ have the matrix representations
\begin{equation}\label{Xzero2}
\mathfrak{g}_{0,+}^{\mathrm{red}}=\text{span}_{\mathbb C}
\left\{
\left(
\begin{array}{cc}
 \Omega& A\\
 0& -H_\ell^{-1}\Omega^TH_\ell
\end{array}
\right)
\right\}
\quad\mbox{ and }\quad
\mathfrak{g}_{0,-}^{\mathrm{red}}=\text{span}_{\mathbb C}
\left\{
\left(
\begin{array}{cc}
-\overline{H_\ell}^{-1}\Omega^* \overline{H_\ell}&  0\\
\overline{A}&  {\overline{\Omega}}
\end{array}
\right)
\right\}.
\end{equation}
In \cite{SZ2020}, we show that $\mathfrak{g}^{\mathrm{red}}_0$ is a subalgebra of $\mathfrak{csp}(\mathfrak{g}_{-1})$ and establish the following lemma.

\begin{lemma}[{\cite[Proposition 5.4]{SZ2020}}]\label{system of conditions for homogeneity lemma}
There exists a subalgebra $\mathscr{A}_0$ of $\mathscr{A}$ invariant under the transformation $\alpha\mapsto \overline{H_\ell}^{-1}\alpha^*\overline{H_\ell}$ such that 
\begin{align}\label{g00 subalg. representation}
\mathfrak{g}_{0,0}^{\mathrm{red}}=
\left\{\left.
\left(
\begin{array}{cc}
\alpha & 0\\
0 & -H_\ell^{-1} \alpha^T H_\ell
\end{array}
\right)+cI\, \right|\, \alpha\in \mathscr{A}_0, \mbox{ and } c\in \mathbb C
\right\}, 
\end{align}
and there exist coefficients $\{\eta_\alpha\}_{\alpha\in\mathscr{A}_0}\subset\mathbb C$ and $\mu\in\mathbb C$ such that the system of relations
\begin{align}\label{system}
\left.\mbox{
\begin{minipage}{.56\textwidth} 
\textit{i)}\quad\quad$\displaystyle \,\,\,\,\alpha AH_\ell^{-1} +   AH_\ell^{-1}\alpha^T  =\eta_\alpha
 AH_\ell^{-1}$\\
\textit{ii)}\quad\quad$\displaystyle  [\alpha,\Omega]-\eta_\alpha\Omega\in\mathscr{A}_0 $\\
\textit{iii)}\quad\quad$\displaystyle  \,\, \, \Omega^TH_\ell\overline{A}+H_\ell\overline{A} \Omega=\mu H_\ell\overline{A} $\\
\textit{iv)}\quad\quad$\displaystyle \left[\overline{H_\ell}^{-1}\Omega^* \overline{H_\ell},\Omega \right]+A\overline{A}-\overline{\mu}\Omega -\mu \overline{H_\ell}^{-1}\Omega^* \overline{H_\ell} \in\mathscr{A}_0 $\\
\end{minipage}
}
\right\}
\end{align}
holds for all $ \alpha\in\mathscr{A}_0$.
\end{lemma}

We have the following basic lemma.
\begin{lemma}[{\cite[Proposition 3.6]{SZ2020}}]\label{omega in A implies regularity}
The following are equivalent. 
\begin{enumerate}
    \item $\mathfrak{g}^0$ is regular.
    \item $A\overline{A}A $ is a scalar multiple of $A$.
\end{enumerate}
Moreover, if $\Omega$ is in $\mathscr{A}$ then $\mathfrak{g}^0$ is regular.
\end{lemma}
\begin{proof}
Equivalence of (1) and (2) follows from \eqref{cube}. 
The latter statement is also shown in \cite{SZ2020}, although we prove it more directly here because it is not given as a numbered result there. For this, let $v_{+}$ and $v_{-}$ be elements in $\mathfrak{g}_{0,2}^{\mathrm{red}}$ and $ \mathfrak{g}_{0,-2}^{\mathrm{red}}$ respectively. Note that if $\Omega$ is in $\mathscr{A}$ then there exist vectors $w_+,w_{-}\in \mathfrak{g}_{0,0}$ such that $v_{\pm}+w_{\pm}$ belongs to $\mathfrak{g}_{0,\pm 2}$. Accordingly,
\[
[v_{+}+w_{+},v_{-}+w_{-}]=[w_{-},w_{+}]+[v_{+}+w_{+},w_{-}]+[w_{+},v_{-}+w_{-}]+[v_{+},v_{-}].
\]
Since 
\begin{align}\label{g00 def consequence}
[\mathfrak{g}_{0,0},\mathfrak{g}_{0}]\subset \mathfrak{g}_{0}
\end{align}
by the definition of $\mathfrak{g}_{0,0}$, the first three terms in the right side of this last equation belong to $\mathfrak{g}_{0}$. Since $\mathfrak{g}_{0}^{\mathrm{red}}$ is closed under Lie brackets, $[v_{+},v_{-}]$ belongs to $\mathfrak{g}_{0}^{\mathrm{red}}$. Hence if $\Omega$ is in $\mathscr{A}$ then $[\mathfrak{g}_{0,2},\mathfrak{g}_{0,-2}]\subset \mathfrak{g}_{0}+\mathfrak{g}_{0}^{\mathrm{red}}$. On the other hand if $\Omega$ is in $\mathscr{A}$ then $\mathfrak{g}_{0}^{\mathrm{red}}\subset \mathfrak{g}_{0}$. Therefore, if $\Omega$ is in $\mathscr{A}$ then $[\mathfrak{g}_{0,2},\mathfrak{g}_{0,-2}]\subset \mathfrak{g}_{0}$. Noting \eqref{g00 def consequence}, it follows that if $\Omega$ is in $\mathscr{A}$ then $[\mathfrak{g}_{0},\mathfrak{g}_{0}]\subset \mathfrak{g}_{0}$, that is, $\mathfrak{g}^0$ is regular.
\end{proof}

Now, for completeness, given a non-regular CR symbol $\mathfrak g^0$ encoded by the pair $(\ell, \boldsymbol{A})$, represented by the pair of matrices  $(H_\ell, A)$ in the canonical basis as in Theorem \ref{simultaneous canonical form theorem} we will give a more precise (i.e., in terms of integers $m_1, \ldots m_\gamma$ and numbers $\lambda_1, \ldots,\lambda_\gamma$) upper bound  for the dimension of the algebra of infinitesimal symmetries of a $2$-nondegenerate $(2n+1)$-dimensional CR structure of hypersurface type with $1$-dimensional Levi kernel admitting a constant reduced modified symbol corresponding to CR symbol $\mathfrak g^0$. For this, for every $1\leq i, j\leq \gamma$,
let 
\begin{align}
\label{di,j}
d(i,j)=\begin{cases}0, &  (\lambda_i\neq \lambda_j) \mbox{ or }  (i=j \mbox{ and } \lambda_i^2 \mbox{ is  not a nonpositive real number)}\\
\min\{m_i, m_j\} & (i\neq j \mbox { and } \lambda_i=\lambda_j>0) \mbox { or } (i=j \mbox{ and } \lambda_i^2<0)\\
2\min\{m_i, m_j\} & i\neq j, \lambda_i=\lambda_j\mbox{ and } (\lambda_i^2\notin \mathbb R \mbox{ or } \lambda_i=0) \\
4\min\{m_i, m_j\} & i\neq j, \lambda_i=\lambda_j\mbox{ and } \lambda_i^2<0\\
\left\lceil \frac{m_i}{2}\right\rceil & i=j \mbox { and } \lambda_i=0
\end{cases}	
\end{align}
where $\lceil \tfrac{m}{2}\rceil$ denotes the ceiling function, i.e. the smallest integer not less than $\tfrac{m_i}{2}$.

Let 
\begin{equation} 
\label{dtotal}
d_{\mathrm{total}}:=\sum_{i\leq j} d(i,j).
\end{equation}

Then the following theorem is the direct consequence of item (1) of Theorem \ref{first prolongation for non-regular structures is zero} and Lemmas \ref{equal eValue off diag dimension}, \ref{Bij formula with zero lambda},  Corollary \ref{Bii formula}, and Lemma \ref{condition on C for 2-dimensional scaling component}: 
\begin{theorem}
\label{more precise bound}
Given a non-regular CR symbol $\mathfrak g^0$ encoded by the pair $(\ell, \boldsymbol{A})$ in the canonical basis as in Theorem \ref{simultaneous canonical form theorem}, the dimension of the algebra of infinitesimal symmetries of a $2$-nondegenerate $(2n+1)$-dimensional CR structure of hypersurface type with 1-dimensional Levi kernel admitting a constant reduced modified symbol corresponding to CR symbol $\mathfrak g^0$ is not greater than 
$d_{\mathrm{total}}+2n+3$ if the operator $\boldsymbol{A}$ is not nilpotent, and it is not greater than $d_{\mathrm{total}}+2n+4$, if the operator $\boldsymbol{A}$ is nilpotent. 
\end{theorem}

Note that the mentioned Lemmas and Corollaries from the appendix (Section \ref{Matrix representations of intersection algebra general formula}) together with \eqref{g00 subalg. representation} imply  that  $\dim \mathfrak g_{0,0}^{\mbox{red}}$ is either not greater than  $d_{\mbox{total}}+2$ or $d_{\mbox{total}}+3$ depending whether or not $A$ is nilpotent. The estimate  for $\mathfrak u(\mathfrak g^{0, \mbox{red}})= \mathfrak g^{0, \mbox{red}}$ in Theorem \ref{more precise bound} follows from this and the fact that $\dim (\mathfrak g_-+ \mathfrak g_{0, -2}+\mathfrak g_{0, 2}) =2n+1$.

\section{Proof of Theorem \ref{first prolongation for non-regular structures is zero}}\label{First prolongation of non-regular symbols}

Here we carry out the final steps in the proof of Theorem \ref{first prolongation for non-regular structures is zero}. The analysis relies on formulas derived in the appendix (section \ref{Matrix representations of intersection algebra general formula}) below, and we have deferred deriving these formulas to the appendix because it requires somewhat digressive linear algebra ancillary to this paper's main results.

\subsection{Preparatory lemmas and notations}
\label{preparsec}
Let $\sigma: \mathfrak{g}^{0,\mathrm{red}} \to  \mathfrak{g}^{0,\mathrm{red}}$ denote the antilinear involution induced by the natural complex conjugation of $\mathbb C T M$. We introduce this $\sigma$ notation to avoid confusion because while working with matrix representations in coordinates we will use the overline notation to denote the standard complex conjugation of coordinates, which is a different involution. 
Let  
\begin{align}\label{fixed basis of V}
(e_1,\ldots, e_{2n-2})
\end{align}
be a basis of $\mathfrak{g}_{-1}$ with respect to which we get the matrix representation of $\mathfrak{g}_0^{\mathrm{red}}$ given by \eqref{Xzero2} and \eqref{g00 subalg. representation}. Notice in particular that $(e_1,\ldots, e_{n-1})$ spans $\mathfrak{g}_{-1,1}$ and 
\[
\sigma(e_i)=e_{n+i-1}\quad\quad\forall\, 1\leq i\leq n-1.
\]
Note that $\sigma$ extends to an involution defined of $\mathfrak{g}^{\mathrm{red}}_1$ by same formula (see \eqref{conj}) that we used to extend the natural conjugation from $\mathfrak{g}_{-}$ to be defined on $\mathfrak{csp}(\mathfrak{g}_{-1})$, that is
\begin{align}\label{conj on g1}
\sigma(\varphi)(v):=\sigma\circ\varphi\circ\sigma(v)
\quad\quad \forall\, v\in\mathfrak{g}^{0,\mathrm{red}},\,\varphi\in\mathfrak{g}^{\mathrm{red}}_1
\end{align}
defines an involution of $\mathfrak{g}^{\mathrm{red}}_1$.

An element $\varphi$ in $\Hom(\mathfrak{g}_{-2},\mathfrak{g}_{-1})\oplus \Hom(\mathfrak{g}_{-1},\mathfrak{g}_{0}^{\mathrm{red}})$ belongs to $\mathfrak{g}_{1}^{\mathrm{red}}$ if and only if 
\begin{equation}
\label{symmetry of tensor property}
\varphi([e_i,e_j])=\big(\varphi(e_i)\big)(e_j)-\big(\varphi(e_j)\big)(e_i) \quad \forall \, i, j\in \{1,\ldots, 2n-2\}.
\end{equation}
Note, here $\phi(e_i)\in \mathfrak{g}_0^{\mathrm{red}}\subset \mathfrak{csp} (\mathfrak g_{-1})$. 

Given any element $v\in \mathfrak g_{-1}$ let $v_-$ and $v_+$ be the canonical projections of $v$ to $\mathfrak g_{-1,-1}$ and $\mathfrak g_{-1,1}$, respectively, with respect to the splitting $\mathfrak g_{-1}=\mathfrak g_{-1,-1}\oplus \mathfrak g_{-1,1}$. 

As a direct consequence of  \eqref{symmetry of tensor property} and \eqref{Xzero2}, if $n\leq j\leq 2n-2$  and $1\leq i\leq n-1$, then 
\begin{align}
\label{matrix rep of larger index rows of tensor phi lemma prereq}
~&\big(\big(\varphi(e_j)\big)e_i)_+\in
\mathrm{span}\{\boldsymbol{A} e_{j-n+1}\}-\big(\varphi([e_i, e_j])\big)_{+} \subset 
\mathrm{span}\{\boldsymbol{A} e_{j-n+1}, \big(\varphi(1)\big)_+\},\\
~&\big(\big(\varphi(e_i)\big)e_j)_-\in
\mathrm{span}\left\{\sigma\big( \boldsymbol{A} e_{i}\big)\right\}- \big(\varphi([e_i, e_j])\big)_{-}\subset
\mathrm{span}\{\boldsymbol{A} e_{i}, \big(\varphi(1)\big)_-\}
\end{align}
In particular, the upper left $(n-1)\times(n-1)$ block in the matrix $\varphi(e_j)$ and  the lower right $(n-1)\times(n-1)$ block in the matrix $\varphi(e_i)$ both have rank at most 2.

Also from \eqref{symmetry of tensor property} and the fact that $[e_i, e_j]=0$ for $n\leq i, j\leq 2n-2$, we immediately have that
\begin{equation}
\label{symmetry of tensor property lower right}
 \varphi (e_i) e_j=\varphi(e_j) e_i, \quad  n\leq i, j\leq 2n-2. 
\end{equation}
\begin{lemma}\label{matrix rep of larger index rows of tensor phi lemma}
If the antilinear operator  $\boldsymbol{A}$ (or, equivalently the matrix $A$) has rank greater than 1 and $i\geq n$ then $\varphi(e_i)\in \mathfrak{g}_{0,0}^{\mathrm{red}}\oplus \mathfrak{g}_{0,-}^{\mathrm{red}}$, or, equivalently,
\begin{align}\label{matrix rep of larger index rows of tensor phi}
\varphi(e_i)=
\left(
\begin{array}{cc}
\alpha_i &0\\
c\overline{A}&-H_\ell^{-1}\alpha_i^T H_\ell
\end{array}
\right)\quad\quad\mbox{for some $c\in \mathbb C$ and }\alpha_i\in\mathscr{A}_0+\mathbb C(\overline{H_\ell}^{-1}\Omega^*\overline{H_\ell}).
\end{align}
\end{lemma}
\begin{proof}
By \eqref{Xzero2}, there exists $c\in \mathbb C$ such that for every  $n\leq j\leq 2n-2$ 
\begin{equation}
\label{inter1}
\big(\big(\varphi(e_i)\big)e_j\big)_+= c\boldsymbol{A} e_{j-n+1}
\quad\mbox{ and }\quad
\big(\big(\varphi(e_j)\big)e_i\big)_+\in \mathrm{span} \{\boldsymbol{A} e_{i-n+1}\}.    
\end{equation}
By \eqref{symmetry of tensor property lower right}, for all $n\leq j\leq 2n-2$,
\[
c \boldsymbol{A} e_{j-n+1}\in  \mathrm{span} \{\boldsymbol{A} e_{i-n+1}\}.
\]
This implies that $c=0$, because otherwise  $\mathrm{rank} \,\boldsymbol{A} \leq 1$, contradicting our assumption. Therefore, $\big(\varphi(e_i) v\big)_+=0$ for all $v\in \mathfrak g_{-1,-1}$, which is equivalent to the statement of the lemma.
\end{proof}

Similarly, we have the following Lemma.
\begin{lemma}\label{matrix rep of smaller index rows of tensor phi lemma}
If the antilinear operator  $\boldsymbol{A}$ (or, equivalently the matrix $A$) has rank greater than 1 and $i< n$ then $\varphi(e_i)\in \mathfrak{g}_{0,0}^{\mathrm{red}}\oplus \mathfrak{g}_{0,+}^{\mathrm{red}}$ or, equivalently,
\begin{align}\label{matrix rep of smaller index rows of tensor phi}
\varphi(e_i)=
\left(
\begin{array}{cc}
\alpha_i &cA\\
0&-H_\ell^{-1}\alpha_i^T H_\ell
\end{array}
\right)\quad\quad\mbox{for some $c\in \mathbb C$ and }\alpha_i\in\mathscr{A}_0+\mathbb C \Omega.
\end{align}
\end{lemma}

\begin{lemma}
If $A$ has rank greater than 1 and $\alpha_i$ is the matrix defined by \eqref{matrix rep of larger index rows of tensor phi} and \eqref{matrix rep of smaller index rows of tensor phi} then, for $i<n$, we have
\begin{align}\label{orthogonal algebra 1 condition}
\left(H_\ell\overline{A}\alpha_i\right)^T+H_\ell\overline{A}\alpha_i=\eta H_\ell\overline{A} \mbox{ for some }\eta\in \mathbb C
\end{align}
and, for $n\leq i$, we have
\begin{align}\label{orthogonal algebra 2 condition}
\alpha_i AH_\ell^{-1}+\left(\alpha_iAH_\ell^{-1}\right)^T=\eta AH_\ell^{-1}\mbox{ for some }\eta\in \mathbb C.
\end{align}
\end{lemma}
\begin{proof}
If $\alpha_i$ is as in \eqref{matrix rep of smaller index rows of tensor phi} then $\alpha_i\in \mathscr{A}+ \mathbb C\Omega$, so the definition of $\mathscr{A}$ and 
item (iii) of \eqref{system} imply \eqref{orthogonal algebra 1 condition}. If, on the other hand, $\alpha_i$ is as in \eqref{matrix rep of larger index rows of tensor phi} then $\alpha_i\in \mathscr{A}+ \mathbb C(\overline{H_\ell}^{-1}\Omega^*\overline{H_\ell})$, so the definition of $\mathscr{A}$ and item (iii) of \eqref{system}
imply \eqref{orthogonal algebra 2 condition}.
\end{proof}

\begin{corollary}\label{non-regular implication for alphai=0}
If the CR symbol is not regular and the matrix $\alpha_i$ given in \eqref{matrix rep of larger index rows of tensor phi} or \eqref{matrix rep of smaller index rows of tensor phi} is zero, then $\varphi(e_i)=0$.
\end{corollary}
\begin{proof}
Suppose $\alpha_i=0$. By \eqref{Xzero2}, \eqref{g00 subalg. representation}, and Lemmas \ref{matrix rep of larger index rows of tensor phi lemma} and \ref{matrix rep of smaller index rows of tensor phi lemma}, if $\varphi(e_i)\neq 0$ then either $\Omega\in \mathscr{A}$ or $\overline{H_\ell}^{-1}\Omega^*\overline{H_\ell}\in \mathscr{A}$. The conditions $\Omega\in \mathscr{A}$ and $\overline{H_\ell}^{-1}\Omega^*\overline{H_\ell}\in \mathscr{A}$ are, however, equivalent, so either $\varphi(e_i)\neq 0$ or $\Omega\in \mathscr{A}$. If the CR symbol is not regular then, by Lemma \ref{omega in A implies regularity}, $\Omega\not\in\mathscr{A}$, and hence $\varphi(e_i)=0$.
\end{proof}
\begin{lemma}\label{case 1 main result}
If an element $\varphi$ in $ \mathfrak{g}_{1}^{\mathrm{red}}$ satisfies $\varphi(1)=0$ and \begin{align}\label{case 1 varphi(e_i) are zero for large i}
 \varphi(e_i)=0\quad\quad\forall \, i\geq n
\end{align}
then 
\begin{align}\label{case 1 varphi(e_i) are zero for small i}
 \varphi(e_i)=0\quad\quad\forall \, i< n,
\end{align}
and so $\varphi=0$.
\end{lemma}
\begin{proof}
Since $\varphi(1)=0$, the left side of \eqref{symmetry of tensor property} is zero for all $i$ and $j$. Accordingly, for any $i\in\{1,\ldots,n-1\}$ and $j\in\{n,\ldots, 2n-2\}$, \eqref{symmetry of tensor property} and \eqref{case 1 varphi(e_i) are zero for large i} imply that the $j$ column of $\varphi(e_i)$ is zero. Hence, for all $i\in\{1,\ldots,n-1\}$,  the latter $n-1$ columns of $\varphi(e_i)$ are all zero. From this and Lemma \ref{matrix rep of smaller index rows of tensor phi lemma} (and specifically \eqref{matrix rep of smaller index rows of tensor phi}), it follows that  $H_\ell^{-1}\alpha_i^T H_\ell=0$. Hence $\alpha_i=0$ and therefore by \eqref{matrix rep of smaller index rows of tensor phi} again  
\eqref{case 1 varphi(e_i) are zero for small i} holds.
\end{proof}

The \emph{general strategy} of our proof of item (1) of Theorem \ref{first prolongation for non-regular structures is zero} is, for a given arbitrary $\varphi \in \mathfrak{g}_{1}^{\mathrm{red}}$, first to prove that $\varphi(1)=0$ and then to prove \eqref{case 1 varphi(e_i) are zero for large i}.

We will also need the following equations and notation.
In the sequel every $(n-1)\times (n-1)$ matrix $X$  will be also regarded as an operator having the matrix representation $X$  with respect to the basis $(e_1,\ldots, e_{n-1})$. Let $\{\varphi_i\}_{i=1}^{2n-2}\subset\mathbb C$ denote the coefficients satisfying 
\begin{equation}
\label{varphi_i}
\varphi(1)=\sum_{i=1}^{2n-2} \varphi_i e_i.
\end{equation} 

By \eqref{matrix rep of larger index rows of tensor phi}, it follows that   
\begin{equation}
\label{ij>n}
\big(\varphi(e_i)) e_j\big)_-= -\left(H^{-1}_\ell\alpha_i^TH_\ell\right)e_{j-n+1}, \quad \forall\,n\leq i, j\leq 2n-2.
\end{equation}
This together with \eqref{symmetry of tensor property lower right} yields 
\begin{equation}
\label{HAH>n}
\left(H^{-1}_\ell\alpha_i^TH_\ell\right)e_{j-n+1}=\left(H^{-1}_\ell\alpha_j^TH_\ell\right)e_{i-n+1}, \quad \forall\, n\leq i, j\leq 2n-2.
\end{equation}

Condition \eqref{HAH>n} is crucial in the subsequent analysis, namely in the proof of Lemmas \ref{case 1 lemma} and \ref{case 2 lemma}. Therefore, we need to describe  the matrix $H^{-1}_\ell\alpha_j^TH_\ell$, which we begin by first describing the matrix $\alpha_j$.  By \eqref{matrix rep of larger index rows of tensor phi}, it follows that,  for $n\leq j\leq 2n-2$  and $1\leq i\leq n-1$,
\[
\big(\varphi(e_j) e_i\big)_+=\alpha_j e_i.
\]
From this and  \eqref{matrix rep of larger index rows of tensor phi lemma prereq}, taking into account that the matrix $A$ represents the antilinear operator $\boldsymbol{A}$, we have that there exists the unique tuple  $(\kappa_{i})_{i=1}^{n-1}$ such that 
\begin{equation}
\label{alpha_through_C_and _H}
\alpha_j e_i= \kappa_i A e_{j-n+1}-(H_\ell)_{i,j-n+1}  \big(\varphi(1)\big)_{+}
\end{equation}
for all $1\leq i\leq n-1$ and $n\leq j\leq 2n-2$.
The uniqueness of $(\kappa_{i})_{i=1}^{n-1}$ follows from the assumption that $A\neq 0$ and that  $\kappa_i$ in \eqref{alpha_through_C_and _H} is independent of $j$.

\subsection{The first special case: }\label{case 1 subsection} In this subsection, \ref{case 1 subsection}, we consider the special case wherein, for some integer $m$ satisfying $2\leq m\leq n-1$, we have 
\begin{align}\label{case 1 condition a}
H_\ell=S_m\oplus H_\ell^\prime
\end{align}
where $H_\ell^\prime$ is an arbitrary nondegenerate Hermitian matrix, and
\begin{align}\label{case 1 condition b} 
A=J_{\lambda,m}\oplus A^\prime
\quad\quad\mbox{ for some $\lambda\geq 0$},
\end{align}
where $A^\prime$ is such that $(\ell,\boldsymbol{A})$ is represented by $(H_\ell, A)$. Moreover, we assume that $(H_\ell, A)$ is in the canonical form of Theorem \ref{simultaneous canonical form theorem}. In particular, 
\begin{equation}
\label{C_basis}
A e_1=\lambda e_1, \quad A e_i=\lambda e_i+e_{i-1}
\quad\quad \forall\, 2\leq i\leq m,
\end{equation}
and
\begin{equation}
\label{H_basis}
H_\ell e_i=e_{m+1-i}
\quad\quad \forall\, 1\leq i\leq m.
\end{equation}

Using \eqref{case 1 condition a} and \eqref{C_basis} we obtain 
\begin{align}\label{case 1 alpha_n}
\alpha_n e_i=\kappa_i\lambda e_1-\delta_{i, m} (\varphi(1))_+
\quad \forall \, i\in \{1,\ldots, n-1\},
\end{align}
and, for $0<p< m$,
\begin{align}\label{case 1 alpha_n+p}
\alpha_{n+p} e_i=\kappa_i e_p+\kappa_i\lambda e_{p+1}-\delta_{i, m-p}(\varphi(1))_+ 
\quad \forall \, i\in \{1,\ldots, n-1\}.
\end{align}

Now from \eqref{case 1 alpha_n}, we get
\begin{equation}
\label{transpose_n}
\alpha_n^T e_1=\sum_{j=1}^{n-1} \kappa_j e_j-\varphi_1 e_m 
\quad\mbox{ and }\quad
\alpha_n^T e_i=-\varphi_i e_m 
\quad\quad\forall\, 2\leq i\leq n-1.    
\end{equation}
Using this together with \eqref{H_basis} we can get 
\begin{align}\label{H^(-1)alpha_n^T H}
(H_\ell^{-1}\alpha_n^T H_\ell)e_i=-\varphi_{m+1-i}e_1
\quad\quad\forall \, i\in\{1,\ldots, m-1\},
\end{align}
\begin{align}\label{H^(-1)alpha_n^T H second}
(H_\ell^{-1}\alpha_n^T H_\ell)e_m\equiv -\varphi_1e_1+\lambda\sum_{j=1}^{m}\kappa_{m+1-j} e_j \pmod{\mathrm{span}\{e_{m+1},e_{m+2},\ldots, e_{n-1}\}},
\end{align}
and
\begin{align}\label{H^(-1)alpha_n^T H third}
(H_\ell^{-1}\alpha_n^T H_\ell)e_i=-\left(\sum_{j=m+1}^{n-1} (H_\ell)_{j,i} \varphi_j\right)e_1= -\left(\sum_{j=1}^{n-1-m} (H'_\ell)_{j, i-m} \varphi_{j+m}\right)e_1\quad\forall \, i>m,
\end{align}
where $H_\ell$ is as in \eqref{case 1 condition a}. 

Similarly, for $0< p<m$, from \eqref{case 1 alpha_n+p} we have 
\begin{equation}
\label{transpose_n+p}
\alpha_{n+p}^T e_i=\begin{cases} -\varphi_i e_{m-p} , & i\in \{1, \ldots, n-1\}\setminus\{p, p+1\}\\
-\varphi_p e_{m-p}+
\displaystyle{\sum_{j=1}^{n-1} \kappa_j e_j},& i=p\\
-\varphi_{p+1} e_{m-p}+\lambda\displaystyle{\sum_{j=1}^{n-1}} \kappa_j e_j& i=p+1,  \end{cases}
\end{equation}
\begin{align}\label{H^(-1)alpha_n+p^T H}
(H_\ell^{-1}\alpha_{n+p}^T H_\ell)e_i=-\varphi_{m+1-i}e_{p+1}
\quad\quad\forall \, i\in\{1,\ldots, m\}\setminus \{m-p,m-p+1\},
\end{align}
\begin{align}\label{H^(-1)alpha_n+p^T H second}
(H_\ell^{-1}\alpha_{n+p}^T H_\ell)e_{m-p}\equiv -\varphi_{p+1}e_{p+1}+\lambda\sum_{j=1}^{m}\kappa_{m+1-j} e_j \pmod{\mathrm{span}\{e_{m+1},\ldots, e_{n-1}\}},
\end{align}
and
\begin{align}\label{H^(-1)alpha_n+p^T H third}
(H_\ell^{-1}\alpha_{n+p}^T H_\ell)e_{m-p+1}\equiv -\varphi_{p}e_{p+1}+\sum_{j=1}^{m}\kappa_{m+1-j} e_j \pmod{\mathrm{span}\{e_{m+1},\ldots, e_{n-1}\}}.
\end{align}
For $p\geq m$,
\begin{align}\label{H^(-1)alpha_n+p^T H fourth}
(H_\ell^{-1}\alpha_{n+p}^T H_\ell)e_i \in \mathrm{span}\{e_{m+1},\ldots, e_{n-1}\}.
\end{align}

\begin{lemma}\label{case 1 lemma}
In the special case of \ref{case 1 subsection} wherein \eqref{case 1 condition a} and \eqref{case 1 condition b} hold, if $\mathrm{rank}\,A>1$ then
\begin{equation}
\label{phi_1=0}
\varphi(1)=0.
\end{equation}
\end{lemma}
\begin{proof}
We will begin by showing that 
\begin{equation}
\label{phi_1+=0}
\left(\varphi(1)\right)_+=0.
\end{equation}
The proof consists of analysis of equation \eqref{HAH>n} in three cases:

{\bf 1.} \emph{ Equation \eqref{HAH>n} for $i=n$ and $j=n+p$ with $0\leq p<m-1$.} By \eqref{H^(-1)alpha_n^T H}
\begin{align}\label{case 1 lemma eq1}
\left(H^{-1}_\ell\alpha_{n}^TH_\ell\right)e_{p+1}=\varphi_{m-p}e_1 \quad\quad\forall\,\, 0\leq p< m-1,
\end{align}
and, by \eqref{H^(-1)alpha_n+p^T H},
\begin{align}\label{case 1 lemma eq2}
\left(H^{-1}_\ell\alpha_{n+p}^TH_\ell\right)e_{1}=\varphi_{m}e_{p+1}\quad\quad\forall\,\, 0\leq p<m-1.
\end{align}
Applying \eqref{case 1 lemma eq1} and \eqref{case 1 lemma eq2} to  \eqref{HAH>n} with $i=n$ and $j=n+p$ we get
\[
\varphi_{m-p}e_1=\varphi_{m}e_{p+1} \quad\quad\forall \,\,0\leq p<m-1.
\]
Therefore, using the  last equation for $1\leq p< m-1$ (as for p=0 this equation is a tautology), we get 
\begin{align}\label{case 1 coeff phi a}
\varphi_2=\cdots=\varphi_{m-1}=0,  
\end{align}
and also that $\varphi_m=0$ for $m>2$ (we will give another way to prove the latter identity including the case $m=2$ in item 3 of the proof below).

{\bf 2.}\emph{ Equation \eqref{HAH>n} for $i=n$ and $j=n+p$ with $p\geq m$.} By \eqref{H^(-1)alpha_n^T H third} we get that 
\begin{align}\label{case 1 lemma eq3} 
\left(H^{-1}_\ell\alpha_{n}^TH_\ell\right)e_{p+1}=\left(\sum_{j=1}^{n-1-m} (H'_\ell)_{j, p+1-m} \varphi_{j+m}\right)e_1.
\end{align}
Using \eqref{HAH>n}, from \eqref{case 1 lemma eq3} and \eqref{H^(-1)alpha_n+p^T H fourth} it follows that $\left(H^{-1}_\ell\alpha_{n}^TH_\ell\right)e_{p+1}=0$ or, equivalently, 
\[
\sum_{j=1}^{n-1-m} (H'_\ell)_{j, i} \varphi_{j+m}=0, \quad 1\leq i\leq n-1-m.
\]
Since the matrix $H_\ell^\prime$ is nonsigular, this yields
\begin{align}\label{case 1 coeff phi b}
\varphi_{m+1}=\cdots=\varphi_{n-1}=0.
\end{align} 

{\bf 3.}\emph{ Equation \eqref{HAH>n} for $i=n$ and $j=n+m-1$.}
If $v=\lambda\sum_{j=1}^{m}\kappa_{m+1-j} e_j$, then, by  \eqref{H^(-1)alpha_n^T H second},
\begin{align}\label{case 1 lemma eq5}
(H_\ell^{-1}\alpha_n^T H_\ell)e_m\equiv -\varphi_1e_1+v \pmod{\mathrm{span}\{e_i\}_{i=m+1}^{n-1}},
\end{align}
and, by \eqref{H^(-1)alpha_n+p^T H second},
\begin{align}\label{case 1 lemma eq6}
(H_\ell^{-1}\alpha_{n+m-1}^T H_\ell)e_{1}\equiv -\varphi_{m}e_{m}+v \pmod{\mathrm{span}\{e_i\}_{i=m+1}^{n-1}}.
\end{align}
Using  \eqref{HAH>n} again  and the fact that $m\geq 2$, from \eqref{case 1 lemma eq5} and \eqref{case 1 lemma eq6} it follows that $\varphi_1=0$ and $\varphi_{m}=0$. This completes the proof of \eqref{phi_1+=0}. 

Since \eqref{conj on g1} defines an involution of $\mathfrak{g}_1^{\mathrm{red}}$, $\sigma(\varphi)$ also belongs to $\mathfrak{g}_1^{\mathrm{red}}$, so, since $\varphi$ was an arbitrary element in $\mathfrak{g}_1^{\mathrm{red}}$, the exact same arguments applied above show that $\left(\sigma(\varphi)(1)\right)_+=0$. Since $\sigma(1)=1$,
\[
\sigma\left(\left(\varphi(1)\right)_{-}\right)=\left(\sigma\circ\varphi(1)\right)_+=\left(\sigma(\varphi)(1)\right)_+=0,
\]
and hence $\left(\varphi(1)\right)_{-}=0$, which, together with \eqref{phi_1+=0} implies \eqref{phi_1=0}.
\end{proof}

\begin{lemma}\label{case 1 lemma_new}
In the special case of \ref{case 1 subsection} wherein \eqref{case 1 condition a} and \eqref{case 1 condition b} hold, if $\mathrm{rank}\,A>2$ 
then  $(\kappa_1,\ldots, \kappa_n)A=0$.
\end{lemma}

\begin{proof} 
Consider now the equation in \eqref{orthogonal algebra 2 condition} with $i=n$. The matrix on the right side of  \eqref{orthogonal algebra 2 condition} is either zero or it  has rank equal to $\mathrm{rank}\,A$, which is at least 3 under this lemma's hypothesis. On the other hand, applying \eqref{C_basis}, \eqref{H_basis}, \eqref{case 1 alpha_n} and Lemma \ref{case 1 lemma}, we get
\begin{align}\label{general lambda eqn for case 1 lemma}
(\alpha_n A H_\ell^{-1})e_i\in \mathrm{span}\{e_1\}
\quad\quad \forall \, i\in \{1,\ldots, m-1\},
\end{align}
and, applying \eqref{case 1 alpha_n+p} additionally, if $\lambda=0$ then 
\begin{align}\label{lambda zero eqn for case 1 lemma}
(\alpha_{n+1} A H_\ell^{-1})e_i\in \mathrm{span}\{e_1\}
\quad\quad \forall \, i\in \{1,\ldots, m-1\}.
\end{align}
Hence, by \eqref{general lambda eqn for case 1 lemma},
\begin{equation}\label{rank aCH} 
\mathrm{rank}\left( \alpha_n A H_\ell^{-1}\right)\leq 1
\end{equation}
and $\mathrm{rank}\left( \alpha_n A H_\ell^{-1}+\left(\alpha_n A H_\ell^{-1}\right)^T\right)\leq 2$
because $\alpha_n A H_\ell^{-1}$ has at most one nonzero row. Similarly, if $\lambda=0$ then \eqref{lambda zero eqn for case 1 lemma}
\begin{equation}\label{rank aCH lambda zero} 
\mathrm{rank}\left( \alpha_{n+1} A H_\ell^{-1}\right)\leq 1
\end{equation}
and $\mathrm{rank}\left( \alpha_{n+1} A H_\ell^{-1}+\left(\alpha_{n+1} A H_\ell^{-1}\right)^T\right)\leq 2$. Since
the matrix on the left side of \eqref{orthogonal algebra 2 condition} has rank at most 2 whenever $i=n$ or $(\lambda,i)=(0,n+1)$, the matrix on the right side of  \eqref{orthogonal algebra 2 condition} is zero whenever $i=n$ or $(\lambda,i)=(0,n+1)$. Thus by \eqref{orthogonal algebra 2 condition}  the matrix  $\alpha_n A H_\ell^{-1}$ is skew symmetric, and the matrix $\alpha_{n+1} A H_\ell^{-1}$ is skew symmetric whenever $\lambda=0$. This together with \eqref{rank aCH} implies that 
\begin{align}\label{alpha C H}
\alpha_n A H_\ell^{-1}=0,
\end{align}
whereas applying \eqref{rank aCH lambda zero} yields 
\begin{align}\label{alpha C H lambda zero}
\alpha_{n+1} A H_\ell^{-1}=0,
\end{align} whenever $\lambda=0$. By \eqref{alpha C H} and \eqref{case 1 alpha_n} for $\lambda\neq 0$, or by \eqref{alpha C H lambda zero} and \eqref{case 1 alpha_n+p} for $\lambda=0$, 
we get that the vector $(\kappa_1,\ldots, \kappa_n)AH_\ell^{-1}=0$, which completes this proof.
\end{proof}

In the subsequent three Lemmas \ref{case 1 alphas with high index are zero}-\ref{case 1 alphas with high index are zero with 2 by 2 nilpotent blocks} we prove item (1) of Theorem \ref{first prolongation for non-regular structures is zero} in three special special cases  that together cover all non-regular CR symbols not treated in subsequent sections.

\begin{lemma}\label{case 1 alphas with high index are zero}
In the special case of \ref{case 1 subsection} wherein \eqref{case 1 condition a} and \eqref{case 1 condition b} hold,  if $\mathrm{rank}\,A>2$ and $(\lambda,m)\not\in\{(0,2),(0,3)\}$ then $\mathfrak{g}_{1}^{\mathrm{red}}=0$.
\end{lemma}
\begin{proof}
Let $\varphi\in \mathfrak{g}_{1}^{\mathrm{red}}$ and let $(\kappa_i)_{i=1}^{n-1}$ be as in \eqref{alpha_through_C_and _H}. 
It will suffice to show that $\kappa_i=0$ for every $1\leq i\leq n-1$. Indeed, first plugging this condition and the conclusion \eqref{phi_1=0} of Lemma  \ref{case 1 lemma}  into relation \eqref{alpha_through_C_and _H} we obtain that $\alpha_j=0$ for all $n\leq j\leq 2n-2$. This and  Corollary \ref{non-regular implication for alphai=0} imply \eqref{case 1 varphi(e_i) are zero for large i}. Thus,  the conclusion of the present lemma will follow from \eqref{phi_1=0} and Lemma \ref{case 1 main result}.

Notice that since $(\kappa_1,\ldots, \kappa_n)A=0$,
we have that $\kappa_i=0$ for $1\leq i\leq m$ if $\lambda\neq 0$, and $\kappa_i=0$ for $1\leq i\leq m-1$ if $\lambda=0$ . In particular, as $m\geq 2$ we have $\kappa_1=\kappa_2=0$ always, and, since it is assumed that $m>3$ when $\lambda=0$,  if $\lambda=0$ then $\kappa_3=0$ as well.

To produce a contradiction, assume that there exists an index $r$ such that $\kappa_r\neq 0$ and let $r$ be the minimal such index.
By \eqref{case 1 alpha_n},
\begin{align}\label{first columns of alpha are zero, two special cases}
\alpha_n e_i=\delta_{i, r}\kappa_i\lambda e_1
\quad \forall \, i\leq r,
\end{align}
and, by \eqref{case 1 alpha_n+p}, for $0<p< m$,
\begin{align}\label{first columns of alpha are zero, two special cases second}
\alpha_{n+p} e_i=\delta_{i,r}(\kappa_i e_p+\kappa_i\lambda e_{p+1})
\quad \forall \, i\leq r.
\end{align}

Note that, by Lemma \ref{matrix rep of larger index rows of tensor phi lemma}, $\mathrm{span} \{\alpha_{n}, \alpha_{n+1}\}$ is a $2$-dimensional subspace in $\mathscr{A}+\mathbb C (\overline{H}^{-1}_\ell\Omega^* \overline{H}_\ell)$. Since $\mathscr{A}$ is a subspace in $\mathscr{A}+\mathbb C (\overline{H}^{-1}_\ell\Omega^* \overline{H}_\ell)$ of codimension at most  $1$, the subspaces 
$\mathrm{span} \{\alpha_n, \alpha_{n+1}\}$ and $\mathscr{A}$ have a nontrivial intersection. That is,
there exist  $b_1,b_2\in\mathbb C$ such that  $(b_1,b_2)\neq(0,0)$ and 
\begin{equation}
\label{induction on r setup eqn}
b_1\alpha_n+b_2\alpha_{n+1}\in \mathscr{A}.
\end{equation}
By \eqref{first columns of alpha are zero, two special cases} and \eqref{first columns of alpha are zero, two special cases second} again the first $r-1$ columns of the matrix $b_1\alpha_n+b_2\alpha_n$ vanish and 
\begin{equation}
\label{induction on r Lemma eqn}
(b_1 \alpha_n+b_2 \alpha_{n+1})e_r=\kappa_r\Bigl ((\lambda b_1+b_2)e_1+\lambda b_2 e_2\Bigr) \end{equation}

By applying formulas from the appendix (i.e., Section \ref{Matrix representations of intersection algebra general formula}), we can derive a contradiction from the assumption $\lambda\neq 0$ as follows. Let $b_1\alpha_n+b_2\alpha_{n+1}$ be partitioned as a block matrix whose diagonal blocks have the same size as the diagonal blocks of $A$ (referring to the block diagonal partition of $A$ given in \eqref{simultaneous canonical form theorem eqn}).

By \eqref{induction on r setup eqn}, if $\lambda> 0$ then each $(i,j)$ block of $b_1\alpha_n+b_2\alpha_{n+1}$ is either characterized by Lemma \ref{lemma for reduction to eSpaces} or Corollary \ref{Bii formula}  and identically zero or it is characterized by Corollary \ref{Bij formula with nonzero lambda} and more specifically characterized by \eqref{intersection algebra lambda real off diag}.
In particular, if the $(1,j)$ block of $b_1\alpha_n+b_2\alpha_{n+1}$ is nonzero (and therefore characterized by \eqref{intersection algebra lambda real off diag}) and contains part of the $r$ column of $b_1\alpha_n+b_2\alpha_{n+1}$, then \eqref{intersection algebra lambda real off diag} implies that the $(j,1)$ block of $b_1\alpha_n+b_2\alpha_{n+1}$ is nonzero and contained in the first $r-1$ columns of $b_1\alpha_n+b_2\alpha_{n+1}$, which contradicts our definition of $r$. Accordingly, if $\lambda>0$ then the $(1,j)$ block of $b_1\alpha_n+b_2\alpha_{n+1}$ containing part of the $r$ column of $b_1\alpha_n+b_2\alpha_{n+1}$ is identically zero, which implies $\lambda b_1+b_2=0$ and $\lambda b_2=0$ by \eqref{induction on r Lemma eqn}. So, if $\lambda> 0$, then we obtain the contradiction $(b_1,b_2)=(0,0)$.

On the other hand, if $\lambda= 0$ then, by Lemma \ref{matrix rep of larger index rows of tensor phi lemma},
$\mathrm{span} \{\alpha_{n+2}, \alpha_{n+3}\}$ is a $2$-dimensional subspace in $\mathscr{A}+\mathbb C (\overline{H}^{-1}_\ell\Omega^* \overline{H}_\ell)$. Similarly to the previous case, $\mathscr{A}$ and $\mathrm{span} \{\alpha_{n+2}, \alpha_{n+3}\}$ have a nontrivial intersection, that is,
there exist  $b_1,b_2\in\mathbb C$ such that  $(b_1,b_2)\neq(0,0)$ and 
\begin{align}\label{induction on r setup lambda zero eqn}
b_1\alpha_{n+2}+b_2\alpha_{n+3}\in \mathscr{A}.
\end{align}
Note that we are now redefining $b_1$ and $b_2$ because the previous definition is no longer needed, and that the $b_i$s in \eqref{induction on r setup lambda zero eqn} are not related to the $b_i$s in \eqref{induction on r setup eqn}. By \eqref{first columns of alpha are zero, two special cases} and \eqref{first columns of alpha are zero, two special cases second} the first $r-1$ columns of the matrix $b_1\alpha_{n+2}+b_2\alpha_{n+3}$ vanish and 
\begin{equation}
\label{induction on r Lemma lambda zero eqn}
(b_1 \alpha_{n+2}+b_2 \alpha_{n+3})e_r=\kappa_r\Bigl(b_1e_2+b_2 e_3\Bigr) .
\end{equation}

By applying formulas from the appendix again, we can derive a contradiction now from the assumption $\lambda=0$. For this, let $b_1\alpha_{n+2}+b_2\alpha_{n+3}$ in \eqref{induction on r setup lambda zero eqn} be partitioned as a block matrix whose diagonal blocks have the same size as the diagonal blocks of $A$. By \eqref{induction on r setup lambda zero eqn}, if $\lambda=0$ then each $(i,j)$ block of $b_1\alpha_n+b_2\alpha_{n+1}$ is either characterized by Lemma \ref{lemma for reduction to eSpaces} and identically zero or it is characterized by Lemmas \ref{Bij formula with zero lambda} and \ref{condition on C for 2-dimensional scaling component} and Corollary \ref{Bii formula} and more specifically characterized by \eqref{intersection algebra lambda zero off diag 1}, \eqref{intersection algebra lambda zero off diag 2}, \eqref{intersection algebra lambda diag}, and \eqref{diagformula}. In particular, if $\lambda= 0$ and the $(1,j)$ block of $b_1\alpha_{n+2}+b_2\alpha_{n+3}$ contains part of the $r$ column of $b_1\alpha_{n+2}+b_2\alpha_{n+3}$, and, furthermore, we assume that the $(1,j)$ block is not identically zero, then this $(1,j)$ block is either characterized by \eqref{intersection algebra lambda diag} and \eqref{diagformula} or by \eqref{intersection algebra lambda zero off diag 1} and \eqref{intersection algebra lambda zero off diag 2}. 

Considering the first possibility where the $(1,j)$ block containing part of the $r$ column of $b_1\alpha_{n+2}+b_2\alpha_{n+3}$ is characterized by \eqref{intersection algebra lambda diag} and \eqref{diagformula} (i.e., $j=1$), by \eqref{induction on r Lemma lambda zero eqn}, the first $m$ entries of $b_1e_2+b_2 e_3$ form the $r$ column of the $(1,1)$ block of $b_1\alpha_{n+2}+b_2\alpha_{n+3}$. Since we are assuming that this $(1,1)$ block is a linear combination of matrices \eqref{intersection algebra lambda diag} and \eqref{diagformula} with the latter being a diagonal matrix, noting that $r>3$, it follows that the first entry in the $r-1$ column of this $(1,1)$ block is $-b_1$ and the second entry in the $r-1$ column of this $(1,1)$ block is $-b_2$. Yet the $r-1$ column of the $(1,1)$ block of $b_1\alpha_{n+2}+b_2\alpha_{n+3}$ is zero by the definition of $r$, so we have obtained the contradiction that $(b_1,b_2)=(0,0)$.

Considering the remaining possibility, which is where the $(1,j)$ block containing part of the $r$ column of $b_1\alpha_{n+2}+b_2\alpha_{n+3}$ is characterized by \eqref{intersection algebra lambda zero off diag 1} or \eqref{intersection algebra lambda zero off diag 2}, if this $(1,j)$ block is nonzero then \eqref{intersection algebra lambda zero off diag 1} and \eqref{intersection algebra lambda zero off diag 2} imply that the $(j,1)$ block is nonzero and contained in the first $r-1$ columns of $b_1\alpha_{n+2}+b_2\alpha_{n+3}$, which contradicts the definition of $r$.

Hence, the $(1,j)$ block containing part of the $r$ column of $b_1\alpha_{n+2}+b_2\alpha_{n+3}$ must be identically zero because all other possibilities yield contradictions, and yet, by \eqref{induction on r Lemma lambda zero eqn}, setting this $(1,j)$ block equal to zero again implies the contradiction $(b_1,b_2)=(0,0)$. Therefore, there is no index $r$ such that $\kappa_r\neq 0$.
\end{proof}

\begin{lemma}\label{case 1 alphas with high index are zero with small nilpotent blocks}
In the special case of \ref{case 1 subsection} wherein \eqref{case 1 condition a} and \eqref{case 1 condition b} hold,  if there is a basis with respect to which $\boldsymbol{A}$ is represented by the matrix
\begin{align}
A=J_{0, 3}\oplus J_{1, c}\oplus A''  \quad\mbox{ for some $c>0$} \label{case 1 alphas with high index are zero with small nilpotent blocks tilde C a}
\end{align}
or
\begin{align}
A=J_{0, 2}\oplus J_{1, c}\oplus J_{1, c'}\oplus  A''   \quad \mbox{ for some $c,c^\prime>0$.} \label{case 1 alphas with high index are zero with small nilpotent blocks tilde C b}
\end{align}
then $\mathfrak{g}_{1}^{\mathrm{red}}=0$.
\end{lemma}
\begin{proof}
Let $\varphi\in \mathfrak{g}_{1}^{\mathrm{red}}$ and let $(\kappa_i)_{i=1}^{n-1}$ be as in \eqref{alpha_through_C_and _H}. By the same arguments as in the beginning of the proof of Lemma \ref{case 1 alphas with high index are zero} , it will suffice to show that $\kappa_i=0$ for every $1\leq i\leq n-1$.
Note that, by Lemma \ref{case 1 lemma}, in the considered cases $\varphi(1)=0$. It is more convenient to work with matrices 
\begin{align}\widetilde A= J_{c,1}\oplus J_{0, 3}\oplus A'' \label{tC_1}
\end{align}
or
\begin{align}
\widetilde A=J_{ c,1}\oplus J_{ c',1}\oplus J_{0, 2}\oplus   A'' \label{tC_2} \end{align}
instead of $A$ in \eqref{case 1 alphas with high index are zero with small nilpotent blocks tilde C a} and \eqref{case 1 alphas with high index are zero with small nilpotent blocks tilde C b}, respectively.
This can be done by an obvious permutation of the basis. Also, in the considered cases the rank assumptions of Lemma \ref{case 1 lemma_new} with $A$ replaced by $\widetilde A$ holds. Therefore, using \eqref{alpha_through_C_and _H} with $A$ replaced by $\widetilde A$ we get
\begin{equation}\label{three_a}
\kappa_1=\kappa_2=\kappa_3=0.
\end{equation}
Note that if we would not replace $A$ by $\widetilde A$ we could  conclude that  $\kappa_1=\kappa_2=\kappa_4=0$ in the case of \eqref{case 1 alphas with high index are zero with small nilpotent blocks tilde C a} and that $\kappa_1=\kappa_3=\kappa_4=0$ in the case of \eqref{case 1 alphas with high index are zero with small nilpotent blocks tilde C b}, so that is why we make this permutation of the blocks.

Assume for a proof by contradiction that there exists  $r$ such that  $\kappa_r\neq 0$ and moreover that this is the minimal  such index, that is, $\kappa_i=0$ for all $i<r$.
By \eqref{three_a}, $r>3$. 
From \eqref{alpha_through_C_and _H} with $A$ replaced by $\widetilde{A}$ it follows that in both cases the first $r-1$ columns of the  matrices $\alpha_i$ with $n\leq i\leq n+3$ vanish, 
\begin{align}\label{span generators to intersect a}
\alpha_n e_r=\kappa_r c e_1, 
\quad\mbox{ and }\quad
\alpha_{n+3} e_r=\kappa_r e_3.
\end{align}
Further, 
\begin{align}\label{span generators to intersect b}
\alpha_{n+2} e_r=\kappa_r e_2
\end{align}
if $\widetilde{A}$ satisfies \eqref{tC_1}, and 
\begin{align}\label{span generators to intersect c}
\alpha_{n+1} e_r=\kappa_rc' e_2
\end{align}
if $\widetilde{A}$ satisfies \eqref{tC_2}.
Note that,  by Lemma \ref{matrix rep of larger index rows of tensor phi lemma}, each $\alpha_i$ in these equations belongs to $\mathscr{A}+\mathbb C\left(\overline{H_\ell}^{-1}\Omega^*\overline{H_\ell}\right)$. 

Hence, using similar arguments as in the proof of Lemma \ref{case 1 alphas with high index are zero} we get that the $3$-dimensional subspace $\mathrm{span} \{\alpha_n,\alpha_{n+2} , \alpha_{n+3}\}$ in the first case and $\mathrm{span} \{\alpha_n,\alpha_{n+1} , \alpha_{n+3}\}$
in the second case has at least a two dimensional intersection with $\mathscr {A}$. Notice further that in either case, the $r$th columns of matrices in these intersections must have a two-dimensional span because the natural map from the space $\mathrm{span} \{\alpha_n,\alpha_{n+2} , \alpha_{n+3}\}$ (or $\mathrm{span} \{\alpha_n,\alpha_{n+1} , \alpha_{n+3}\}$) to $\mathbb C^{n-1}$ sending a matrix to its $r$ column in this space is injective. 

Let us now first assume that $\widetilde{A}$ satisfies \eqref{tC_1}. Let $B^{(1)}$ and $B^{(2)}$ be matrices belonging to the intersection of $\mathrm{span} \{\alpha_n,\alpha_{n+2} , \alpha_{n+3}\}$ and $\mathscr{A}$ such that the $r$ column of $B^{(1)}$ is linearly independent from the $r$ column of $B^{(2)}$. For an $(n-1)\times (n-1)$ matrix $B$, let $(B_{(i,j)})$ be a partition of $B$ into a block matrix whose diagonal blocks have the same size as the diagonal blocks of $A$. Let $j$ be the index such that $B_{(1,j)}$ contains part of the $r$ column of $B$. By Lemma \ref{lemma for reduction to eSpaces}, since $c\neq 0$ there exists $i\in\{1,2\}$ such that $B_{(i,j)}=0$ for all $B\in\mathscr{A}$, because otherwise Lemma \ref{lemma for reduction to eSpaces} implies that the $(1,1)$ and $(2,2)$ blocks of $A\overline{A}$ have the same eigenvalues. In particular, at most one of the $(1,j)$ and $(2,j)$ blocks of any linear combination of $B^{(1)}$ and $B^{(2)}$ is nonzero. It follows that, for each $k\in\{1,2\}$, $B^{(k)}_{(1,j)}=0$ and $B^{(k)}_{(2,j)}\neq 0$ because otherwise the $r$ column of each $B^{(k)}$ belongs to $\mathrm{span}\{e_1\}$, which contradicts our choice of $B^{(1)}$ and $B^{(2)}$. Moreover, by \eqref{span generators to intersect a} and \eqref{span generators to intersect b}, the first nonzero column of each block $B^{(k)}_{(2,j)}$ has zero in all but its first two entries.

Each $B^{(k)}_{(2,j)}$ is either characterized by Lemma \ref{lemma for reduction to eSpaces} and is identically zero or characterized by Lemma \ref{Bij formula with zero lambda} and Corollary \ref{Bii formula} and more specifically characterized by \eqref{intersection algebra lambda zero off diag 1}, \eqref{intersection algebra lambda zero off diag 2}, or \eqref{intersection algebra lambda diag} (with $\lambda_i=0$). If $B^{(k)}_{(2,j)}$ is characterized by \eqref{intersection algebra lambda diag} then $j=2$ and, by \eqref{intersection algebra lambda diag}, the second entry of the first nonzero column of $B^{(k)}_{(2,2)}$ is zero. If, on the other hand, $B^{(k)}_{(2,j)}$ is characterized by \eqref{intersection algebra lambda zero off diag 1} (or \eqref{intersection algebra lambda zero off diag 2}) and the second  entry of the first nonzero column of $B^{(k)}_{(2,j)}$ is nonzero, then, by \eqref{intersection algebra lambda zero off diag 2} (or respectively \eqref{intersection algebra lambda zero off diag 1}), the $B^{(k)}_{(j,2)}$ block of $B^{(k)}$ is nonzero and contained in the first $r-1$ columns of $B^{(k)}$, which contradicts our choice of $r$. Therefore if $B^{(k)}_{(2,j)}$ is nonzero then the second entry of the first nonzero column of $B^{(k)}_{(2,j)}$ is zero. Yet this contradicts our choice of $B^{(1)}$ and $B^{(2)}$ because it means that the only nonzero entry in the $r$ column of $B^{(1)}$ and $B^{(2)}$ is the second entry.

Let us now address the remaining case, that is, assume that $\widetilde{A}$ satisfies \eqref{tC_2}. Again, let $j$ be the index such that $B_{(1,j)}$ contains part of the $r$ column a given $(n-1)\times(n-1)$ matrix $B$. Let $B^{(1)}$ and $B^{(2)}$ be matrices belonging to the intersection of $\mathrm{span} \{\alpha_n,\alpha_{n+1} , \alpha_{n+3}\}$ and $\mathscr{A}$
such that the $r$ column of $B^{(1)}$ is linearly independent from the $r$ column of $B^{(2)}$. From this independence condition and the fact that nonzero entries  of these respective $r$th columns of $B^{(1)}$ and of $B^{(2)}$  appear within their first three entries (the latter is a consequence of 
\eqref{span generators to intersect a} and 
\eqref{span generators to intersect c}), it follows that there exists a matrix $B$ in $\mathrm{span}\{B^{(1)},B^{(2)}\}$
such that there exists $i\in\{1,2\}$ with $B_{(i,j)}\neq 0$ (because otherwise, the third entry is the only nonzero entry of $r$th columns of $B^{(1)}$ and $B^{(2)}$, which contradicts the independence of these columns). Since $r>3$ it follows that $j>2$. Thus, it follows from Lemma \ref{lemma for reduction to eSpaces} and Corollary \ref{Bij formula with nonzero lambda} that this nonzero $B_{(i,j)}$ with $i\in\{1,2\}$ is characterized by \eqref{intersection algebra lambda real off diag}. Yet \eqref{intersection algebra lambda real off diag} implies that the $B_{(j,i)}$ is a nonzero block contained in the first $r-1$ rows of $B$, which contradicts our choice of $r$.
\end{proof}

 \begin{lemma}\label{case 1 alphas with high index are zero with 2 by 2 nilpotent blocks}
In the special case of \ref{case 1 subsection} wherein \eqref{case 1 condition a} and \eqref{case 1 condition b} hold,  if 
\begin{align}\label{case 1 alphas with high index are zero with all but one small nilpotent block}
A=   
\overbrace{J_{0,2}\oplus \cdots \oplus J_{0,2}}^{k\, \mbox{\small copies}}\oplus J_{c,1}
\oplus J_{0,1}\oplus \cdots \oplus J_{0,1},
\end{align}
for some integer $k$ and some $c>0$ then $\mathfrak{g}_{1}^{\mathrm{red}}=0$.
\end{lemma}
\begin{proof}
Let $\varphi\in \mathfrak{g}_{1}^{\mathrm{red}}$ and let $(\kappa_i)_{i=1}^{n-1}$ be as in \eqref{alpha_through_C_and _H}. By the same arguments as in the beginning of the proof of Lemma \ref{case 1 alphas with high index are zero} , it will suffice to show that $\kappa_i=0$ for every $1\leq i\leq n-1$.
We work with $(H_\ell, A)$ in the canonical form of Theorem \ref{simultaneous canonical form theorem}, so $H_\ell$ is as in \eqref{simultaneous canonical form theorem eqn}, that is
\begin{align}\label{special subcase H in case 1}
H_{\ell}=
\epsilon_{1}N_{0,2}\oplus \cdots \oplus \epsilon_{k}N_{0,2}\oplus \epsilon_{k+1}N_{c,1}\oplus \cdots \oplus \epsilon_{\gamma}N_{0,1}
\end{align}
for some coefficients $\epsilon_i=\pm 1$.

For a matrix $B$ in $\mathscr{A}$,  let $(B_{(i,j)})$ be a partition of $B$ into a block matrix whose diagonal blocks have the same size as the diagonal blocks of $A$. By Lemma \ref{Bij formula with zero lambda} and Corollary \ref{Bii formula} (in the appendix below), we have
\begin{align}\label{case 1.3 intersection eqn a}
B_{(i,j)}=
\epsilon_{i}\left(
\begin{array}{cc}
b & c\\
0&d
\end{array}
\right)
\quad\mbox{ and }\quad
B_{(j,i)}=-\epsilon_j
\left(
\begin{array}{cc}
b & e\\
0&d
\end{array}
\right)
\quad\quad\forall\, i,j\leq k
\end{align}
and
\begin{align}\label{case 1.3 intersection eqn b}
B_{(i,j)}=
\left(
\begin{array}{c}
a \\
0
\end{array}
\right)
\quad\mbox{ and }\quad
B_{(j,i)}=
\left(
\begin{array}{cc}
0 & b
\end{array}
\right)
\quad\quad\forall\, i\leq k<j
\end{align}
for some $b,c,  d,e\in \mathbb C$ that depend on $(i,j)$. By Corollary \ref{Bii formula} and Lemma \ref{condition on C for 2-dimensional scaling component} (in the appendix below),
\begin{align}\label{case 1.3 intersection eqn c}
B_{1,1}=B_{2,2}=\cdots=B_{2k+1,2k+1},
\end{align}
 
where here $B_{i,j}$ denotes the $(i,j)$ entry of $B$ rather than the $(i,j)$ block $B_{(i,j)}$.
By Lemma \ref{lemma for reduction to eSpaces} and Corollary \ref{Bii formula} (in the appendix below),
\begin{align}\label{case 1.3 intersection eqn d}
B_{(i,k+1)}=0
\quad\mbox{ and }\quad
B_{(k+1,i)}=0
\quad\quad\forall \, i\neq k.
\end{align}

Since, by Lemma \ref{case 1 lemma_new}, $(\kappa_1,\ldots, \kappa_{n-1})A=0$, we have 
\begin{align}\label{some alphas are zero case 1 final subcase}
\kappa_i=0\quad\quad\mbox{ whenever $i$ is odd and }i\leq 2k+1.
\end{align}
From \eqref{alpha_through_C_and _H} and Lemma \ref{case 1 lemma} it follows that, for $0\leq p\leq n-1$, the $i$ column of the matrix $\alpha_{n+p}$ is equal to $\kappa_i$ times the $p+1$ column of $A$. In particular, the $(i,j)$ entry of $\alpha_{n+2k}$ is
\begin{align}\label{case 1.3 alpha element eqn}
\left(\alpha_{n+2k}\right)_{i,j}=\kappa_j c\delta_{i,2k+1}.
\end{align}
Since, by Lemma \ref{matrix rep of larger index rows of tensor phi lemma}, each $\alpha_{n+p}$ belongs to $\mathscr{A}_0+\mathbb C\left(\overline{H_\ell}^{-1}\Omega^*\overline{H_\ell}\right)$ and $\alpha_{n+2k}$ does not belong to $\mathscr{A}_0\setminus\{0\}$, which can be seen by contrasting \eqref{case 1.3 intersection eqn d} and \eqref{case 1.3 alpha element eqn}, it follows that 
\[
\mbox{ either }\quad\alpha_{n+2k}=0\quad\mbox{ or }\quad\overline{H_\ell}^{-1}\Omega^*\overline{H_\ell}\in \mathscr{A}_0+\text{span}_{\mathbb C}\{\alpha_{n+2k}\}.
\]
But $\alpha_{n+2k}=0$ if and only if $\kappa_1=\cdots=\kappa_{n-1}=0$, which is equivalent to what we want to show, so let us proceed assuming 
\begin{align}\label{calculating omega case 1 final subcase}
\overline{H_\ell}^{-1}\Omega^*\overline{H_\ell}\in \mathscr{A}_0+\text{span}_{\mathbb C}\{\alpha_{n+2k}\}
\end{align}
in order to produce a contradiction. Accordingly, let $\Omega_0\in \mathscr{A}_0$ and $s\in\mathbb C$ be such that 
\begin{align}\label{case 1.3 omega decomp a}
\overline{H_\ell}^{-1}\Omega^*\overline{H_\ell}=\overline{H_\ell}^{-1}\Omega_0^*\overline{H_\ell}+s\alpha_{n+2k},
\end{align}
or, equivalently,
\begin{align}\label{case 1.3 omega decomp b}
\Omega=\Omega_0+\overline{s}\overline{H_\ell}^{-1}\alpha_{n+2k}^*\overline{H_\ell}.
\end{align}

Here we will apply another result from the appendix (below), namely Corollary \ref{nonnilp_cor}, which states that for $B\in\mathscr{A}$, since $A$ is not nilpotent, if 
$\left(H_\ell\overline{A} B\right)^T+H_\ell\overline{A} B=\mu H_\ell\overline{A} $ then $B AH_\ell^{-1} +   AH_\ell^{-1}B^T  =\mu AH_\ell^{-1}$. Noting that, by \eqref{some alphas are zero case 1 final subcase} and \eqref{case 1.3 alpha element eqn}, $\overline{A}\overline{H_\ell}^{-1}\alpha_{n+2k}^*\overline{H_\ell}=0$, item (iii) in \eqref{system} and \eqref{case 1.3 omega decomp b} imply that 
\begin{align}\label{case 1.3 omega zero coef}
\left(H_\ell\overline{A} \Omega_0\right)^T+H_\ell\overline{A} \Omega_0=\mu H_\ell\overline{A},
\end{align}
and hence Corollary \ref{nonnilp_cor} implies that 
\begin{align}\label{case 1.3 omega zero eta}
\eta_{\Omega_0}=\mu,
\end{align}
where this notation $\eta_{\Omega_0}$ refers to the coefficient with that label in items (i) and (ii) or \eqref{system}. 

Since the matrix equation $\left(H_\ell\overline{A} X\right)^T+H_\ell\overline{A} X=\mu H_\ell\overline{A}$ is equivalent to 
\[
\left(\overline{H}_\ell^{-1}X^*\overline{H}_\ell\right) AH_\ell^{-1}+AH_\ell^{-1}\left(\overline{H}_\ell^{-1}X^*\overline{H}_\ell\right)^T=\overline{\mu}AH_\ell^{-1},
\]
\eqref{case 1.3 omega zero coef} implies
\begin{align}\label{case 1.3 omega zero adjoint eta}
\eta_{\overline{H_\ell}^{-1}\Omega_0^*\overline{H_\ell}}=\overline{\mu}.
\end{align}
By \eqref{case 1.3 omega zero adjoint eta}, items (i) and (ii) in \eqref{system} imply
\begin{align}\label{case 1.3 omega zero adjoint condition 2}
\left[\Omega, \overline{H_\ell}^{-1}\Omega_0^*\overline{H_\ell}\right]+\overline{\mu}\Omega\in\mathscr{A}_0,
\end{align}
and applying the transformation $X\mapsto \overline{H_\ell}^{-1}X^*\overline{H_\ell}$ to the matrix in \eqref{case 1.3 omega zero adjoint eta} yields
\begin{align}\label{case 1.3 omega zero condition 2}
\left[\overline{H_\ell}^{-1}\Omega^*\overline{H_\ell}, \Omega_0\right]-\mu\overline{H_\ell}^{-1}\Omega_0^*\overline{H_\ell}\in\mathscr{A}_0.
\end{align}

Now we analyze item (iv) of \eqref{system}. Using \eqref{case 1.3 omega decomp a}, \eqref{case 1.3 omega decomp b}, and lastly \eqref{case 1.3 omega zero adjoint condition 2}, we have 
\begin{align}
\left[\overline{H_\ell}^{-1}\Omega^*\overline{H_\ell}, \Omega\right]&= \left[\overline{H_\ell}^{-1}\Omega_0^*\overline{H_\ell}, \Omega\right]+ \left[s\alpha_{n+2k},\Omega_0\right]+|s|^2\left[\alpha_{n+2k},\overline{H_\ell}^{-1}\alpha_{n+2k}^*\overline{H_\ell}\right] \\
&\equiv\overline \mu \Omega  +\left[s\alpha_{n+2k},\Omega_0\right]+|s|^2\left[\alpha_{n+2k},\overline{H_\ell}^{-1}\alpha_{n+2k}^*\overline{H_\ell}\right]
\pmod{\mathscr{A}_0}.
\end{align}
Substituting the last equation into  item (iv) of \eqref{system} we get, after the obvious cancellation, that
\begin{align}\label{case 1.3 omega zero adjoint condition 4}
\left[s\alpha_{n+2k},\Omega_0\right]+|s|^2\left[\alpha_{n+2k}, \overline{H_\ell}^{-1}\alpha_{n+2k}^*\overline{H_\ell}\right]+A\overline{A}-\mu\overline{H_\ell}^{-1}\Omega^*\overline{H_\ell}\in\mathscr{A}_0.
\end{align}
Similarly, \eqref{case 1.3 omega decomp a}, \eqref{case 1.3 omega decomp b}, and then \eqref{case 1.3 omega zero condition 2}  yields
\begin{align}
\left[\overline{H_\ell}^{-1}\Omega^*\overline{H_\ell}, \Omega\right]&=\left[\overline{H_\ell}^{-1}\Omega^*\overline{H_\ell}, \Omega_0\right]+\left[\overline{H_\ell}^{-1}\Omega_0^*\overline{H_\ell},\overline{s}\overline{H_\ell}^{-1}\alpha_{n+2k}^*\overline{H_\ell}\right]+|s|^2\left[\alpha_{n+2k},\overline{H_\ell}^{-1}\alpha_{n+2k}^*\overline{H_\ell}\right]\\
&\equiv \mu\overline{H_\ell}^{-1}\Omega_0^*\overline{H_\ell}+\left[\overline{H_\ell}^{-1}\Omega_0^*\overline{H_\ell},\overline{s}\overline{H_\ell}^{-1}\alpha_{n+2k}^*\overline{H_\ell}\right]+|s|^2\left[\alpha_{n+2k},\overline{H_\ell}^{-1}\alpha_{n+2k}^*\overline{H_\ell}\right],
\end{align}
where the equivalence is modulo $\mathscr{A}_0$. Substituting the last equation into  item (iv) of \eqref{system} we get
\begin{align}\label{case 1.3 omega zero condition 4}
\left[\overline{H_\ell}^{-1}\Omega_0^*\overline{H_\ell},\overline{s}\overline{H_\ell}^{-1}\alpha_{n+2k}^*\overline{H_\ell}\right]+|s|^2\left[\alpha_{n+2k},\overline{H_\ell}^{-1}\alpha_{n+2k}^*\overline{H_\ell}\right]+A\overline{A}-\overline{\mu}\Omega\in\mathscr{A}_0.
\end{align}
On the other hand, again from \eqref{case 1.3 omega decomp a} , \eqref{case 1.3 omega decomp b}, and using that   $\left[\overline{H_\ell}^{-1}\Omega_0^*\overline{H_\ell},\Omega_0\right]\in\mathscr{A}_0$, we can write
\begin{equation*}
\label{commutator Omega cong Omega}
\left[\overline{H_\ell}^{-1}\Omega^*\overline{H_\ell}, \Omega\right]\equiv \left[s\alpha_{n+2k},\Omega_0\right]+\left[\overline{H_\ell}^{-1}\Omega_0^*\overline{H_\ell},\overline{s}\overline{H_\ell}^{-1}\alpha_{n+2k}^*\overline{H_\ell}\right]+|s|^2\left[\alpha_{n+2k},\overline{H_\ell}^{-1}\alpha_{n+2k}^*\overline{H_\ell}\right],
\end{equation*}
where here again the equivalence is modulo $\mathscr{A}_0$. 
By subtracting the matrix in item (iv) of \eqref{system} from the sum of the matrices in \eqref{case 1.3 omega zero adjoint condition 4} and  \eqref{case 1.3 omega zero condition 4} and using the last relation, we get
\[
A\overline{A}+|s|^2\left[\alpha_{n+2k},\overline{H_\ell}^{-1}\alpha_{n+2k}^*\overline{H_\ell}\right]\in\mathscr{A}_0,
\]
or, equivalently,
\begin{align}\label{case 1.3 A squared in condition 4}
\left(A\overline{A}+|s|^2\alpha_{n+2k}\overline{H_\ell}^{-1}\alpha_{n+2k}^*\overline{H_\ell}\right)-|s|^2\overline{H_\ell}^{-1}\alpha_{n+2k}^*\overline{H_\ell}\alpha_{n+2k}\in\mathscr{A}_0.
\end{align}
Notice that the first two terms in \eqref{case 1.3 A squared in condition 4}, grouped together by parentheses, are matrices whose only potentially nonzero entry is the $(2k+1,2k+1)$ entry, whereas the other term has the same value in the first $2k+1$ entries of its main diagonal. By \eqref{case 1.3 intersection eqn c}, each matrix in $\mathscr{A}_0$ also has the same values in the first $2k+1$ entries of its main diagonal. Moreover, the $(2k+1,2k+1)$ entry of $A\overline{A}$ is nonzero. Therefore, by \eqref{case 1.3 A squared in condition 4},
\begin{align}\label{case 1.3 A squared rewritten}
A\overline{A}=-|s|^2\alpha_{n+2k}\overline{H_\ell}^{-1}\alpha_{n+2k}^*\overline{H_\ell}.
\end{align}
Defining
\begin{align}\label{case 1.3 A squared in condition 4 simplified a}
\alpha:=|s|^2\overline{H_\ell}^{-1}\alpha_{n+2k}^*\overline{H_\ell}\alpha_{n+2k},
\end{align}
\eqref{case 1.3 A squared in condition 4} and \eqref{case 1.3 A squared rewritten} imply that $\alpha$ is in $\mathscr{A}_0$.

It is straightforward to check that, with this definition for $\alpha$, $\eta_{\alpha}=0$ in the notation of item (i) of \eqref{system} (by calculating, for example, the $(1,1)$ entries of the terms in item (i)), and hence items (i) and (ii) of \eqref{system} yield $[\Omega,\alpha]\in\mathscr{A}_0$. Or, equivalently, by \eqref{case 1.3 omega decomp b}, noting that $[\Omega_0,\alpha]\in\mathscr{A}_0$,
\begin{align}\label{case 1.3 A squared in condition 4 simplified b}
\overline{s}\left[\overline{H_\ell}^{-1}\alpha_{n+2k}^*\overline{H_\ell},\alpha\right]\in\mathscr{A}_0.
\end{align}
Notice that $\overline{H_\ell}^{-1}\alpha_{n+2k}^*\overline{H_\ell}\alpha=0$ because $\left(\overline{H_\ell}^{-1}\alpha_{n+2k}^*\overline{H_\ell}\right)^2=0$, and hence \eqref{case 1.3 A squared in condition 4 simplified b} implies
\begin{align}\label{case 1.3 A squared in condition 4 simplified c}
\overline{s}|s|^2\overline{H_\ell}^{-1}\alpha_{n+2k}^*\overline{H_\ell}\left(\alpha_{n+2k}\overline{H_\ell}^{-1}\alpha_{n+2k}^*\overline{H_\ell}\right)\in\mathscr{A}_0.
\end{align}
Applying \eqref{case 1.3 A squared rewritten}, we get
\begin{align}\label{case 1.3 A squared in condition 4 simplified d}
-\frac{\overline{s}|s|^2}{|c|^2}\overline{H_\ell}^{-1}\alpha_{n+2k}^*\overline{H_\ell}\left(\alpha_{n+2k}\overline{H_\ell}^{-1}\alpha_{n+2k}^*\overline{H_\ell}\right)&=\frac{\overline{s}}{|c|^2}\overline{H_\ell}^{-1}\alpha_{n+2k}^*\overline{H_\ell}\left(A\overline{A}\right)\\
&=\overline{s}\overline{H_\ell}^{-1}\alpha_{n+2k}^*\overline{H_\ell},
\end{align}
where this last equality follows easily from \eqref{case 1.3 alpha element eqn}. 

By \eqref{case 1.3 omega decomp b}, \eqref{case 1.3 A squared in condition 4 simplified c}, and \eqref{case 1.3 A squared in condition 4 simplified d}, we get that $\Omega$ is in $\mathscr{A}_0$, but this contradicts Lemma \ref{omega in A implies regularity}. Therefore, the assumption that $\alpha_{n+2k}\neq 0$ must be false, which in turn implies that $a_1=\cdots=a_{n-1}=0$, completing this proof.
\end{proof}

\subsection{The second special case: }\label{case 2 subsection} In this subsection, \ref{case 2 subsection}, we consider the special case where we have some integer $1\leq m\leq n-1$ such that 
\begin{align}\label{case 2 condition a}
 H_\ell=
\left(
\begin{array}{c|c}
S_{2m}
& 
0
\\\hline
0
&H_\ell^\prime
\end{array}
\right),
\end{align}
where $H_\ell^\prime$ is an arbitrary nondegenerate Hermitian matrix, and
\begingroup\setlength{\abovedisplayskip}{16pt}
\begin{align}\label{case 2 condition b}
A=
\left(
\begin{array}{c|c}
\bovermat{$2m$ columns}{\mbox{
$\begin{array}{c|c} 0 & J_{m,\lambda} \\\hline I & 0\end{array}$
}}
& 
0
\\\hline
0 & A^\prime
\end{array}
\right)
\begin{array}{c}
\left.\phantom{\mbox{$\begin{matrix}0\\0 \end{matrix}$}}\right\}\text{\scriptsize $2m$ rows}\\\phantom{0}
\end{array},
\quad\quad\quad
  \mbox{for some  $\lambda\in\mathbb C\setminus\{x\in\R\,|\, x\geq0\}$,}
\end{align}
\endgroup
where $A^\prime$ is a matrix such that $(\ell, \boldsymbol{A})$ is represented by $(H_\ell,A)$.
The analysis in \ref{case 2 subsection} is similar to that of \ref{case 1 subsection}, but some formulas differ.

By Lemma \ref{matrix rep of larger index rows of tensor phi lemma prereq} there exist coefficients $\kappa_{1},\ldots,\kappa_{n-1}$, given in \eqref{alpha_through_C_and _H}, such that, first,
\[
\alpha_{n+m}e_i=-\delta_{i,m-1}\big(\varphi(1)\big)_{+}+\lambda \kappa_i e_{1},
\]
second, for any nonnegative integer  $p<m$,
\[
\alpha_{n+p}e_i=-\delta_{i,2m-p}\big(\varphi(1)\big)_{+}+\lambda \kappa_i e_{m+p},
\]
and, third, if $0<p<m$ then
\[
\alpha_{n+m+p}e_i=-\delta_{i,2m-p}\big(\varphi(1)\big)_{+}+ \kappa_i e_{p}+\lambda \kappa_i e_{p+1},
\]
which we use to obtain the following formulas. For $0\leq p<m$, we have
\begin{align}\label{case 2 alpha times CH^(-1) first eqn first}
(\alpha_{n+p}AH_\ell^{-1})e_i=(\kappa_{m-i}\lambda-\kappa_{m+1-i}\lambda^{2})e_{m+p+1}
\quad\quad\forall\, i\in\{1,\ldots, m-1\},
\end{align}
\begin{align}\label{case 2 alpha times CH^(-1) first eqn second}
(\alpha_{n+p}AH_\ell^{-1})e_m=\kappa_1\lambda^2e_{m+p+1},
\end{align}
\begin{align}\label{case 2 alpha times CH^(-1) first eqn third}
(\alpha_{n+p}AH_\ell^{-1})e_i=\kappa_{3m+1-i}\lambda e_{m+p+1}-\delta_{i,m+p+1}\big(\varphi(1)\big)_{+}
\quad\quad\forall i\in\{m+1,\ldots, 2m\},
\end{align}
and
\begin{align}\label{case 2 alpha times CH^(-1) first eqn fourth}
(\alpha_{n+p}AH_\ell^{-1})e_i\in\mathrm{span}\{e_{m+p+1}\}
\quad\quad\forall i>2m.
\end{align}

For any nonnegative integer  $p<m$ and $1\leq i\leq 2m$
\begin{align}\label{case 2 H^(-1)alpha_n^T H first}
\left(H_\ell^{-1}\alpha_{n+p}^T H_\ell\right) e_i \equiv -\varphi_{2m+1-i}e_{p+1}+\delta_{i,m-p}\sum_{j=1}^{2m}\kappa_{2m+1-j}\lambda e_{j} \pmod{\mathrm{span}\{e_{k}\}_{k=2m+1}^{n-1}}
\end{align}
and, moreover, this equivalence modulo $\mathrm{span}\{e_{k}\}_{k=2m+1}^{n-1}$ can be replaced with ordinary strict equivalence whenever $\delta_{i,m-p}=0$. Also, for $1\leq i\leq 2m$,
\begin{align}\label{case 2 H^(-1)alpha_n^T H second}
\left(H_\ell^{-1}\alpha_{n+m}^T H_\ell\right) e_i \equiv -\varphi_{2m+1-i}e_{m+1}+\delta_{i,2m}\sum_{j=1}^{2m}\kappa_{2m+1-j}\lambda e_{j} \pmod{\mathrm{span}\{e_{k}\}_{k=2m+1}^{n-1}},
\end{align}
where equivalence modulo $\mathrm{span}\{e_{k}\}_{k=2m+1}^{n-1}$ can be replaced with ordinary strict equivalence whenever $\delta_{i,2m-1}=0$.
For any $0<p<m$ and $0<i<2m+1$,
\begin{align}\label{case 2 H^(-1)alpha_n+m+p^T H first}
\left(H_\ell^{-1}\alpha_{n+m+p}^T H_\ell\right) e_i = -\varphi_{2m+1-i}e_{m+p+1}& +\delta_{i,2m-p}\left(\sum_{j=1}^{2m}\kappa_{2m+1-j}\lambda e_{j}+\sum_{k=2m+1}^{n-1}\kappa_{k}\lambda e_{k}\right) \\
& + \delta_{i,2m-p+1}\left(\sum_{j=1}^{2m}\kappa_{2m+1-j} e_{j}+\sum_{k=2m+1}^{n-1}\kappa_{k} e_{k}\right),
\end{align}
and for any $0<p<m$ and $2m<i<n$
\begin{align}\label{case 2 H^(-1)alpha_n+m+p^T H second}
\left(H_\ell^{-1}\alpha_{n+m+p}^T H_\ell\right) e_i = -\varphi_{i}e_{m+p+1}.
\end{align} 
Lastly, for all $i\geq 2m$
\begin{align}\label{case 2 H^(-1)alpha_n+p^T H large p second}
(H_\ell^{-1}\alpha_{n+p}^T H_\ell)e_i=-\left(\sum_{j=2m+1}^{n-1} (H_\ell)_{j,i} \varphi_j\right)e_{p+1}
\quad\quad\forall\, 0\leq p<m
\end{align}
and
\begin{align}\label{case 2 H^(-1)alpha_n+p^T H large p}
\left(H^{-1}_\ell\alpha_{n+p}^TH_\ell\right)e_{i}\subset \mathrm{span}\{e_{2m+1}, \ldots, e_{n-1}\}
\quad\quad\forall\, 2m\leq p.
\end{align}

\begin{lemma}\label{case 2 lemma}
In the special case of \ref{case 2 subsection} wherein \eqref{case 2 condition a} and \eqref{case 2 condition b} hold
\begin{equation}
\label{phi_1=0 case 2}
\varphi(1)=0.
\end{equation} 
\end{lemma}
\begin{proof}
By the same argument applied at the end of the proof of Lemma \ref{case 1 lemma}, it will suffice to show that $\big(\varphi(1)\big)_{+}=0$. Similar to the proof of Lemma \ref{case 1 lemma}, this proof consists of analysis of equation \eqref{HAH>n} in four cases:

{\bf 1.}\emph{ Equation \eqref{HAH>n} for $i=n$ and $j=n+p$ with $0\leq p < m$ and $m\neq 1$.} By \eqref{case 2 H^(-1)alpha_n^T H first} replacing $p$ with $0$ and replacing $i$ with $p+1$,
\begin{align}\label{case 2 lemma eq1}
\left(H^{-1}_\ell\alpha_{n}^TH_\ell\right)e_{p+1}=-\varphi_{2m-p}e_1 \quad\quad\forall\,\, 0\leq p< m-1,
\end{align}
and, by \eqref{case 2 H^(-1)alpha_n^T H first} with $i=1$,
\begin{align}\label{case 2 lemma eq2}
\left(H^{-1}_\ell\alpha_{n+p}^TH_\ell\right)e_{1}=-\varphi_{2m}e_{p+1}\quad\quad\forall\,\, 0\leq p<m-1.
\end{align}
Applying \eqref{HAH>n}, \eqref{case 2 lemma eq1}, and \eqref{case 2 lemma eq2} we get $\varphi_{m+2}=\varphi_{m+3}=\cdots=\varphi_{2m}=0$.
Furthermore, by \eqref{case 2 H^(-1)alpha_n^T H first} with $p=0$ and $i=m$,
\begin{align}\label{case 2 lemma eq1 a}
\left(H^{-1}_\ell\alpha_{n}^TH_\ell\right)e_{m}\equiv -\varphi_{m+1}e_{1}+\sum_{j=1}^{2m}\kappa_{2m+1-j}\lambda e_{j} \pmod{\mathrm{span}\{e_{k}\}_{k=2m+1}^{n-1}}
\end{align}
whereas, by \eqref{case 2 H^(-1)alpha_n^T H first} with $p=m-1$ and $i=1$,
\begin{align}\label{case 2 lemma eq2 a}
\left(H^{-1}_\ell\alpha_{n+m-1}^TH_\ell\right)e_{1}\equiv-\varphi_{2m}e_{m}+\sum_{j=1}^{2m}\kappa_{2m+1-j}\lambda e_{j} \pmod{\mathrm{span}\{e_{k}\}_{k=2m+1}^{n-1}}.
\end{align}
Applying \eqref{HAH>n}, \eqref{case 2 lemma eq2 a}, and \eqref{case 2 lemma eq1 a} yields $\varphi_{m+1}=0$ so, altogether, we have shown 
\begin{align}\label{case 2 coeff phi a}
\varphi_{m+1}=\cdots=\varphi_{2m}=0.
\end{align}

{\bf 2.}\emph{ Equation \eqref{HAH>n} for $i=n+m$ and $j=n+m+p$ with $0\leq p < m$ and $m\neq 1$.} By \eqref{case 2 H^(-1)alpha_n+m+p^T H first} replacing $p$ with $0$ and replacing $i$ with $m+p+1$,
\begin{align}\label{case 2 lemma eq1 second}
\left(H^{-1}_\ell\alpha_{n+m}^TH_\ell\right)e_{m+p+1}=-\varphi_{m-p}e_{m+1} \quad\quad\forall\,\, 0< p< m-1,
\end{align}
and, by \eqref{case 2 H^(-1)alpha_n+m+p^T H first} with $i=m+1$,
\begin{align}\label{case 2 lemma eq2 second}
\left(H^{-1}_\ell\alpha_{n+m+p}^TH_\ell\right)e_{m+1}=-\varphi_{m}e_{m+p+1}
\quad\quad\forall\,\, 0< p<m-1.
\end{align}
Applying \eqref{HAH>n}, \eqref{case 2 lemma eq1 second}, and \eqref{case 2 lemma eq2 second} we get $\varphi_{2}=\varphi_{3}=\cdots=\varphi_{m}=0$.
Furthermore, by \eqref{case 2 H^(-1)alpha_n+m+p^T H first} with $p=0$ and $i=2m$,
\begin{align}\label{case 2 lemma eq1 second a}
\left(H^{-1}_\ell\alpha_{n+m}^TH_\ell\right)e_{2m}=-\varphi_{1}e_{m+1}& +\left(\sum_{j=1}^{2m}\kappa_{2m+1-j}\lambda e_{j}+\sum_{k=2m+1}^{n-1}\kappa_{k}\lambda e_{k}\right) 
\end{align}
and, by \eqref{case 2 H^(-1)alpha_n^T H first} with $p=m-1$ and $i=m+1$,
\begin{align}\label{case 2 lemma eq2 second a}
\left(H^{-1}_\ell\alpha_{n+2m-1}^TH_\ell\right)e_{m+1}=-\varphi_{m}e_{2m}+\left(\sum_{j=1}^{2m}\kappa_{2m+1-j}\lambda e_{j}+\sum_{k=2m+1}^{n-1}\kappa_{k}\lambda e_{k}\right).
\end{align}
Applying \eqref{HAH>n}, \eqref{case 2 lemma eq1 second a}, and \eqref{case 2 lemma eq2 second a} yields $\varphi_{1}=\varphi_{m}=0$ so, altogether, noting \eqref{case 2 coeff phi a}, we have shown
\begin{align}\label{case 2 coeff phi second a large m}
\varphi_{1}=\cdots=\varphi_{2m}=0 \quad\mbox{ if }m>1.
\end{align}

{\bf 3.}\emph{ Equation \eqref{HAH>n} for $i=n$ and $j=n+m$.} By \eqref{case 2 H^(-1)alpha_n^T H first} and \eqref{case 2 H^(-1)alpha_n^T H second}
\begin{align}\label{case 2 lemma eq3.0} 
\left(H^{-1}_\ell\alpha_{n}^TH_\ell\right)e_{m+1}=-\varphi_{m}e_1
\quad\mbox{ and }\quad
\left(H^{-1}_\ell\alpha_{n+m}^TH_\ell\right)e_{1}=-\varphi_{2m}e_{m+1}.
\end{align}
By \eqref{HAH>n}, $\left(H^{-1}_\ell\alpha_{n}^TH_\ell\right)e_{m+1}=\left(H^{-1}_\ell\alpha_{n+m}^TH_\ell\right)e_{1}$, and hence \eqref{case 2 lemma eq3.0} implies $\varphi_{m}=\varphi_{2m}$. This is true in particular when $m=1$, which together with \eqref{case 2 coeff phi second a large m} yields the general result
\begin{align}\label{case 2 coeff phi second a}
\varphi_{1}=\cdots=\varphi_{2m}=0.
\end{align}

{\bf 4.}
\emph{ Equation \eqref{HAH>n} for $i=n$ and $j=n+p$ with $p\geq 2m$.} By \eqref{case 2 H^(-1)alpha_n+p^T H large p second} we get that 
\begin{align}\label{case 2 lemma eq3} 
\left(H^{-1}_\ell\alpha_{n}^TH_\ell\right)e_{p+1}=\left(\sum_{j=2m+1}^{n-1} (H'_\ell)_{j, i} \varphi_j\right)e_1.
\end{align}
Using \eqref{HAH>n} again, from \eqref{case 2 H^(-1)alpha_n+p^T H large p} and \eqref{case 2 lemma eq3} it follows that $\left(H^{-1}_\ell\alpha_{n}^TH_\ell\right)e_{p+1}=0$ or, equivalently, 
\begin{align}\label{case 2 lemma eq4} 
\sum_{j=1}^{n-1-m} (H'_\ell)_{j, i} \varphi_j=0, \quad \forall\, 1\leq i\leq n-1-2m.
\end{align}
Since the matrix $H_\ell^\prime$ is nonsigular, \eqref{case 2 lemma eq4} implies $\varphi_{2m+1}=\cdots=\varphi_{n-1}=0$, which together with \eqref{case 2 coeff phi second a} yields $\varphi_{1}=\cdots=\varphi_{n-1}=0$, that is, $\big(\varphi(1)\big)_{+}=0$.
\end{proof}

\begin{lemma}\label{case 2 lemma alt}
In the special case of \ref{case 2 subsection} wherein \eqref{case 2 condition a} and \eqref{case 2 condition b} hold, $(\kappa_1,\ldots, \kappa_{n-1})A=0$.
\end{lemma}
\begin{proof}
First we want to show that $\alpha_{n}AH_\ell^{-1}$ is skew symmetric, and we do so by considering two separate cases. 

First, consider the case where $m=1$. By \eqref{case 2 alpha times CH^(-1) first eqn second}, the $(1,1)$ entry of $\alpha_{n}AH_\ell^{-1}$ is zero. But the $(1,1)$ entry of $AH_\ell^{-1}$ is nonzero, so \eqref{orthogonal algebra 2 condition} implies that $\alpha_{n}AH_\ell^{-1}$ is skew symmetric.

Now let us consider the second case, which is where $m>1$. The right side of  \eqref{orthogonal algebra 2 condition} is either zero or its right side has rank equal to $\mathrm{rank}\,A$ (which is at least 4 because $m>1$). On the other hand, using formulas \eqref{case 2 alpha times CH^(-1) first eqn first}, \eqref{case 2 alpha times CH^(-1) first eqn second}, \eqref{case 2 alpha times CH^(-1) first eqn third}, and \eqref{case 2 alpha times CH^(-1) first eqn fourth} for the matrix $\alpha_{n}AH_\ell^{-1}$ together with Lemma \ref{case 2 lemma}, we can see that the matrix $\alpha_{n}AH_\ell^{-1}$ has rank at most $1$. Therefore the matrix on the left side of \eqref{orthogonal algebra 2 condition}  (when setting $i=n$) has rank at most $2$, and hence the matrix $\alpha_{n}AH_\ell^{-1}$ appearing in \eqref{case 2 alpha times CH^(-1) first eqn first} must be skew symmetric if $m>1$. 

So, for all values of $m$, we have shown that $\alpha_{n}AH_\ell^{-1}$ is skew symmetric and of rank at most $1$. Thus it is identically zero, which implies that the rows of $\alpha_n$ are in the left kernel of $AH_\ell^{-1}$. In particular,  $(\kappa_1,\ldots, \kappa_n)AH_\ell^{-1}=0$, which completes this proof because $H_\ell^\prime$ is nonsingular.
\end{proof}

\begin{lemma}\label{case 2 alphas with high index are zero}
In the special case of \ref{case 2 subsection} wherein \eqref{case 2 condition a} and \eqref{case 2 condition b} hold,  if $A$ corresponds to a non-regular CR structure then $\mathfrak{g}_1^{\mathrm{red}}=0$.
\end{lemma}
\begin{proof}
Let $\varphi\in \mathfrak{g}_{1}^{\mathrm{red}}$ and let $(\kappa_i)_{i=1}^{n-1}$ be as in \eqref{alpha_through_C_and _H}. By the same arguments as in the beginning of the proof of Lemma \ref{case 1 alphas with high index are zero} , it will suffice to show that $\kappa_i=0$ for every $1\leq i\leq n-1$. 

To produce a contradiction, let us assume there exists an index $i$ such that $\kappa_{i}\neq0$, and let $r$ be the smallest such index.
Since, by Lemma \ref{case 2 lemma alt}, $(\kappa_1,\ldots, \kappa_n)A=0$, we have $\kappa_1=\kappa_2=0$, and hence $2<r$. Also,
\begin{align}\label{first columns of alpha are zero, two special cases, case 2}
\alpha_{n+m}e_i=\delta_{i,r}\kappa_r\lambda e_1
\quad\mbox{ and }\quad
\alpha_{n+m+1}e_i=\delta_{i,r}(\kappa_r e_1+\kappa_r\lambda e_2)
\quad\quad\forall\,  i\leq r.
\end{align}

 By Lemma \ref{matrix rep of larger index rows of tensor phi lemma}, $\mathrm{span} \{\alpha_{n+m}, \alpha_{n+m+1}\}$ is a $2$-dimensional subspace in $\mathscr{A}+\mathbb C (\overline{H}^{-1}_\ell\Omega^* \overline{H}_\ell)$. Since $\mathscr{A}$ is a subspace in $\mathscr{A}+\mathbb C (\overline{H}^{-1}_\ell\Omega^* \overline{H}_\ell)$ of codimension at most  $1$ it has a nontrivial intersection with $\mathrm{span} \{\alpha_{n+m}, \alpha_{n+m+1}\}$, and hence there exist  $b_1,b_2\in\mathbb C$ such that  $(b_1,b_2)\neq(0,0)$ and 
\begin{equation}
\label{case 2 induction on r setup eqn}
b_1\alpha_{n+m}+b_2\alpha_{n+m+1}\in \mathscr{A}.
\end{equation}
By \eqref{first columns of alpha are zero, two special cases, case 2} the first $r-1$ columns of the matrix $b_1\alpha_n+b_2\alpha_n$ vanish and 
\begin{equation}
\label{case 2 induction on r Lemma eqn}
(b_1 \alpha_n+b_2 \alpha_{n+1})e_r=\kappa_r\Bigl ((\lambda b_1+b_2)e_1+\lambda b_2 e_2\Bigr).
\end{equation}

Using results from the appendix (Section \ref{Matrix representations of intersection algebra general formula} below), we can now derive a contradiction as follows. Let $b_1\alpha_n+b_2\alpha_{n+1}$ be partitioned as a block matrix whose diagonal blocks have the same size as the diagonal blocks of $A$. By \eqref{case 2 induction on r setup eqn}, each $(i,j)$ block of $b_1\alpha_n+b_2\alpha_{n+1}$ is either characterized by Lemma \ref{lemma for reduction to eSpaces} and identically zero or it is characterized by Corollaries \ref{Bij formula with nonzero lambda} and \ref{Bii formula} and more specifically characterized by \eqref{intersection algebra lambda nonreal off diag a}, \eqref{intersection algebra lambda nonreal off diag b}, \eqref{intersection algebra lambda negative off diag a}, \eqref{intersection algebra lambda negative off diag b}, and \eqref{intersection algebra lambda diag}. 
Notice that if this $(1,j)$ block of $b_1\alpha_n+b_2\alpha_{n+1}$ is characterized by \eqref{intersection algebra lambda diag} then $j=1$, and clearly no matrix of the form in \eqref{intersection algebra lambda diag} can have nonzero values in either of the first two entries of its first nonzero column, which shows that this $(1,j)$ block of $b_1\alpha_n+b_2\alpha_{n+1}$ containing part of the $r$ column of $b_1\alpha_n+b_2\alpha_{n+1}$ must be zero if it is characterized by \eqref{intersection algebra lambda diag}. 

If, on the other hand, the $(1,j)$ block of $b_1\alpha_n+b_2\alpha_{n+1}$ is characterized by \eqref{intersection algebra lambda nonreal off diag a} or \eqref{intersection algebra lambda nonreal off diag b} (respectively \eqref{intersection algebra lambda negative off diag a} or \eqref{intersection algebra lambda negative off diag b}), is nonzero, and contains part of the $r$ column of $b_1\alpha_n+b_2\alpha_{n+1}$, then \eqref{intersection algebra lambda nonreal off diag a} and \eqref{intersection algebra lambda nonreal off diag b} (respectively \eqref{intersection algebra lambda negative off diag a} and \eqref{intersection algebra lambda negative off diag b}) imply that the $(j,1)$ block of $b_1\alpha_n+b_2\alpha_{n+1}$ is nonzero and contained in the first $r-1$ columns of $b_1\alpha_n+b_2\alpha_{n+1}$, which contradicts our definition of $r$. Therefore, the $(1,j)$ block of $b_1\alpha_n+b_2\alpha_{n+1}$ containing part of the $r$ column of $b_1\alpha_n+b_2\alpha_{n+1}$ is identically zero, which, by \eqref{case 2 induction on r Lemma eqn}, implies that $\lambda b_1+b_2=0$ and $\lambda b_2=0$. Yet this yields the contradiction $(b_1,b_2)=(0,0)$. 
\end{proof}

\subsection{The third special case:}\label{case 3 subsection} In this subsection, \ref{case 3 subsection}, we consider the special case where $(H_\ell, A)$ corresponds to a non-regular CR structure and $A$ is  diagonal. Working in the normal form of Theorem \ref{simultaneous canonical form theorem}, $H_\ell$ is diagonal too. Since $A$ corresponds to a non-regular CR structure, the matrix $A\overline{A}$ has at least two distinct nonzero eigenvalues, so we can assume without loss of generality that there are numbers $\lambda_1,\ldots, \lambda_{n-1},\in\mathbb C$  and $\epsilon_1,\ldots, \epsilon_{n-1}\in\{1,-1\}$ such that $|\lambda_1|\neq|\lambda_2|$, $\lambda_1\neq0$,  $\lambda_2\neq0$, and
\[
A=\mathrm{diag}\,(\lambda_1, \ldots, \lambda_{n-1})
\quad\mbox{ and }\quad 
H_\ell=\mathrm{diag}\,(\epsilon_1, \ldots, \epsilon_{n-1})
.
\]
Accordingly, by \eqref{alpha_through_C_and _H},
\begin{align}\label{alpha_n formulas for diagonal case}
\alpha_{n+p}e_i=\kappa_i \lambda_{p+1}  e_{p+1}-\delta_{i, p+1} \varphi(1) 
\quad\quad\forall\, 0\leq p<n,
\end{align}
\begin{align}\label{alpha C H inverse formula for diagonal case}
\alpha_nAH_\ell^{-1} e_i=\lambda_i \varepsilon_i \kappa_i\big(\lambda_1  e_1-\delta_{i, 1} \varphi(1)\bigr),
\end{align}
\begin{align}\label{alpha_n formulas for diagonal case b}
H^{-1}\alpha_{n+p}^T H e_1=\pm \varphi_1 e_{p+1}
\quad\quad\forall 0\leq p<n,
\end{align}
and 
\begin{align}\label{alpha_n formulas for diagonal case c}
H^{-1}\alpha_{n}^T H e_{p+1}=\pm \varphi_{p+1} e_{1}
\quad\quad\forall 0<p<n.
\end{align}

By \eqref{HAH>n}, we can equate
$H^{-1}\alpha_{n}^T H e_{p+1}$, and hence \eqref{alpha_n formulas for diagonal case b} and \eqref{alpha_n formulas for diagonal case c} yields
\begin{align}\label{varphi first coefficients are zero in diagonal case}
\varphi_1=\varphi_2=\cdots=\varphi_{n-1}=0.
\end{align}

Formula in \eqref{alpha C H inverse formula for diagonal case} now simplifies giving that $\alpha_nAH_\ell^{-1}$ is a matrix with at most 1 nonzero row, and hence the left side of \eqref{orthogonal algebra 2 condition} (when setting $i=n$) cannot be a diagonal matrix of rank greater than one. Yet the right side of  \eqref{orthogonal algebra 2 condition} is a diagonal matrix that is either zero or of rank greater than 1, so the right side of  \eqref{orthogonal algebra 2 condition} must be zero for the equation to hold. Since the left side of \eqref{orthogonal algebra 2 condition} is zero, \eqref{alpha C H inverse formula for diagonal case} and \eqref{varphi first coefficients are zero in diagonal case} imply that 
\begin{align}\label{a coefficients are zero in diagonal case}
\lambda_1\kappa_1=\lambda_2\kappa_2=\cdots=\lambda_{n-1}\kappa_{n-1}=0
\end{align}
because $\lambda_1\neq 0$. In particular,
\begin{align}\label{a coefficients are zero in diagonal case a}
\kappa_1=\kappa_2=0
\end{align}
because $\lambda_1$ and $\lambda_2$ are both nonzero.

\begin{lemma}\label{subcase 3 alpha i is zero}
If $(H_\ell, A)$ corresponds to a non-regular CR structure and $A$ is  diagonal then $\mathfrak{g}_1^{\mathrm{red}}=0$.
\end{lemma}
\begin{proof}

Let $\varphi\in \mathfrak{g}_{1}^{\mathrm{red}}$ and let $(\kappa_i)_{i=1}^{n-1}$ be as in \eqref{alpha_through_C_and _H}. Recall that $\big(\varphi(1)\big)_{+}=0$ implies $\varphi(1)=0$, by the same argument applied at the end of the proof of Lemma \ref{case 1 lemma}, and hence $\varphi(1)=0$ by \eqref{varphi first coefficients are zero in diagonal case}. Accordingly, by the same arguments as in the beginning of the proof of Lemma \ref{case 1 alphas with high index are zero}, it will suffice to show that $\kappa_i=0$ for every $1\leq i\leq n-1$.

Assume that there exists $r$ such that $\kappa_r\neq 0$ and $r$ is the minimal index with this property. By \eqref{a coefficients are zero in diagonal case a} we have that $r>2$. Noting \eqref{alpha_n formulas for diagonal case}, by Lemma \ref{matrix rep of larger index rows of tensor phi lemma},  $\kappa_r\neq 0$ implies
$\mathrm{span} \{\alpha_{n}, \alpha_{n+1}\}$ is a $2$-dimensional subspace in $\mathscr{A}+\mathbb C (\overline{H}^{-1}_\ell\Omega^* \overline{H}_\ell)$. Accordingly, $\kappa_r\neq0$ yields that   $\mathrm{span} \{\alpha_{n}, \alpha_{n+1}\}$ and $\mathscr{A}$ have at least a $1$-dimensional intersection. By \eqref{a coefficients are zero in diagonal case a} and\eqref{alpha_n formulas for diagonal case}, nonzero entries in the matrices in $\mathrm{span} \{\alpha_{n}, \alpha_{n+1}\}$ can only appear in their first two rows and moreover they do not appear in their first two columns. Yet, in the appendix (Section \ref{Matrix representations of intersection algebra general formula} below), we describe the matrices in $\mathscr{A}$ explicitly. In particular, given that $H_\ell$ and $A$ are diagonal, the description of $\mathscr{A}$ in the appendix implies that every matrix in $\mathscr{A}$ with nonzero entries in its first two rows also has nonzero entries in its first two columns, which implies that $\mathrm{span} \{\alpha_{n}, \alpha_{n+1}\}$ and $\mathscr{A}$ have a trivial intersection, a clear contradiction. 
\end{proof}

By combining the results of Lemmas \ref{case 1 alphas with high index are zero}, \ref{case 1 alphas with high index are zero with small nilpotent blocks}, \ref{case 1 alphas with high index are zero with 2 by 2 nilpotent blocks}, \ref{case 2 alphas with high index are zero} , and \ref{subcase 3 alpha i is zero}, we finish the proof of item (1) of Theorem \ref{first prolongation for non-regular structures is zero}, because these lemmas account for all non-regular symbols.

To prove item (2) of Theorem \ref{first prolongation for non-regular structures is zero} note that by \eqref{g00 subalg. representation}  and Lemma \ref {dimension_cor},  for the reduced modified CR symbol corresponding to a non-regular symbol,
\[
\dim\,\mathfrak{g}_{0,0}^{\mathrm{red}}=\dim\,\mathscr A+1< n^2-4n+7.
\] 
Therefore, from item (1) of the theorem  under consideration and the fact that  
$\dim \,\mathfrak{g}_{0}^{\mathrm{red}}=\dim \,\mathfrak{g}_{0,0}^{\mathrm{red}}+2$ and $\dim \,\mathfrak g_-=2n-1$, it follows that
\[
\dim \, \mathfrak u(\mathfrak{g}^{0, \mathrm{red}})=\dim\,\mathfrak{g}^{0, \mathrm{red}}<(2n-1)+ (n^2-4n+7)+2=(n-1)^2+7,
\]  
which together with Theorem \ref{main theorem for non-regular symbols} completes the proof of item (2) of Theorem \ref{first prolongation for non-regular structures is zero}. Item (3) of Theorem \ref{first prolongation for non-regular structures is zero} follows from item (1) of Theorem \ref{first prolongation for non-regular structures is zero} and the parallelism construction referred to in \cite[Theorem 6.2]{SZ2020}.

\section{Appendix: Matrix representations of the algebra  \texorpdfstring{$\mathscr{A}$}{A}}\label{Matrix representations of intersection algebra general formula}

In this appendix we give a general formula for matrices in the algebra $\mathscr{A}$ defined in \eqref{intersection algebra} together with an outline for how the formula can be verified. {The complete formula is presented in several parts in Lemmas \ref{lemma for reduction to eSpaces}, \ref{Bij formula with zero lambda}, and \ref{condition on C for 2-dimensional scaling component} and Corollaries \ref{Bij formula with nonzero lambda}, \ref{Bii formula}, and  \ref{nonnilp_cor}. We use this explicit formula to derive upper bounds for the dimension of $\mathscr{A}$ given in Lemma \ref{dimension_cor}, which is essential for proving item (2) in Theorem \ref{first prolongation for non-regular structures is zero}. These upper bounds also immediately lead to the previously stated Theorem \ref{more precise bound}, which gives more precise bounds than those in Theorem \ref{first prolongation for non-regular structures is zero}. Furthermore, the matrix representation formula presented in this section plays a fundamental role in the proof of item (1) in Theorem \ref{first prolongation for non-regular structures is zero} given in Section \ref{First prolongation of non-regular symbols}.}

Naturally, it is easier to verify the formula than to derive it, and, since the formula is ancillary to this paper's topic, we omit the analysis used to derive it. The formula depends on the matrices $H_\ell$ and $A$ representing the pair $(\ell,\boldsymbol{A})$.

In the sequel we assume that  $H_\ell$ and $A$ are in the canonical form prescribed by Theorem \ref{simultaneous canonical form theorem}, namely as given in \eqref{simultaneous canonical form theorem eqn}. We will also use the notation of Section \ref{Matrix representations of local invariants}, and, in particular, we let $\lambda_1,\ldots, \lambda_\gamma$, $m_1,\ldots, m_\gamma$, $\epsilon_1,\ldots, \epsilon_\gamma$, $M_{\lambda_i,m_i}$ and $N_{\lambda_i,m_i}$ as in Theorem \ref{simultaneous canonical form theorem}. Recall that, in particular, this means the real and imaginary parts of each $\lambda_i$ are both nonnegative.

Define the \emph{bi-orthogonal subalgebra} of $\mathscr{A}$ to be
\[
\mathscr{A}^o:=\{B\in\mathscr{A}\, | \, B AH_\ell ^{-1} +   AH_\ell ^{-1}B^T = B^TH_\ell \overline{A}+H_\ell \overline{A} B=0\},
\]
where this name is reflecting the observation that $\mathscr{A}^o$ is analogous to an intersection of two orthogonal algebras. In this appendix, we first obtain a formula describing the elements in $\mathscr{A}^o$ and then obtain a formula for a subspace $\mathscr{A}^s\subset \mathscr{A}$ complementary to $\mathscr{A}^o$, that is,
such that 
\begin{align}\label{decomp of intersection algebra}
\mathscr{A}=\mathscr{A}^o\oplus\mathscr{A}^s.
\end{align}
Such a space $\mathscr{A}^s$ is spanned by elements that we call \emph{conformal scaling} elements of $\mathscr{A}$,  referring to the observation that these are analogous to non-orthogonal elements in an intersection of two conformally orthogonal algebras.

To begin, let $B$ be an $(n-1)\times (n-1)$ matrix in $\mathscr{A}^o$ and partition $B$ into blocks $\{B_{(i,j)}\}_{i,j=1}^\gamma$ 
where the number of rows in  $B_{(i,j)}$ is the same  as in the  matrix $M_{\lambda_i,m_i}$ and the number of columns in $B_{(i,j)}$ is the same  as in the  matrix $M_{\lambda_j,m_j}$.
Similarly, we partition $H_\ell\overline{A}B$ and $BAH_\ell^{-1}$ into blocks $\{(H_\ell\overline{A}B)_{(i,j)}\}_{i,j=1}^\gamma$ and $\{(BAH_\ell^{-1})_{(i,j)}\}_{i,j=1}^\gamma$ whose sizes are the same as in the partition of $B$. 

Let us now derive a relationship between the blocks $B_{(i,j)}$ and $B_{(j,i)}$. To simplify formulas, we assume $\epsilon_i=\epsilon_j$. To treat the more general case where possibly $\epsilon_i\neq\epsilon_j$, one can simply replace $N_{\lambda_i,m_i}$ (or $N_{\lambda_j,m_j}$) with $\epsilon_iN_{\lambda_i,m_i}$ (or $\epsilon_jN_{\lambda_j,m_j}$) in all of the subsequent formulas.

We have
\[
(BAH_\ell^{-1})_{(i,j)}=B_{(i,j)} M_{\lambda_j,m_j}N_{\lambda_j,m_j}
\quad\mbox{ and }\quad
(H_\ell\overline{A}B)_{(i,j)}=N_{\lambda_i,m_i}\overline{M_{\lambda_i,m_i}}B_{(i,j)},
\]
so, since $B\in \mathscr{A}$, 
\[
\left(M_{\lambda_i,m_i}N_{\lambda_i,m_i}\right)^TB_{(j,i)}^T=- B_{(i,j)} M_{\lambda_j,m_j}N_{\lambda_j,m_j}
\]
and
\[
B_{(j,i)}^T \left(N_{\lambda_j,m_j}\overline{M_{\lambda_j,m_j}}\right)^T=-N_{\lambda_i,m_i}\overline{M_{\lambda_i,m_i}}B_{(i,j)}.
\]
Since $\boldsymbol{A}$ is $\ell$-self-adjoint, each matrix $N_{\lambda_k,m_k}\overline{M_{\lambda_k,m_k}}$ and $M_{\lambda_k,m_k}N_{\lambda_k,m_k}$ is symmetric (one can also verify this by directly using the canonical form), and hence
\begin{align}\label{Bij block for nonequal eVals a}
M_{\lambda_i,m_i}N_{\lambda_i,m_i}B_{(j,i)}^T=-B_{(i,j)} M_{\lambda_j,m_j}N_{\lambda_j,m_j},
\end{align}
and
\begin{align}\label{Bij block for nonequal eVals b}
B_{(j,i)}^T N_{\lambda_j,m_j}\overline{M_{\lambda_j,m_j}}=-N_{\lambda_i,m_i}\overline{M_{\lambda_i,m_i}}B_{(i,j)}.
\end{align}
Multiplying both sides of \eqref{Bij block for nonequal eVals b} by $M_{\lambda_j,m_j}N_{\lambda_j,m_j}$ from the right and then applying \eqref{Bij block for nonequal eVals a} yields
\begin{align}\label{Bij block for nonequal eVals c}
B_{(j,i)}^T N_{\lambda_j,m_j}\overline{M_{\lambda_j,m_j}}M_{\lambda_j,m_j}N_{\lambda_j,m_j}&=-N_{\lambda_i,m_i}\overline{M_{\lambda_i,m_i}}B_{(i,j)}M_{\lambda_j,m_j}N_{\lambda_j,m_j}\\
&=N_{\lambda_i,m_i} \overline{M_{\lambda_i,m_i}}M_{\lambda_i,m_i}N_{\lambda_i,m_i}B_{(j,i)}^T.
\end{align}
Multiplying \eqref{Bij block for nonequal eVals c} by $N_{\lambda_i,m_i}$ from the left and by $N_{\lambda_j,m_i}$ from the right yields
\begin{align}\label{Bij block for nonequal eVals d}
\left(N_{\lambda_i,m_i}B_{(i,j)}^T N_{\lambda_j,m_j}\right)\overline{M_{\lambda_j,m_j}}M_{\lambda_j,m_j}=\overline{M_{\lambda_i,m_i}}M_{\lambda_i,m_i}\left(N_{\lambda_i,m_i}B_{(i,j)}^TN_{\lambda_j,m_j} \right).
\end{align}
Notice that \eqref{Bij block for nonequal eVals a} is also equivalent to
\begin{align}\label{Bij block for nonequal eVals e}
N_{\lambda_i,m_i}M_{\lambda_i,m_i}\left(N_{\lambda_j,m_j}B_{(j,i)}N_{\lambda_i,m_i}\right)^T=-\left(N_{\lambda_i,m_i}B_{(i,j)} N_{\lambda_j,m_j}\right)N_{\lambda_j,m_j}M_{\lambda_j,m_j}.
\end{align}

Equation \eqref{Bij block for nonequal eVals d} gives us all restrictions on the general form of $B_{(i,j)}$ that are not coming from the relationship between $B_{(i,j)}$ and other blocks in the matrix $B$. Equation \eqref{Bij block for nonequal eVals e}, on the other hand, gives us the restrictions on the general form of $B_{(i,j)}$ coming from its relationship with $B_{(j,i)}$. Moreover, if \eqref{Bij block for nonequal eVals d} and \eqref{Bij block for nonequal eVals e} are satisfied for $i$ and $j$ then $B$ is in $\mathscr{A}^o$ because \eqref{Bij block for nonequal eVals a} and \eqref{Bij block for nonequal eVals b} hold. In other words, our present goal is to solve the system of matrix equations in \eqref{Bij block for nonequal eVals d} and \eqref{Bij block for nonequal eVals e}, and whenever $(\lambda_i,\lambda_j)\neq(0,0)$, this exercise is equivalent to first solving the matrix equation
\begin{align}\label{Bij block for nonequal eVals f}
X\overline{M_{\lambda_j,m_j}}M_{\lambda_j,m_j}=\overline{M_{\lambda_i,m_i}}M_{\lambda_i,m_i}X,
\end{align}
and then, for the case where $i=j$, solving the system of equations consisting of \eqref{Bij block for nonequal eVals f} and
\begin{align}\label{Bij block for nonequal eVals g}
N_{\lambda_i,m_i}M_{\lambda_i,m_i}X^T=-XN_{\lambda_i,m_i}M_{\lambda_i,m_i}.
\end{align}
The case where $\lambda_i=\lambda_j=0$ requires special treatment because, in this case, contrary to the case where $(\lambda_i,\lambda_j)\neq(0,0)$, even if $i\neq j$ solutions for $B_{(i,j)}$ in \eqref{Bij block for nonequal eVals d} need not satisfy \eqref{Bij block for nonequal eVals e} for any matrix $B_{(j,i)}$.

Equation \eqref{Bij block for nonequal eVals f} is of the form analyzed in \cite[Chapter 8]{gantmakher1953theory}. In fact, an explicit solution to \eqref{Bij block for nonequal eVals f} is given in \cite[Chapter 8]{gantmakher1953theory}, but the solution is expressed in terms of a basis with respect to which $\overline{M_{\lambda_i,m_i}}M_{\lambda_i,m_i}$ and $\overline{M_{\lambda_j,m_j}}M_{\lambda_j,m_j}$ have their Jordan normal forms. On the other hand, the transition matrix from the initially considered basis to a basis of the Jordan normal form is block-diagonal with the blocks corresponding to the Jordan blocks. Hence, the following lemma can be obtained from the solution in \cite[Chapter 8]{gantmakher1953theory}.
\begin{lemma}\label{lemma for reduction to eSpaces}
If $\lambda_i\neq \lambda_j$ then $B_{(i,j)}=0$.
\end{lemma}
\begin{proof}
Since the real and imaginary parts of $\lambda_i$ and  $\lambda_j$ are all nonnegative, if $\lambda_i\neq \lambda_j$ then the eigenvalues of $\overline{M_{\lambda_i,m_i}}M_{\lambda_i,m_i}$ all differ from the eigenvalues of $\overline{M_{\lambda_j,m_j}}M_{\lambda_j,m_j}$. Accordingly, by \cite[Chapter 8, Theorem 1 and  Equation (11)]{gantmakher1953theory}, the matrix $X$ in \eqref{Bij block for nonequal eVals f} is zero.
\end{proof}

Given Lemma \ref{lemma for reduction to eSpaces}, all that remains is to find the general formula for $B_{(i,j)}$ when $\lambda_i=\lambda_j$. We will say that a Toeplitz $p\times q$ matrix is an \emph{upper-triangular Toeplitz matrix}, if the only nonzero entries appear on or above the main diagonal in their right-most $p\times p$ block if $p\leq q$, and the top-most $q\times q$ block if $p\geq q$ (in the terminology of \cite[Chapter 8]{gantmakher1953theory} they are called regular upper-triangular, but we avoid this terminology because the term ``regular'' is already assigned in the present paper to another concept).   
\begin{lemma}\label{equal eValue off diag dimension}
Suppose $\lambda_i= \lambda_j$ and $m_{i}\leq m_j$. The dimension of the space of solutions of \eqref{Bij block for nonequal eVals f} is equal to 
\begin{enumerate}
    \item $m_i$ if $\lambda_i>0$;
    \item $2m_i$ if $\lambda_i^2\not\in\R$;
    \item $4m_i$ if $\lambda_i^2<0$.
\end{enumerate}
\end{lemma}
\begin{proof}
We use \cite[Chapter 8, Theorem 1]{gantmakher1953theory} again for each of the cases.

Suppose first that $\lambda_i>0$. If $\lambda>0$ then $\overline{M_{\lambda,m}}M_{\lambda,m}$ is similar to the Jordan matrix $J_{\lambda^2,m}$. Let $U_i$ and $U_j$ be invertible matrices such that $U_j\overline{M_{\lambda_j,m_j}}M_{\lambda_j,m_j}U_j^{-1}=J_{\lambda_j^2,m_j}$ and $U_i\overline{M_{\lambda_i,m_i}}M_{\lambda_i,m_i}U_i^{-1}=J_{\lambda_i^2,m_i}$. For a matrix $X$ satisfying \eqref{Bij block for nonequal eVals f}, set $\widetilde{X}=U_j^{-1}XU_i$ so that, by \eqref{Bij block for nonequal eVals f}, 
\begin{align}\label{Bij block for nonequal eVals f new basis}
\widetilde{X} J_{\lambda_j^2,m_j}=J_{\lambda_i^2,m_i} \widetilde{X}.
\end{align}
It is shown in \cite[Chapter 8, Theorem 1]{gantmakher1953theory} that the space of solutions of \eqref{Bij block for nonequal eVals f new basis} consists of upper-triangular Toeplitz matrices. Therefore, the space of solutions of \eqref{Bij block for nonequal eVals f new basis} has dimension $m_i$, which shows item (1) because $X\mapsto U_j^{-1}XU_i$ gives an isomorphism between the space of solutions of \eqref{Bij block for nonequal eVals f new basis} and the space of solutions of \eqref{Bij block for nonequal eVals f}.

Let us now suppose $\lambda_i^2\not\in\R$ or $\lambda_i^2<0$. If $\lambda^2\not\in\R$ or $\lambda^2<0$ then
\begin{align}\label{mBar m with two blocks}
\overline{M_{\lambda,m}}M_{\lambda,m}=J_{\lambda^2,m}\oplus J_{\overline{\lambda}^2,m}.
\end{align}
For a matrix $X$ satisfying \eqref{Bij block for nonequal eVals f}, consider the $2\times 2$ block matrix partition $(X_{(r,s)})_{r,s\in\{1,2\}}$ of $X$ whose blocks are all $m_i\times m_j$ matrices. It is shown in \cite[Chapter 8, Theorem 1]{gantmakher1953theory} that the space of solutions of \eqref{Bij block for nonequal eVals f} with $\overline{M_{\lambda_i,m_i}}M_{\lambda_i,m_i}$ and $\overline{M_{\lambda_j,m_j}}M_{\lambda_j,m_j}$ of the form in \eqref{mBar m with two blocks} consists of matrices $(X_{(r,s)})_{r,s\in\{1,2\}}$ for which each $X_{(r,s)}$ is an upper-triangular Toeplitz matrix, where, moreover, if $\lambda_i^2\neq \overline{\lambda_i}^2$ then $X_{(1,2)}=X_{(2,1)}=0$. Accordingly, if $\lambda_i^2\not\in \R$ (respectively $\lambda_i^2<0$) then solutions to \eqref{Bij block for nonequal eVals f} are determined by two (respectively four)  upper-triangular Toeplitz $m_i\times m_i$ matrices. Items (2) and (3) follow because each upper-triangular Toeplitz $m_i\times m_i$ is determined by $m_i$ variables.
\end{proof}
\begin{corollary}\label{Bij formula with nonzero lambda}
If $m_i\leq m_j$, $\lambda_i=\lambda_j=\lambda$ and $\lambda\neq0$ then the matrices $B_{(i,j)}$ and $B_{(j,i)}$ are described by one of three formulas, where the correct formula depends on $\lambda$. In the formulas below , as before, $T_m$ denotes the $m\times m$ nilpotent Jordan block $J_{0, m}$.

\begin{enumerate} 
\item If $\lambda>0$ then $B_{(i,j)}$ and $B_{(j,i)}$ respectively equal
\begingroup\setlength{\abovedisplayskip}{18pt}
\begin{equation}\label{intersection algebra lambda real off diag}
\left(
\begin{array}{c|c}
\bovermat{\parbox{2cm}{\linespread{.6}\selectfont\centering$m_j-m_i$ columns}}{\mbox{
$\begin{matrix}
0&\cdots&0\\
\vdots&&\vdots\\
\vdots&&\vdots\\
0&\cdots&0\end{matrix}$
}}
&
\sum\limits_{k=0}^{m_i-1} b_k T_{m_i}^k
\end{array}
\right)
,\quad \mbox{ and} \quad
-\epsilon_i\epsilon_j\left(
\begin{array}{c}
\sum\limits_{k=0}^{m_i-1} b_k T_{m_i}^k\\\hline
\begin{matrix}
0&\cdots&0\\
\vdots&&\vdots\\
0&\cdots&0\end{matrix}
\end{array}
\right)
\begin{array}{c}
\vphantom{0}\\
\vphantom{0}\\
\left.\vphantom{\begin{matrix}0\\0\\0\end{matrix}}
\right\}\parbox{1.6cm}{\scriptsize $m_j-m_i$ rows,}
\end{array}
\end{equation}
\endgroup
for some coefficients $\{b_k\}$. 
\item 
 If  $\lambda^2\not\in\R$ then
\begin{align}\label{intersection algebra lambda nonreal off diag a}
\phantom{\begin{array}{c}\vdots\\\vdots\\\vdots\\\vdots\\\vdots\\a \end{array}}
B_{(i,j)}=
\left(
\begin{array}{c|c|c|c}
\bovermat{\parbox{2cm}{\linespread{.6}\selectfont\centering$m_j-m_i$ columns}}{\mbox{
$\begin{matrix}
0&\cdots&0\\
\vdots&&\vdots\\
\vdots&&\vdots\\
0&\cdots&0\end{matrix}$
}}
&
\begin{matrix}
\sum\limits_{k=0}^{m_i-1} a_k T_{m_i}
\\\hline
0\vphantom{\sum\limits_{k=0}^{m_i}}
\end{matrix}
&
\bovermat{\parbox{2cm}{\linespread{.6}\selectfont\centering$m_j-m_i$ columns}}{\mbox{
$\begin{matrix}
0&\cdots&0\\
\vdots&&\vdots\\
\vdots&&\vdots\\
0&\cdots&0\end{matrix}$
}}
&
\begin{matrix}
0\vphantom{\sum\limits_{k=0}^{m_i}}
\\\hline
\sum\limits_{k=0}^{m_i} b_k T_{m_i}
\end{matrix}
\end{array}
\right),
\end{align}
and
\begin{align}\label{intersection algebra lambda nonreal off diag b}
B_{(j,i)}=
-\epsilon_i\epsilon_j\left(
\begin{array}{c}
\begin{array}{c|c}
\sum\limits_{k=0}^{m_i} a_k T_{m_i}&\hspace{.7cm}0\hspace{.7cm}
\end{array}
\\\hline
\begin{matrix}
0&\cdots\cdots\cdots\cdots\cdots&0\\
\vdots&&\vdots\\
0&\cdots\cdots\cdots\cdots\cdots&0
\end{matrix}
\\\hline
\begin{array}{c|c}
\hspace{.7cm}0\hspace{.7cm}&\sum\limits_{k=0}^{m_i} b_k T_{m_i}
\end{array}
\\\hline
\begin{matrix}
0&\cdots\cdots\cdots\cdots\cdots&0\\
\vdots&&\vdots
\end{matrix}
\end{array}
\right)
\begin{array}{c}
\vphantom{\sum\limits_{k=0}^{m_i} }\\
\left.\vphantom{\begin{matrix}0\\\vdots\\0\end{matrix}}
\right\}\parbox{1.1cm}{\scriptsize $m_j-m_i$ rows}\\
\vphantom{\sum\limits_{k=0}^{m_i} }\\
\left.\vphantom{\begin{matrix}0\\\vdots\end{matrix}}
\right\}\parbox{1.1cm}{\scriptsize $m_j-m_i$ rows,}
\end{array}
\end{align}
for some coefficients $\{a_k,b_k\}$.
\item
If  $\lambda^2<0$ then 
\begin{align}\label{intersection algebra lambda negative off diag a}
\phantom{\begin{array}{c}\vdots\\\vdots\\\vdots\\\vdots\\\vdots\\a \end{array}}
B_{(i,j)}=
\left(
\begin{array}{c|c|c|c}
\bovermat{\parbox{2cm}{\linespread{.6}\selectfont\centering$m_j-m_i$ columns}}{\mbox{
$\begin{matrix}
0&\cdots&0\\
\vdots&&\vdots\\
\vdots&&\vdots\\
0&\cdots&0\end{matrix}$
}}
&
\begin{matrix}
\sum\limits_{k=0}^{m_i-1} a_k T_{m_i}^k
\\\hline
\sum\limits_{k=0}^{m_i-1} \left(\sum_{r=0}^kc_r\right) T_{m_i}^k
\end{matrix}
&
\bovermat{\parbox{2cm}{\linespread{.6}\selectfont\centering$m_j-m_i$ columns}}{\mbox{
$\begin{matrix}
0&\cdots&0\\
\vdots&&\vdots\\
\vdots&&\vdots\\
0&\cdots&0\end{matrix}$
}}
&
\begin{matrix}
\sum\limits_{k=0}^{m_i} b_k T_{m_i}^k
\\\hline
\sum\limits_{k=0}^{m_i} d_k T_{m_i}^k
\end{matrix}
\end{array}
\right),
\end{align}
and
\begin{align}\label{intersection algebra lambda negative off diag b}
B_{(j,i)}=
\epsilon_i\epsilon_j\left(
\begin{array}{c}
\begin{array}{c|c}
-\sum\limits_{k=0}^{m_i} a_k T_{m_i}^k
&
\sum\limits_{k=0}^{m_i} c_k T_{m_i}^k
\end{array}
\\\hline
\begin{matrix}
0&\cdots\cdots\cdots\cdots&0\\
\vdots&&\vdots\\
0&\cdots\cdots\cdots\cdots&0
\end{matrix}
\\\hline
\begin{array}{c|c}
\sum\limits_{k=0}^{m_i} \left(\sum_{r=0}^kb_r\right) T_{m_i}^k
&
-\sum\limits_{k=0}^{m_i} d_k T_{m_i}^k
\end{array}
\\\hline
\begin{matrix}
0&\cdots\cdots\cdots\cdots&0\\
\vdots&&\vdots
\end{matrix}
\end{array}
\right)
\begin{array}{c}
\vphantom{\sum\limits_{k=0}^{m_i} }\\
\left.\vphantom{\begin{matrix}0\\\vdots\\0\end{matrix}}
\right\}\parbox{1.1cm}{\scriptsize $m_j-m_i$ rows}\\
\vphantom{\sum\limits_{k=0}^{m_i} }\\
\left.\vphantom{\begin{matrix}0\\\vdots\end{matrix}}
\right\}\parbox{1.1cm}{\scriptsize $m_j-m_i$ rows,}
\end{array}
\end{align}
for some coefficients $\{a_k,b_k,c_k,d_k\}$. 
\end{enumerate}
\end{corollary}
\begin{proof}
Using the formula for $B_{(i,j)}$ given in \eqref{intersection algebra lambda real off diag}, \eqref{intersection algebra lambda nonreal off diag a}, and \eqref{intersection algebra lambda negative off diag a}, it is straightforward to check that \eqref{Bij block for nonequal eVals f} holds with $X=B_{(i,j)}$. Moreover, this formula for $B_{(i,j)}$ is the most general formula with this property because, by Lemma \ref{equal eValue off diag dimension}, it has the maximum number of parameters possible. Lastly, the formula for $B_{(j,i)}$ given in \eqref{intersection algebra lambda real off diag}, \eqref{intersection algebra lambda nonreal off diag b}, and \eqref{intersection algebra lambda negative off diag b} is obtained through another straightforward calculation by applying \eqref{Bij block for nonequal eVals e} directly to the formula for $B_{(i,j)}$.
\end{proof}

To simplify notation in the following lemma, for an integer $q$, we let $[q]_2$ denote the residue of $q$ modulo 2, that is, $[q]_2=0$ if $q$ is even and $[q]_2=1$ if $q$ is odd.
\begin{lemma}\label{Bij formula with zero lambda}
If $m_{i}\leq m_j$ and $\lambda_i=\lambda_j=0$ then
\begin{align}\label{intersection algebra lambda zero off diag 1}
\phantom{\begin{array}{c}\vdots\\\vdots\\\vdots\\\vdots\\\vdots\\\vdots \end{array}}
B_{(i,j)}=
\left(
\begin{array}{c|c}
\bovermat{\parbox{2cm}{\linespread{.6}\selectfont\centering$m_j-m_i$ columns}}{\mbox{
$\begin{matrix}
0&\cdots&0\\
\vdots&&\vdots\\
\vdots&&\vdots\\
\vdots&&\vdots\\
0&\cdots&0\end{matrix}$
}}
&
\begin{matrix}
c^{1}_1&c^{1}_2&\cdots&&\cdots&c^{1}_{m_i}\\
0&c^{0}_1&c^{0}_2&\cdots&\cdots&c^{0}_{m_i-1}\\
0&0&c^{1}_1&c^{1}_2&\cdots&c^{1}_{m_i-2}\\
\vdots&&\rdots{4}{-4pt}&c^{0}_1&\cdots&c^{0}_{m_i-3}\\
\vdots&&&\rdots{4}{-4pt}&\rdots{4}{-4pt}&\vdots\\
0&\cdots&&\cdots&0&c_1^{[m_i]_2}
\end{matrix}
\end{array}
\right),
\end{align}
and
\begin{align}\label{intersection algebra lambda zero off diag 2}
B_{(j,i)}=-\epsilon_i\epsilon_j\left(
\begin{array}{c}
\begin{matrix}
c^{[m_i+1]_2}_1&c^{[m_i+2]_2}_2&\cdots&&\cdots&c^{[2m_i]_2}_{m_i}\\
0&c^{[m_i+2]_2}_1&c^{[m_i+3]_2}_2&\cdots&\cdots&c^{[2m_i]_2}_{m_i-1}\\
0&0&c^{[m_i+3]_2}_1&c^{[m_i+4]_2}_2&\cdots&c^{[2m_i]_2}_{m_i-2}\\
\vdots&&\rdots{14}{-4pt}&c^{[m_i+4]_2}_1&\cdots&c^{[2m_i]_2}_{m_i-3}\\
\vdots&&&\rdots{14}{-4pt}&\rdots{14}{-4pt}&\vdots\\
0&\cdots&\cdots&\cdots&0&c_1^{[2m_i]_2}
\end{matrix}
\\\hline
\left.\begin{matrix}
\quad \,\,0& \cdots\,\quad\cdots\,\quad\cdots\,\quad\cdots\,\quad\cdots&0\\
\quad \,\,\vdots&&\vdots\\
\quad \,\,0& \cdots\,\quad\cdots\,\quad\cdots\,\quad\cdots\,\quad\cdots&0\end{matrix}\right\}\parbox{2cm}{\scriptsize $m_j-m_i$ rows}
\end{array}
\right)
\end{align}
for some coefficients $\{a_k,b_k,c_k^{1},c_k^{0}\}$.
\end{lemma}
\begin{proof}
Let us refer to the main diagonal of the upper right $m_i\times m_i$ block in each matrix $B_{(i,j)}$ and $B_{(j,i)}$ as that matrix's reference diagonal.

Notice that equations \eqref{Bij block for nonequal eVals a} and \eqref{Bij block for nonequal eVals b} hold in the present context with $\lambda_i=\lambda_j=0$ and $\epsilon_i=\epsilon_j$. Let us assume $\epsilon_i=\epsilon_j$, noting that for the other case, where $\epsilon_i\neq \epsilon_j$, we would first change the sign of the right side of \eqref{Bij block for nonequal eVals a}  and \eqref{Bij block for nonequal eVals b} and then proceed with exactly the same calculations.

Applying  \eqref{Bij block for nonequal eVals a}, we  find that the last row of $B_{(i,j)}$ contains only zeros below the reference diagonal, and, applying \eqref{Bij block for nonequal eVals b}, we find that the first column of $B_{(i,j)}$ contains only zeros to the left of the reference diagonal. Similarly, by \eqref{Bij block for nonequal eVals a} and  \eqref{Bij block for nonequal eVals b}, the first column and last row of   $B_{(j,i)}$ contain zeros in their entries that are below or to the left of the reference diagonal. After substituting 0 in for those entries, applying  \eqref{Bij block for nonequal eVals a} again, we now find that the second to last row of $B_{(i,j)}$ (or of $B_{(j,i)}$) contains only zeros below (or to the left of) the reference diagonal, whereas, by applying \eqref{Bij block for nonequal eVals b} again, we find that the second column of $B_{(i,j)}$ (or of $B_{(j,i)}$) contains only zeros to the left of (or below) the reference diagonal. Repeating this analysis, we eventually find that all entries in $B_{(i,j)}$ and $B_{(j,i)}$ that are below or to the left of the reference diagonal are zero.

Let us now calculate the restrictions that  \eqref{Bij block for nonequal eVals a} and  \eqref{Bij block for nonequal eVals b} impose on the remaining nonzero entries in $B_{(i,j)}$ and $B_{(j,i)}$. For the next observations, we use the term \emph{secondary transpose} to refer to the transformation of square matrices described by reflecting their entries over the secondary diagonal, that is, sending the $(i,j)$ entry of an $m\times m$ matrix to the $(m+1-j,m+1-i)$ entry. Applying \eqref{Bij block for nonequal eVals a}, we see that upper left $(m_i-1)\times (m_i-1)$ block of the upper right $m_i\times m_i$ block of $B_{(i,j)}$ is equal to $-1$ (or $-\epsilon_i\epsilon_j$ in the general case) times the secondary transpose of the upper left $(m_i-1)\times (m_i-1)$ block of $B_{(j,i)}$. Similarly, applying \eqref{Bij block for nonequal eVals b}, we see that lower right $(m_i-1)\times (m_i-1)$ block of the upper right $m_i\times m_i$ block of $B_{(i,j)}$ is equal to $-1$ (or $-\epsilon_i\epsilon_j$ in the general case) times the secondary transpose of the lower right $(m_i-1)\times (m_i-1)$ block of $B_{(j,i)}$. These last two observations, taken together, complete this proof.
\end{proof}

\begin{corollary}\label{Bii formula}
For all $i\in\{1,\ldots, \gamma\}$,
\begin{align}\label{intersection algebra lambda diag}
B_{(i,i)}=
\begin{cases}
\left(\sum_{k=1}^{\lceil m_i/2\rceil} a_kT_{m_i}^{m_i-2k+1}\right) I_{\mathrm{alt},m_i} &\mbox{ if $\lambda_i=0$}
\\
\left(
\begin{array}{cc}
0&\sum\limits_{k=0}^{m_i-1} a_kT_{m_i}^{k}\\
\sum\limits_{k=0}^{m_i-1} \left(\sum_{r=0}^ka_r\right)T_{m_i}^{k}&0
\end{array}
\right)
&\mbox{ if $\lambda_i^2<0$}
\\
0&\mbox{ otherwise,}
\end{cases}
\end{align}
where $I_{\mathrm{alt},m}$ denotes the $m\times m$ diagonal matrix with a 1 in its odd columns and a -1 in its even columns.
\end{corollary}
\begin{proof}
This follows immediately from the formulas in Corollary \ref{Bij formula with nonzero lambda} and Lemma \ref{Bij formula with zero lambda} with $i=j$.
\end{proof}

The previous results provide a general formula for matrices in $\mathscr{A}^o$. We now focus on obtaining a general formula of a subspace $\mathscr{A}^s$ satisfying \eqref{decomp of intersection algebra}. 
\begin{lemma}\label{dimension of scalling component}
Either $\dim(\mathscr{A})-\dim(\mathscr{A}^o)=1$ or $\dim(\mathscr{A})-\dim(\mathscr{A}^o)=2$, and the latter case occurs if and only if there exists a matrix $X$ in $\mathscr{A}$ satisfying
\begin{align}\label{intersection condition 0 simplified}
~&X AH_\ell ^{-1} +   AH_\ell ^{-1}X^T =  2AH_\ell ^{-1}\Leftrightarrow \left(X-I\right)^T H_\ell A^{-1}+H_\ell A^{-1} \left(X-I\right)=0,\\
~&X^TH_\ell \overline{A}+H_\ell \overline{A} X=0.
\end{align}
\end{lemma}
\begin{proof}
Define
\[
\mathscr{A}^o_1:=\left\{X\,\left|\, XAH_\ell^{-1}+AH_\ell^{-1}X^T=0\right.\right\}
\quad\mbox{ and }\quad
\mathscr{A}^o_2:=\left\{X\,\left|\, X^TH_\ell\overline{A}+H_\ell\overline{A}X=0\right.\right\}.
\]
Since $\mathscr{A}^o=\mathscr{A}^o_1\cap \mathscr{A}^o_2$,
\begin{align}\label{orthogonal intersection dimensions}
\dim(\mathscr{A}^o)+\dim(\mathscr{A}^o_1+ \mathscr{A}^o_2)=\dim(\mathscr{A}^o_1)+\dim( \mathscr{A}^o_2),
\end{align}
and, letting $\mathbb C I$ denote $\mathrm{span}\{I\}$, since $\mathscr{A}=\left(\mathscr{A}^o_1+\mathbb C I\right)\cap \left(\mathscr{A}^o_2+\mathbb C I\right)$,
\begin{align}\label{intersection algebra dimensions}
\dim(\mathscr{A})+\dim\left(\mathscr{A}^o_1+\mathscr{A}^o_2+\mathbb C I\right)
&=
\dim\left(\mathscr{A}^o_1+\mathbb C I\right)+\dim \left(\mathscr{A}^o_2+\mathbb C I\right)\\
&=\dim(\mathscr{A}^o_1)+\dim( \mathscr{A}^o_2)+2\\
&=\dim(\mathscr{A}^o)+\dim(\mathscr{A}^o_1+ \mathscr{A}^o_2)+2,
\end{align}
where this last equation holds by \eqref{orthogonal intersection dimensions}.
Therefore,
\[
\dim(\mathscr{A})-\dim(\mathscr{A}^o)=\dim(\mathscr{A}^o_1+ \mathscr{A}^o_2)-\dim\left(\mathscr{A}^o_1+\mathscr{A}^o_2+\mathbb C I\right)+2,
\]
and hence
\begin{align}\label{dimension of scaling component formula}
\dim(\mathscr{A})-\dim(\mathscr{A}^o)=
\begin{cases}
1 & \mbox{ if } I\not\in\mathscr{A}^o_1+ \mathscr{A}^o_2\\
2 & \mbox{ if } I\in\mathscr{A}^o_1+ \mathscr{A}^o_2.
\end{cases}
\end{align}
In particular, $\dim(\mathscr{A})-\dim(\mathscr{A}^o)=2$ if and only if there exists $X\in \mathscr{A}^o_2$ such that $(I-X)\in \mathscr{A}^o_1$, which is equivalent to \eqref{intersection condition 0 simplified}.
\end{proof}

\begin{lemma}\label{condition on C for 1-dimensional scaling component}
If $A=M_{m,\lambda}$ and $\lambda\neq 0$ then $\dim(\mathscr{A})-\dim(\mathscr{A}^o)=1$.
\end{lemma}
\begin{proof}
We assume that $(H_\ell, A)$ is in the canonical form of Theorem \ref{simultaneous canonical form theorem}, so $H_\ell=S_m$, where $S_m$ is defined in \eqref{Matrix representations first notations}. Fix a subspace $\mathscr{A}^s$ of $\mathscr{A}$ satisfying \eqref{decomp of intersection algebra}.
To produce a contradiction, let us assume that $\dim(\mathscr{A})-\dim(\mathscr{A}^o)\neq 1$. By Lemma \ref{dimension of scalling component}, we can assume that 
there exists a matrix $X$ in $\mathscr{A}^s$ satisfying \eqref{intersection condition 0 simplified}.
Since $H_\ell A^{-1}$ and $H_\ell \overline{A}$ are symmetric, condition \eqref{intersection condition 0 simplified} is fundamentally related to the two symmetric forms $Q_1$ and $Q_2$ defined by
\[
Q_1(v,w):=w^TH_\ell A^{-1} v
\quad\mbox{ and }\quad
Q_2(v,w):=w^TH_\ell \overline{A} v.
\]
Note that
\[
Q_2(v,w)=Q_1\left(A\overline{A}v,w\right)=Q_1\left(A^2v,w\right),
\]
where $\boldsymbol{A}$ is, again, the antilinear operator represented by $A$.

Let us now work instead with respect to a basis that is orthonormal with respect to $Q_1$, that is, letting $L$ denote the matrix representing the linear operator $A^2$ in this basis, we have
\[
Q_1(v,w)=w^Tv
\quad\mbox{ and }\quad 
Q_2=w^TLv
\]
in this new basis. By \cite[Chapter 11.3, Corollary 2]{gantmakher1953theory}, we can assume without loss of generality that
\begin{align}\label{normalizing Q2 first}
L=
\begin{cases}
\frac{1}{2}(I+iS_{m})J_{\lambda,m}(I-iS_{m}) &\mbox{ if }\lambda^2>1 \\
\frac{1}{2}(I+iS_{m})J_{\lambda,m}(I-iS_{m})\oplus \frac{1}{2}(I+iS_{m})J_{\lambda,m}(I-iS_{m})&\mbox{ otherwise. }
\end{cases}
\end{align}
%

The second equation in \eqref{intersection condition 0 simplified} implies that $X$ is in the Lie algebra of the transformation group that preserves $Q_2$, whereas the first equation of \eqref{intersection condition 0 simplified} implies that $X-I$ is in the Lie algebra of the transformation group that preserves $Q_1$. That is, with respect to the new basis, $(X-I)=-(X-I)^T$ and $X^TL+L X=0$,which is equivalent to
\begin{align}\label{intersection condition 0 simplified new}
(X-I)=-(X-I)^T
\quad\mbox{ and }\quad
[X,L]=0.
\end{align}
Defining the pair of matrices $(S,J)$ by
\begin{align}\label{change of basis for symmetric to jordan}
(S,J)=
\begin{cases}
\left(I+iS_{m},J_{\lambda,m}\right)&\mbox{ if }\lambda^2>1 \\
\left((I+iS_{m})\oplus(I+iS_{m}),J_{\lambda,m}\oplus J_{\lambda,m}\right) &\mbox{ otherwise, }
\end{cases}
\end{align}
the condition $[X,L]=0$ is equivalent to
\begin{align}\label{intersection condition 0 simplified new 2}
\left[S^{-1}XS,J\right]=0.
\end{align}
Solving for the matrix $X$ in $[X,L]=0$ is a classical problem of Frobenious whose general solution is given in \cite[Chapter 8]{gantmakher1953theory}. In \cite[Chapter 8]{gantmakher1953theory}, a formula is given for matrices that commute with a Jordan matrix such as $J$, so we have rewritten $[X,L]=0$ as in \eqref{intersection condition 0 simplified new 2}, in order to apply the solution of \cite[Chapter 8]{gantmakher1953theory} directly. The formula in \cite[Chapter 8]{gantmakher1953theory} gives that, after partitioning the matrix $S^{-1}XS$ into size $m\times m$ blocks, each block of $S^{-1}XS$ in this partition is an upper-triangular Toeplitz matrix.
If $X$ is a Toeplitz matrix then $(I+iS_{m}) X (I-iS_{m})$ is symmetric because $S_mX$ and $XS_m$ are both symmetric whereas $X^T=S_mXS_m$. Accordingly, letting $X^\prime$  denote the upper left $m\times m$ block of $X$, since $(I-iS_{m}) (X^\prime-I) (I+iS_{m})$ is Toeplitz,
\begin{align}\label{intersection condition 0 condtradiction eqn}
X^\prime-I &=\frac{1}{4}(I+iS_{m}) \big[(I-iS_{m}) (X^\prime-I) (I+iS_{m})\big] (I-iS_{m})\\
&=\left(\frac{1}{4}(I+iS_{m}) \big[(I-iS_{m}) (X^\prime-I) (I+iS_{m})\big] (I-iS_{m})\right)^T=(X^\prime-I)^T.
\end{align}
By \eqref{intersection condition 0 simplified new} and \eqref{intersection condition 0 condtradiction eqn}, $X^\prime=I$, which contradicts the upper left $m\times m$ block of the second matrix equation in \eqref{intersection condition 0 simplified}.
\end{proof}

With Lemmas \ref{dimension of scalling component} and \ref{condition on C for 1-dimensional scaling component} established we now give a general formula for a subspace $\mathscr{A}^s$ of $\mathscr{A}$ satisfying \eqref{decomp of intersection algebra}.
\begin{lemma}\label{condition on C for 2-dimensional scaling component}
For a subspace $\mathscr{A}^s$ of $\mathscr{A}$ satisfying \eqref{decomp of intersection algebra}, $\dim(\mathscr{A}^s)=2$ if and only if $A$ is nilpotent. In particular, if
\begin{align}\label{C mat with two scaling elements}
A=J_{0,m_1}\oplus\ldots\oplus J_{0,m_\gamma}
\end{align}
then, to satisfy \eqref{decomp of intersection algebra}, we can take the subspace $\mathscr{A}^s$ of $\mathscr{A}$ spanned by the identity matrix and the matrix 
\begin{align}\label{scaling element formula}
\bigoplus_{i=1}^{\gamma}D_{m_i}, 
\end{align}
where, for an integer $m$, $D_{m}$ denotes the $m\times m$ diagonal matrix defined by
\begin{equation}
\label{diagformula}    
D_{m}:=
\mathrm{Diag}\left(\frac{m}{2},\frac{m}{2}-1,\ldots, \frac{m}{2}-m+1\right).
\end{equation}
\end{lemma}
\begin{proof}
Suppose that $(H_\ell, A)$ is in the canonical form of Theorem \ref{simultaneous canonical form theorem}, specifically such that
\begin{align}\label{lambda_is are zero init}
A=J_{\lambda_1,m_1}\oplus\cdots\oplus J_{\lambda_\gamma,m_\gamma}
,
\end{align}
and suppose that $\dim(\mathscr{A}^s)=2$. As is shown in the proof of Lemma \ref{dimension of scalling component}, we can assume without loss of generality that
there exists a matrix $X$ in $\mathscr{A}^s$  satisfying \eqref{intersection condition 0 simplified}.
In particular, partitioning $X$ into a block matrix whose diagonal blocks $
X_{(i,i)}$ are size $m_{i}\times m_i$, the blocks $X_{(i,i)}$ satisfy
\begin{align}\label{intersection condition 0 simplified ii a}
X_{(i,i)} M_{m_i,\lambda_i}N_{m_i,\lambda_i} +   M_{m_i,\lambda_i}N_{m_i,\lambda_i}X_{(i,i)}^T =  2M_{m_i,\lambda_i}N_{m_i,\lambda_i}
\end{align}
and
\begin{align}\label{intersection condition 0 simplified ii b}
X_{(i,i)}^TN_{m_i,\lambda_i}\overline{M_{m_i,\lambda_i}}+N_{m_i,\lambda_i}\overline{M_{m_i,\lambda_i}} X_{(i,i)}=0.
\end{align}
Lemma \ref{condition on C for 1-dimensional scaling component} implies that \eqref{intersection condition 0 simplified ii a} and \eqref{intersection condition 0 simplified ii a} are consistent if and only if $\lambda_i=0$, and hence if $\mathscr{A}^s=2$ then $A$ is nilpotent.

Conversely, if $A$ is nilpotent then $\lambda_1=\cdots=\lambda_\gamma=0$. Hence, 
by \eqref{antilinear op canon form} and \eqref{N_def} the relations \eqref{intersection condition 0 simplified ii a} and \eqref{intersection condition 0 simplified ii b} can be rewritten as
\begin{equation}
\label{Diag_eq}
 X_{(i,i)} J_{0,m_i}S_{m_i} +   J_{0,m_i}S_{m_i}X_{(i,i)}^T =  2J_{0,m_i}S_{m_i}
\quad\mbox{ and }\quad
X_{(i,i)}^TS_{m_i}J_{0,m_i}+S_{m_i}J_{0,m_i}X_{(i,i)}=0
\end{equation}
for each $i$ individually. Assuming that $B_{(i,i)}=\mathrm{Diag}\left(x_1^i, \ldots x_{m_i}^i\right)$, by comparing the entries of \eqref{Diag_eq} with the help of the expressions for matrices  $J_{0, m_1}$ and $S_{m_i}$ from \eqref{Matrix representations first notations}, one gets that \eqref{Diag_eq} is equivalent to 
\begin{align}
\label{entrywise}
&x_j^i+x_{m_i-j}^i=2 \quad\quad \forall\, 1\leq j\leq m-1,\\
&x_j^i+x_{m-j+2}^i=0 \quad\quad \forall\, 2\leq j\leq m.
\end{align}
Finally, it is clear that taking $X_{(i,i)}=D_{m_i}$, where $D_{m_i}$ is as in \eqref{diagformula}, satisfies \eqref{entrywise} which completes the proof. 
\end{proof}
As a direct consequence of the previous lemmas, since for non-nilpotent $A$ we have $\mathscr A=\mathscr A^o+ \mathbb C I$, one gets immediately the following
\begin{corollary}
\label{nonnilp_cor}
If $A$ is not nilpotent then in \eqref{intersection algebra} one can take $\eta'=\eta$.
\end{corollary}

Now we prove one more result.
\begin{lemma}
\label{dimension_cor}
If $H_\ell$ and $A$ are in the canonical form prescribed by Theorem \ref{simultaneous canonical form theorem} and $A\neq 0$ then 
\begin{align}\label{max dim of intersection}
\dim(\mathscr{A})\leq n^2-4n+6.
\end{align}
Moreover, this bound is attained if and only if $(\ell,\boldsymbol{A})$  can be represented by the pair $(H_\ell, A)$ in the canonical form of Theorem \ref{simultaneous canonical form theorem} with
\begin{align}\label{cMat with max symmetries}
A=J_{0,2}\oplus\overbrace{J_{0,1}\oplus\cdots\oplus J_{0,1}}^{n-3\mbox{\small $\,$ copies}}.
\end{align}
\end{lemma}
\begin{proof}

Assume that 
\begin{align}\label{max dim of intersection assumption}
\dim(\mathscr{A})\geq n^2-4n+6,
\end{align}
and that $(H_\ell, A)$ are in the canonical form of Theorem \ref{simultaneous canonical form theorem}. We will still use the notation of \eqref{simultaneous canonical form theorem eqn}, in particular referring to the sequence $(\lambda_1,\ldots, \lambda_\gamma)$. 

Suppose that the $\lambda_i$s are not all the same. Without loss of generality, we can assume that $(\lambda_1,\ldots, \lambda_\gamma)$ is enumerated so that there exists an integer $k$ such that 
\begin{align}\label{number of equal eValues}
\lambda_1=\ldots=\lambda_k
\quad\mbox{ and }\quad
\lambda_j\neq \lambda_1
\quad\quad\forall j>k.
\end{align}

Define
\[
s=\sum_{i=1}^k[\mbox{number of rows in }M_{\lambda_i,m_i}]
\]
where $k$ is as in \eqref{number of equal eValues}. By Lemma \ref{lemma for reduction to eSpaces}, for every matrix $B$ in $\dim(\mathscr{A}^o+\mathrm{span}\{I\})$, the upper right $(s)\times(n-1-s)$ block and the lower left $(n-1-s)\times(s)$ block of $B$ is zero. Moreover, since the $\lambda_i$s are not all zero, there is at least one index $i$ such that $B_{(i,i)}$ has zeros on its main diagonal. Accordingly, if the $\lambda_i$s are not all the same, then
\[
\dim(\mathscr{A}^o)+1=\dim(\mathscr{A}^o+\mathrm{span}\{I\}) \leq (n-1)^2-2s(n-1-s).
\]
Since
\[
2n-4\leq 2j(n-1-j)
\quad\quad\forall\, 1\leq j<n-1,
\]
it follows that
\[
\dim(\mathscr{A})=\dim(\mathscr{A}^o)+1\leq (n-1)^2-2s(n-1-s)\leq (n-1)^2-2n+4=n^2-4n+5,
\]
where the identity $\dim(\mathscr{A})=\dim(\mathscr{A}^o)+1$ follows from Lemma \ref{condition on C for 2-dimensional scaling component} and the assumption that the $\lambda_i$s are not all the same. Clearly, this contradicts \eqref{max dim of intersection assumption}, so if \eqref{max dim of intersection assumption} holds then there exists a value $\lambda\in \mathbb C$ such that 
\begin{align}\label{lambda_is are the same}
\lambda=\lambda_i
\quad\quad\forall\,i.
\end{align}

If \eqref{lambda_is are the same} holds with $\lambda\neq 0$ then Corollaries \ref{Bij formula with nonzero lambda} and \ref{Bii formula} imply that each matrix $B$ in $\mathscr{A}^o$ is fully determined by its entries above the main diagonal, and hence, applying Lemma \ref{condition on C for 2-dimensional scaling component}, 
\begin{align}\label{lambda_is are the same nonzer dim A bound}
\dim(\mathscr{A}) \leq\cfrac{(n-1)(n-2)}{2}+1<n^2-4n+6, \quad \forall n\geq 2
\end{align}
Therefore, if \eqref{lambda_is are the same} holds with $\lambda\neq 0$ then our assumption \eqref{max dim of intersection assumption} fails. 

In other words, -- assuming for a moment that \eqref{max dim of intersection assumption} can be satisfied, which we will prove below by giving an explicit example -- if $\dim(\mathscr{A})$ is maximized then we can assume without loss of generality that 
\begin{align}\label{lambda_is are zero}
A=J_{0,m_1}\oplus\cdots\oplus J_{0,m_\gamma}
\quad\mbox{ with }\quad
m_1\geq\cdots \geq m_{\gamma}.
\end{align}
For $B$ in $\mathscr{A}^o$, let us partition $B$ as is done in Lemma \ref{Bij formula with zero lambda}. By Lemma \ref{Bij formula with zero lambda}, for $i<j$ the $B_{(i,j)}$ and $B_{(j,i)}$ blocks are together determined by $2m_{j}$ parameters, whereas, by Corollary \ref{Bii formula}, the $B_{(i,i)}$ block is determined by $\lceil \tfrac{m_i}{2}\rceil$ parameters, where $\lceil \tfrac{m_i}{2}\rceil$ denotes the ceiling function, i.e. the smallest integer not less than $\tfrac{m_i}{2}$. Hence, by counting the number of parameters determining $B$, Lemma \ref{Bij formula with zero lambda} and Corollary \ref{Bii formula} imply that if \eqref{lambda_is are zero} holds then
\begin{align}\label{lambda_is are zero dimension count}
\dim(\mathscr{A}^o)=\sum_{k=1}^\gamma\left( \left\lceil \frac{m_k}{2}\right\rceil+2(k-1)m_k\right).
\end{align}

Let $r\in\{1,\ldots, \gamma\}$ be an integer such that 
\[
m_i=1\quad\quad\forall\, i>  r,
\]
and to compare with $A$, let us also consider the matrix 
\begin{align}\label{lambda_is are zero alt}
A^\prime=J_{0,m_{1}}\oplus \cdots\oplus J_{0,m_{r-1}}\oplus J_{0,1}\oplus\cdots\oplus J_{0,1}.
\end{align}
In other words, $A^\prime$ is obtained from $A$ by replacing the last nonzero block on the diagonal of $A$ with zeros.
We will compute the dimension of $\mathscr{A}^o$ corresponding to the case where $A=A^\prime$, but, since are going to compare this to the sum in \eqref{lambda_is are zero dimension count}, for clarity let $\mathscr{A}^\prime$ denote the algebra that we would otherwise denote by $\mathscr{A}^o$ corresponding to this case where $A=A^\prime$, and let $\mathscr{A}^o$ still denote the algebra refered to in \eqref{lambda_is are zero dimension count}.

Notice that the $k$th summand in \eqref{lambda_is are zero dimension count} counts the number of parameters determining the blocks $B_{(i,j)}$ of a matrix $B$ in $\mathscr{A}^o$ for which $\max\{i,j\}=k$. If we compare the general formula for a matrix $B$ in $\mathscr{A}^o$ to that of a matrix $B^\prime$ in $\mathscr{A}^\prime$, the only difference appears in the blocks $B_{(i,j)}$ of $B$ for which $\max\{i,j\}=r$, and hence a formula for $\dim(\mathscr{A}^\prime)$ should match the formula in \eqref{lambda_is are zero dimension count}, except that the $r$th summand will change. Using Lemma \ref{Bij formula with zero lambda} and Corollary \ref{Bii formula}, it is however straightforward to work out exactly how this $r$th summand of \eqref{lambda_is are zero dimension count}.

Specifically, in replacing the formula for $B$ with the formula for $B^\prime$, the $B_{(r,r)}$ block is replaced with the $m_r\times m_r$ matrix having $m_r^2$ independent parameters, whereas, for all $i<r$, $B_{(i,r)}$ (respectively $B_{(r,i)}$) is replaced with a matrix having $m_r$ independent parameters in its first row (respectively column) and zeros elsewhere. Accordingly,
\begin{align}\label{improved bound}
\dim(\mathscr{A}^\prime)&=\dim(\mathscr{A}^o)-\left(\left\lceil \frac{m_r}{2}\right\rceil+2(r-1)m_r\right)+m_r^2+2(r-1)m_r
\geq\dim(\mathscr{A}^o).
\end{align}
Since equality holds in \eqref{improved bound} if and only if $m_r=1$, the dimension of $\mathscr{A}^o$ is maximized with $A$ as in \eqref{lambda_is are zero} if and only if
\begin{align}\label{cMat with almost max symmetries}
A=J_{0,m_1}\oplus\overbrace{J_{0,1}\oplus\cdots\oplus J_{0,1}}^{n-1-m_1\mbox{\small $\,$ copies}},
\end{align}
in which case, by \eqref{lambda_is are zero dimension count}, 
\begin{align}\label{almost max dimension count}
\dim{\mathscr{A}^o}=\left\lceil \frac{m_1}{2}\right\rceil+\sum_{k=2}^{n-m_1}(2k-1)= \left\lceil \frac{m_1}{2}\right\rceil+ (n-m_1)^2-1.
\end{align}

Since $A\neq 0$, this last sum is maximized with $A$ as in \eqref{cMat with almost max symmetries} if and only if $A$ is as in \eqref{cMat with max symmetries}, in which case applying \eqref{almost max dimension count} with $m_1=2$ yields \eqref{max dim of intersection} because, by Lemma \ref{condition on C for 2-dimensional scaling component}, if $A$ is as in \eqref{cMat with almost max symmetries} then $\dim{\mathscr{A}}=\dim{\mathscr{A}^o}+2$.
\end{proof}







\end{document}